\documentclass[11pt]{amsart}

\usepackage{amsthm,amsmath,amssymb,amsfonts}
\usepackage[colorlinks=true,hypertexnames=false,linkcolor=blue,citecolor=blue]{hyperref}

\usepackage{epsfig,graphics}
\usepackage{amscd}
\usepackage{amsmath}
\usepackage{bm}
\usepackage{amssymb}
\usepackage{graphicx}
\usepackage{verbatim}
\usepackage{enumerate}
\usepackage{pgf,tikz}
\usetikzlibrary{decorations.pathreplacing, shapes.multipart, arrows, matrix, shapes}
\usetikzlibrary{patterns}
\usepackage{caption}
 
\usepackage[refpage]{nomencl}
\makenomenclature

\newtheorem{theorem}{Theorem}[section]

\newtheorem{lema}[theorem]{Lemma}
\newtheorem{claim}[theorem]{Claim}
\newtheorem{prop}[theorem]{Proposition}
\newtheorem{prob}[theorem]{Problem}
\newtheorem{coro}[theorem]{Corollary}

\theoremstyle{definition}
\newtheorem{definition}[theorem]{Definition}

\newtheorem{conj}[theorem]{Conjecture}

\theoremstyle{remark}
\newtheorem{remark}[theorem]{Remark}

\numberwithin{equation}{section}

\theoremstyle{plain}
\newtheorem{maintheorem}{Theorem}

\newcommand{\R}{\ensuremath{\mathbb{R}}}
\newcommand{\Z}{\ensuremath{\mathbb{Z}}}

\newcommand{\nt}{\ensuremath{\mathbb{N}}}
\newcommand{\eqdef}{\stackrel{\scriptscriptstyle\rm def}{=}}
\newcommand{\cA}{\ensuremath{\mathcal{A}}}

\DeclareMathOperator{\sen}{sin}

\DeclareMathOperator{\supp}{supp}
\DeclareMathOperator{\dif}{Diff}
\DeclareMathOperator{\sur}{S^2}

\DeclareMathOperator{\difp}{Diff_{+}^{\infty}(\R^2)}

\newcommand{\eps}{\varepsilon}

\newcommand{\tcap}{\pitchfork}

\newcommand{\Id}{\operatorname{Id}}

\newcommand{\interior}{\operatorname{int}}

\newcommand{\fix}{\operatorname{Fix}}

\newcommand{\Diff}{\operatorname{Diff}}
\newcommand{\gr}{\operatorname{graph}}

\newcommand{\cB}{\mathcal{B}}
\newcommand{\cU}{\mathcal{U}}
\newcommand{\cS}{\mathcal{S}}
\newcommand{\cP}{\mathcal{P}}
\newcommand{\cM}{\mathcal{M}}
\newcommand{\Ba}{\mathbf{B}}
\newcommand{\rB}{B}

\begin{document}

\title{Dirac physical measures on saddle-type fixed points}

\author{Pablo Guarino}
\address[Pablo Guarino]{Instituto de Matem\'atica e Estat\'istica, Universidade Federal Fluminense}
\curraddr{Rua Professor Marcos Waldemar de Freitas Reis, s/n, Bloco H - Campus do Gragoat\'a - Niter\'oi - RJ - Brazil CEP 24.210-201}
\email{pablo\_\,guarino@id.uff.br}

\author{Pierre-Antoine Guih{\'e}neuf}
\address[Pierre-Antoine Guih{\'e}neuf]{
Sorbonne Université, Université de Paris, CNRS, Institut de Mathématiques de Jussieu-Paris Rive Gauche, IMJ-PRG, F-75005 Paris, France}
\email{pierre-antoine.guiheneuf@imj-prg.fr}

\author{Bruno Santiago}
\address[Bruno Santiago]{Instituto de Matem\'atica e Estat\'istica, Universidade Federal Fluminense}
\curraddr{Rua Professor Marcos Waldemar de Freitas Reis, s/n, Bloco H - Campus do Gragoat\'a - Niter\'oi - RJ - Brazil CEP 24.210-201}
\email{brunosantiago@id.uff.br}

\thanks{We warmly thank Sylvain Crovisier for his ideas concerning Theorem \ref{mt.generic}, Christian Bonatti for the attention he paid to this work and Davi Obata for his reading of a first draft. We express our very great appreciation to the referee for very keen remarks leading to a considerable improvement of our exposition.\\
P.G. and P.A.G were partially supported by Coordena\c{c}\~ao de Aperfei\c{c}oamento de Pessoal de N\'ivel Superior -- Brasil (CAPES) grant 23038.009189/2013-05. P.A.G. was partially funded by the France-Brazil network (RFBM). P.A.G and B.S. were partially supported by a PEPS (CNRS) grant. B.S. was partially financed by the Coordena\c{c}\~ao de Aperfei\c{c}oamento de Pessoal de N\'ivel Superior - Brasil (CAPES) - Finance Code 001 and also acknowledges the support of Fondation Louis D -- Institut de France (project coordinated by M. Viana)}



\begin{abstract} In this article we study some statistical aspects of surface diffeomorphisms. We first show that for a $C^1$ generic diffeomorphism, a Dirac invariant measure whose \emph{statistical basin of attraction} is dense in some open set and has positive Lebesgue measure, must be supported in the orbit of a sink. We then construct an example of a $C^1$-diffeomorphism having a Dirac invariant measure, supported on a saddle-type hyperbolic fixed point, whose statistical basin of attraction is a nowhere dense set with positive Lebesgue measure. Our technique can be applied also to construct a $C^1$ diffeomorphism whose set of points with historic behaviour has positive measure and is nowhere dense.
\end{abstract}

\maketitle
\setcounter{tocdepth}{1}

\tableofcontents

\section{Introduction}

A general issue in Ergodic Theory is to describe the space of all probability measures which are invariant under some given dynamical system. There is a large variety of different types of such spaces, ranging from being a singleton (as for irrational rotations) to infinite dimensional simplices with dense extreme points (as for transitive Anosov systems).

This variety of behaviours can also be detected in a pointwise fashion. For instance let $f:M^d \to M^d$ be an arbitrary (say continuous) map on a smooth compact manifold $M^d$ of dimension $d\geq 1$. Given $x\in M$, consider the \emph{empirical probability measure} of $x$ at time $k$: $\mu_k(x)\eqdef\frac{1}{k}\sum_{l=0}^{k-1}\delta_{f^l(x)}$\,, where $\delta_y$ is the Dirac mass at the point $y$. An \emph{asymptotic measure} of $x$ is an accumulation point of the sequence $\mu_k(x)$, in the weak-* topology. By elementary arguments, every asymptotic measure is invariant. We denote by $\cM(x)$ the set of asymptotic measures of $x$. Again, this set can be rather complicated sometimes, as in the case of transitive Anosov diffeomorphisms for which one finds points $x\in M$ such that $\cM(x)$ equals the whole space of invariant measures\footnote{This is a non-trivial consequence of the \emph{specification property}, see \cite[Theorem 4]{sigmund1974dynamical}. The fact that specification holds for transitive Anosov diffeomorphisms was proved by Bowen \cite{bowen1971periodic}.}. In particular, it is not possible to detect clear statistics from the orbit of $x$. The opposite situation occurs when $\cM(x)$ is a singleton, and so the statistics of the orbit is captured by a single measure. Conversely, given an invariant probability measure $\mu$, one can look for the set of points $x$ whose statistics is captured by $\mu$. This leads to the notion of the \emph{statistical basin of attraction} of a measure $\mu$: 
$$\cB_f(\mu)=\big\{x\in M;\,\cM(x)=\{\mu\}\big\}.$$
In other words, the statistical basin of $\mu$ is the set of points $x \in M$ such that$$\lim_{n\to+\infty}\left(\frac{1}{n}\sum_{j=0}^{n-1}\phi\big(f^j(x)\big)\right)=\int_{M}\phi\,d\mu$$ for every continuous function $\phi:M\to\R$. An important and difficult question arises: given a dynamical system, pick a ``random" initial condition $x\in M$. Does $x$ has its statistics well described by some measure? How many such measures do there exist? This question is a part of the well known \emph{Palis conjecture}: for any $f$ in a $C^r$ dense set of diffeomorphisms, one has a finite number of physical measures, whose statistical basins cover a full Lebesgue measure set on $M$ \cite{palis2000global}. Recall that $\mu$ is a \emph{physical measure} for $f$ if $m\big(\cB_f(\mu)\big)>0$, where $m$ denotes the Lebesgue measure on the compact manifold $M$. Thus, a major problem in smooth Ergodic Theory is whether any given dynamical system supports a physical measure, and what properties this measure possesses.

By Birkhoff's Ergodic Theorem, an invariant ergodic probability measure which is absolutely continuous with respect to the Lebesgue measure is automatically a physical measure. As it is well-known, such measures always exist for $C^{1+\alpha}$ uniformly expanding maps on compact manifolds \cite[Chapter III.1]{M}. Furthermore, in the early eighties, Jakobson proved that \emph{many} one-dimensional maps, presenting a critical point, also preserve an absolutely continuous ergodic probability measure \cite{jakobson} (many papers have addressed the problem of the existence of absolutely continuous invariant probability measures, most notably for one-dimensional dynamics with critical points. See for instance \cite{ALdM, BC, bowen, large, CE, CGP, ledra, lyubich, WY2006} and references therein).

In higher dimensions, ergodic invariant measures whose Lyapunov exponents are non-zero and which are absolutely continuous with respect to volume along unstable manifolds are special types of physical measures, called \emph{SRB measures} after Sinai-Ruelle-Bowen. For $C^2$ uniformly hyperbolic systems, SRB measures are the sole physical measures. A survey on this subject may be found in \cite{Young}.

On the other hand, there exist physical measures without any geometrical structure, e.g. Dirac measures on fixed points. Of course the most trivial example would be the Dirac measure supported on a topologically attracting fixed point (whenever it exists). Furthermore, Dirac physical measures may be supported on an indifferent fixed point (for instance, for the well known Manneville-Pomeau map, see \cite{young.israel}), or even on a hyperbolic repelling fixed point (for instance, for some quadratic polynomials leaving invariant the unit interval, see \cite{hk}).

More examples of Dirac physical measures may be obtained from deformations of Anosov diffeomorphism on the two-dimensional torus, with indifferent unstable direction at the fixed point \cite{hu.young}, or from transitive flows on surfaces \cite{vargas}.

Although examples as in \cite{hu.young} have positive topological entropy and are topologically mixing, the fixed point where the Dirac physical measure is supported is not hyperbolic, due to the indifferent unstable direction. Also, a physical measure on a saddle type hyperbolic fixed point is easily built for some systems with zero topological entropy, such as the \emph{figure eight attractor} (see \cite{Young}). 

Further examples of a Dirac physical measure supported on a saddle-type hyperbolic fixed point, whose statistical basin contains wandering domains, are built in \cite{CV} and \cite{japas}, inside Newhouse domains.

In light of all the examples mentioned above, we pose the following questions:

\begin{prob}\label{questionprob}
What dynamical configuration/mechanisms are responsible for the existence of a Dirac physical measure supported on a saddle-type hyperbolic periodic orbit? What is the relation between such a physical measure and the presence of homoclinic tangencies associated with the given periodic orbit?
\end{prob}   

One should not expect a simple answer since, for instance, there exist Cherry flows presenting Dirac physical measures of saddle type and no tangency \cite{saghin2013invariant}. Notice, however, that one can create a tangency by small perturbations.

Nonetheless, all above mentioned examples suggest that a Dirac physical measure supported on a saddle-type hyperbolic fixed point is a highly non-generic phenomena. Therefore, as a \textit{testing} conjecture for Problem~\ref{questionprob}, we propose the following: 

\begin{conj}\label{conj.generic}
The set of $f\in\dif^1(M)$ having a Dirac physical measure supported on a saddle-type hyperbolic fixed (or periodic) point, is meagre (\emph{i.e.} is a countable union of closed sets with empty interior).   
\end{conj}

In \cite{santiago2018dirac}, a special case of this conjecture has been proved, by assuming that the basin of the physical measure is dense in $M$. We were not able to fully prove Conjecture \ref{conj.generic}, but we were able to prove the following, which is our first main result.

\begin{maintheorem}
	\label{mt.generic}
	For any closed manifold $M^d$, of dimension $d\geq 2$, there exists a dense $G_{\delta}$ (residual) subset $\mathcal{R}$ of $\dif^1(M)$ such that for every $f\in\mathcal{R}$, if $\sigma$ is a fixed point\footnote{The same statement holds for periodic points, with a similar proof.} of $f$ such that $\delta_{\sigma}$ is a physical measure whose statistical basin $\cB_f(\delta_{\sigma})$ is dense in some open set, then $\sigma$ is a sink. In particular, if $\sigma$ is a saddle such that 
$\delta_{\sigma}$ is a physical measure, then the basin $\cB_f(\delta_{\sigma})$ must be a nowhere dense set. 
\end{maintheorem}  
  
The main reason why our proof does not solve Conjecture~\ref{conj.generic} is because it is based on an entropy estimation coming from \cite{catsigeras2015pesin}, which demands the saddle point to be inside some non-trivial Lyapunov stable set (see Section~\ref{Sec2}). To obtain this we need some denseness assumption. Nevertheless, we are not aware of any other technique that can be used to prove non-existence of Dirac physical measures on saddle-type hyperbolic fixed points for any given class of systems. For instance, even for partially or uniformly hyperbolic systems, the entropy estimation in \cite{catsigeras2015pesin} is the only tool we know (see Section \ref{subseceleonora} for details).

Moreover, to the best of our knowledge, in all known examples of Dirac physical measures, the basin of attraction either contains wandering domains (and so it has non-empty interior) or has full Lebesgue measure. Thus, one may ask if there exists at least one $f\in\dif^1(M)$ having a fixed point $\sigma$ such that $\cB_{\delta_{\sigma}}$ is a nowhere dense set of positive Lebesgue measure (and thus, $\delta_{\sigma}$ is a physical measure). We prove in this paper that the answer is yes. Our second main result is the following.

\begin{maintheorem}\label{main.exonsurfaces} Let\, $\sur$ be a closed  surface. There exists $f\in\dif^1(\sur)$ having a saddle-type hyperbolic fixed point $p$ whose statistical basin of attraction $\cB_f(\delta_{p})$ is a nowhere dense set (in particular, it has empty interior) with positive Lebesgue measure in $\sur$.
\end{maintheorem} 

The statement of Theorem~\ref{main.exonsurfaces} can be reduced to a local construction on the plane. With this purpose, recall that a diffeomorphism $f\in\Diff^1(\R^2)$ of the plane is said to be \emph{compactly supported} if there exists a ball $B(O,R)$, centred at the origin $O=(0,0)$, such that $f|_{\R^2\setminus B(O,R)}=\operatorname{Id}$. Given a closed surface $\sur$ and a compactly supported diffeomorphism $f\in\Diff^1(\R^2)$, we can choose a local chart $U$ and embed $f$ as a diffeomorphism of $\sur$, which is going to be equal to $f$ in the chart $U$, sending $U$ to itself, and the identity outside $U$. Therefore, the result below implies Theorem~\ref{main.exonsurfaces}. 

\begin{maintheorem}\label{main.exemplonovo} There exists a compactly supported diffeomorphism $f\in\dif^1(\R^2)$ having a saddle-type hyperbolic fixed point at the origin $O$, whose statistical basin of attraction $\cB_{f}(\delta_O)$ is a nowhere dense set of positive two-dimensional Lebesgue measure.	
\end{maintheorem} 

Although it only uses elementary tools from real analysis, the proof of Theorem \ref{main.exemplonovo} contains the majority of the technical part of this paper. We begin with an explicitly devised \emph{figure eight attractor} (see \S \ref{sec.figoito}), where we can ensure that the statistical basin of the saddle fixed point contains wandering domains (in particular, it contains open sets). This first part of the construction has already been done in much more generality, see \cite{CV} and \cite{japas} and the references therein. The main task and novelty in this paper is to remove a big set of points from the basin, in order to obtain a nowhere dense set with positive Lebesgue measure. This is done by an \textit{orbit exclusion procedure}, which consists of two different parts. In the first and most difficult one, we create a trapping region by pushing points away from the stable manifold of the saddle fixed point (see \S \ref{sec.orbitexcum}). After this deformation (which is \emph{huge} in the $C^1$ topology), we are able to prove that the statistical basin of the saddle fixed point consists of a specific tower of wandering domains up to a nowhere dense zero Lebesgue measure set. An interesting feature is that by the very form of those deformations, we end up creating positive topological entropy as well as infinitely many periodic points (these are intrinsic properties of our construction, see Proposition \ref{PropCoding}). In the second part of the orbit exclusion procedure, we remove from the statistical basin the complement of a nowhere dense set with positive measure (see \S \ref{sec.orbitexcii}). This is achieved by composing infinitely many arbitrarily small pushes with disjoint supports that accumulate in a \emph{flat tangency interval}. Unfortunately, for this perturbation to be able to really remove points from the basin, the resulting map is $C^1$ but not $C^2$ (see Lemma \ref{l.cum}). Finally, we notice that, with the same proof, it is possible to get a similar result for points with historic behaviour instead of points in $\cB(\delta_{p})$ (see Proposition \ref{FORA2!}).

\subsection{Organization of the paper}
This paper is organized as follows: in Section~\ref{Sec2} we first introduce some notation and give the basic definitions we shall use, and then we present the proof of Theorem~\ref{mt.generic}. The remaining sections are devoted to the construction of the example in Theorem \ref{main.exemplonovo}. In Section~\ref{sec.figoito} we construct a specific \emph{figure eight attractor}, with suitable affine returns. In Section~\ref{sec.orbitexcum} we perform the first part of the orbit exclusion procedure. This procedure gives rise to a set of \emph{persistent points}, that will be carefully described in Section \ref{SecGeom}. In Section~\ref{sec.orbitexcii} we complete the orbit exclusion procedure, and we finally prove Theorem~\ref{main.exemplonovo}. 

\section{Generic diffeomorphisms: proof of Theorem~\ref{mt.generic}}\label{Sec2}

The proof of Theorem~\ref{mt.generic} is a modification of the argument given by the third author in \cite{santiago2018dirac}. Here we only have the denseness assumption of \cite{santiago2018dirac} in a small part of $M$ and we manage to obtain the same conclusions using Gourmelon's version of Franks Lemma \cite{gourmelon2016franks}. The argument is by contradiction, and the idea behind it is to show that there exists some point in the manifold whose $\omega$-limit set is Lyapunov stable and contains the support of the physical measure. Being Lyapunov stable, it admits a dominated splitting and the result of \cite{catsigeras2015pesin} allows us to perform an entropy estimation leading to a contradiction. 

\subsection{Notations and definitions}\label{notations}
Let $M$ be a closed manifold of dimension $d\ge 2$. We denote by $\dif^1(M)$ the space of $C^1$ diffeomorphisms over $M$, endowed with the $C^1$ topology. Given $f\in\dif^1(M)$ and $x\in M$, the orbit of $x$ is the set $O(x)=\{f^n(x);\,n\in\Z\}$. We denote by $\fix(f)$ the set of fixed points of $f$. Recall that a periodic point is an element $p\in\fix(f^n)$, for some integer $n>0$. The smallest of such $n$ is called the period of $p$, and is denoted by $\pi(p)$. Finally, we denote by $m$ the normalised Lebesgue measure of $M$.

\begin{remark}\label{rem.suportes}
It will be convenient for us to work with two different notions of the support of a function. For a \emph{real valued function} $\varphi:\R^2\to\R$, its support is the set $\supp(\varphi)=\overline{\{x\in\R^2;\,\varphi(x)\neq 0\}}$. For a \emph{diffeomorphism} $h:\R^2\to\R^2$, its support is the set 
$\supp(h)=\overline{\{x\in\R^2;\,h(x)\neq x\}}$. Of course, which notion we are going to use will be clear from the context.     
\end{remark}

\subsubsection{The weak star topology}
$\cP(M)$ denotes the set of probability measures on $M$, endowed with the weak-star topology. 


\subsubsection{Lyapunov exponents} For $x\in M$ and $v\in T_xM\setminus\{0\}$, the \emph{Lyapunov exponent} of $f$ at $x$ in the direction of $v$ is 
$$\lambda(x,v):=\lim_{n\to\infty}\frac{1}{n}\log\|Df^n(x)v\|,$$whenever the limit exists. By Oseledets' theorem, given $\mu\in\cP_f(M)$ there exists a full measure set, the set of \emph{regular points}, and measurable functions
$\chi_1\leq\dots \leq\chi_d:M\to\R$, such that given a regular point $x\in M$, for every $v\in T_xM\setminus\{0\}$ there exists $i$ such that 
$$\chi_i(x)=\lambda(x,v).$$
In the particular case $\mu=\delta_{\sigma}$, for some $\sigma\in\fix(f)$, the Lyapunov exponents are the logarithm of the modulus of the eigenvalues of $Df(\sigma)$. For details, see \cite{M}.

\subsubsection{Homoclinic classes and dominated splitting} Given a compact set $\Lambda\subset M$, invariant under $f\in\dif^1(M)$, we say that $\Lambda$ admits a \emph{dominated splitting} if there exists a decomposition of the tangent bundle $T_{\Lambda}M=E\oplus F$, which is invariant under the derivative $Df$, and numbers $C>0$, $0<\lambda<1$ such that for every $x\in\Lambda$ and $n>0$ one has 
$$\big\|Df^n(x)|_E\big\|\,\big\|Df^{-n}(f^n(x))|_F\big\|\leq C\lambda^n.$$

In particular, the orbit of any saddle-type hyperbolic periodic point $p$ (\emph{i.e.}, no eigenvalue of $Df^{\pi(p)}(p)$ have modulus equal to $1$, and the spectrum of $Df^{\pi(p)}(p)$ intersects both components of $\mathbb{C}\setminus\mathbb{S}^1$) admits a dominated splitting $E^s\oplus E^u$. The stable manifold theorem \cite{palis} assures the existence of submanifolds $W^s(p)$ and $W^u(p)$, which are tangent to $E^s$ and $E^u$, respectively, at $p$. We denote 
$W^a(O(p))=\bigcup_{l=0}^{\pi(p)-1}W^a(f^l(p))$, $a=s,u$.

Given two hyperbolic periodic points $p$ and $q$ we say that they are \emph{homoclinically related} if $W^s(O(p))\tcap W^u(O(q))\neq\emptyset$ and
$W^u(O(p))\tcap W^s(O(q))\neq\emptyset$. The \emph{homoclinic class} of a periodic point $p$, denoted by $H(p)$, is the closure of the set of periodic points $q$, homoclinically related with $p$. Every homoclinic class $H(p)$ is a transitive invariant set, \emph{i.e.}, contains a point whose orbit is dense in $H(p)$, see Proposition 3.2 of \cite{newhouse1980lectures}. If $p$ is a saddle-type hyperbolic periodic point and $\big|\det{Df^{\pi(p)}}(p)\big|<1$, we say that $p$ is \emph{dissipative}.

\subsubsection{Lyapunov stable sets} A compact set $\Lambda\subset M$, invariant under $f\in\dif^1(M)$, is said to be \emph{Lyapunov stable} if for every neighbourhood $U$ of $\Lambda$ it is possible to
find a neighbourhood $V$ of $\Lambda$ such that if $x\in V\cap U$ then $f^n(x)\in U$, for every $n \geq 1$. 

\subsection{Tools for the proof} Let us begin by summarizing the results we shall invoke in our proof. 

\subsubsection{Lyapunov stable sets and unstable manifolds} We begin with an easy lemma linking Lyapunov stable sets and unstable manifolds of fixed points.

\begin{lema}\label{LemLyapUnst}
Let $A$ be a compact invariant Lyapunov stable set for $f\in\dif^1(M)$, and $\sigma\in \fix(f)\cap A$. Then $W^u(\sigma)\subset A$. 
\end{lema}

\begin{proof}
Suppose that there exists $x\in W^u(\sigma)\setminus A$. By compactness, one can find a neighbourhood $V$ of $A$ such that $x\notin V$. As $A$ is Lyapunov stable, there exists a neighbourhood $U$ of $A$ such that any $y\in U$ satisfies $f^n(y)\in V$ for any $n\ge 0$.

But $x\in W^u(\sigma)$, so there exists $m\ge 0$ such that $y=f^{-m}(x)\in U$. Hence, $f^m(y) = x\in V$, which is a contradiction.
\end{proof}

\subsubsection{Generic results} We collect in a single statement the $C^1$-generic results we shall use.

\begin{theorem}
	\label{t.generic}
	There exists a residual set $\mathcal{R}$ of $C^1$ diffeomorphisms such that every $f\in\mathcal{R}$ satisfies:
	\begin{enumerate}
		\item\label{itemKS} $f$ is Kupka-Smale (see \cite{palis}).
		\item If $\sigma\in\fix(f)$ and $\delta_{\sigma}$ is a physical measure then $\big|\det(Df(\sigma))\big|<1$ (see \cite[Lemma 4.7]{santiago2018dirac}). Moreover, the homoclinic class $H(\sigma)$ is non-trivial (this follows from the connecting lemma and a standard semicontinuity argument). 
		\item\label{itemfour} There exists a residual set $\mathcal{R}_f\subset M$ such that
		if $x\in\mathcal{R}_f$, then $\omega(x)$ is a Lyapunov stable set. Moreover, if a homoclinic class intersects $\omega(x)$, then they must coincide (\cite{MP,Bonatti-Crovisier}).
		\item If $\sigma\in\fix(f)$ is such that $|\det(Df(\sigma))|<1$ and the homoclinic class $H(\sigma)$ is Lyapunov stable, then $H(\sigma)$ admits a dominated splitting $E\oplus F$ (see \cite[Theorem 1.2]{potrie2010generic}) 
		\item\label{locallygenericlyap} Locally generically Lyapunov stability is robust: there exists a neighbourhood $\cU$ of $f$ such that for every $g\in\cU$ there exists a continuation $\sigma_g$. Moreover, if $g\in\mathcal{R}\cap\cU$, then $H_g(\sigma_g)$ is Lyapunov stable (\cite{Bonatti-Crovisier} and Conley theory \cite{conley1978isolated}). 
	\end{enumerate}
\end{theorem} 

\subsubsection{A perturbative result}

Theorem \ref{t.bobo} below (which borrows ideas used in \cite{potrie2010generic}) gives a criterion to mix the stable manifold of a fixed point with the stable manifold of a sink, by a $C^1$ small perturbation. The application of this result is the main technical difference of our proof with the previous work \cite{santiago2018dirac} of the third author. The proof is a combination of a result in Bochi-Bonatti \cite{bochi2012perturbation} with Gourmelon's version of Franks' Lemma \cite{gourmelon2016franks}. 

\begin{theorem}
	\label{t.bobo}
	Let $f$ be a diffeomorphism of a $d$-dimensional compact manifold, and $\gamma_n=O(p_n)$ be a sequence of periodic orbits whose periods 
	$\pi(p_n)$ tend to infinity and that converge in the Hausdorff topology to a compact set $\Lambda$. Assume that 
	\begin{enumerate}[(a)]
		\item $\Lambda$ admits a dominated splitting 
		$$E\oplus F.$$
		\item There exists $0<s\leq\dim(E)$ such that for every $n$, the point $p_n$ is hyperbolic of saddle type with stable index $s$. 
		\item There exists a hyperbolic fixed point $\sigma\in\fix(f)$ of saddle type such that 
		$$W^u(\sigma)\tcap W^s(\gamma_n)\neq\emptyset,$$
		for every $n$.
		\item For every $\delta>0$, one has $|\det(Df^{\pi(p_n)}(p_n)|_F)|<1+\delta$, for every $n$ large enough.
	\end{enumerate}
	Then, given $\eps>0$ there is $N$ such that for every $n\geq N$ there exists an	$\eps$-$C^1$-perturbation $g$ of $f$ with support in an arbitrarily small neighbourhood of $\gamma_n$ (in particular, $\sigma$ is also fixed by $g$), preserving the orbit $\gamma_n$,
	and such that $\gamma_n$ is a sink for $g$ and 
	$$W^u_g(\sigma)\tcap W^s_g(\gamma_n)\neq\emptyset.$$ 
\end{theorem}
\begin{proof}
Observe, to begin, that by standard properties of domination, and by condition (b), for each $n$ the tangent space to the stable manifold $W^{s}(\gamma_n)$ along $\gamma_n$ is contained in the bundle $E$. Now, fixes a compact piece $K_n\subset W^s(\gamma_n)$ so that $K_n$ intersects transversally $W^u(\sigma)$. Consider the finest dominated splitting $E_1\oplus\dots\oplus E_k$ over $\Lambda$ and write $F=E_j\oplus\dots\oplus E_k$. Conditions (a) and (b) allow us to apply Theorem 4.11 of \cite{bochi2012perturbation} to create a continuous path of linear cocycles over the finite set $\gamma_n$, starting with the derivative cocycle of $f$ along $\gamma_n$ end ending with a cocycle $A$ having all Lyapunov exponents equal inside each bundle $E_\ell$, for $\ell=j,\dots,k$, while those in the bundle $E$ remain untouched. By condition (d) the sums of these Lyapunov exponents is very small, provided that $n$ is large. This allows us to produce another continuous path of cocycles which starts at $A$ and ends in a cocycle $B$ having only negative Lyapunov exponents in the bundle $F$, and the same Lyapunov exponents as those of $\gamma_n$ inside $E$ (this argument is a variant of Lemma 4.2 in \cite{santiago2018dirac}). Applying Theorem 1 in \cite{gourmelon2016franks}, there exists a diffeomorphism $g$ which is $C^1$ close to $f$, coincides with $f$ along $\gamma_n$ and outside a small neighbourhood of $\gamma_n$, which we can assume do not contains $K_n$. In particular, $g$ preserves $\sigma$ and moreover 
\begin{itemize}
	\item $K_n\subset W^s_g(\gamma_n)$;
	\item the derivative of $g$ along $\gamma_n$ equals the cocycle $B$.
\end{itemize}
Therefore, $\gamma_n$ is a sink for $g$ and $W^u_g(\sigma)\tcap W^s_g(\gamma_n)\neq\emptyset,$ concluding.      
\end{proof}

\subsubsection{Entropy estimation}\label{subseceleonora} The main tool we shall employ is the result below.

\begin{theorem}[Catsigeras-Cerminara-Enrich \cite{catsigeras2015pesin}]
	\label{t.eleonora}
	Let $\Lambda$ be a Lyapunov stable set for $f\in\dif^1(M)$. Suppose that there exists a dominated splitting $T_{\Lambda}M=E\oplus F$. Then, for every physical measure $\mu$ supported in $\Lambda$ one has $$h_{\mu}(f)\geq\int\sum_{l=1}^{\dim F}\chi_l(x)d\mu,$$
	where $h_\mu(f)$ is the metric entropy of $f$ with respect to $\mu$.
\end{theorem}

\medskip

We are now in position to give the proof of Theorem~\ref{mt.generic}.

\subsection{Proof of Theorem~\ref{mt.generic}}
Let $f\in\mathcal{R}$, where $\mathcal{R}$ is the residual set of Theorem~\ref{t.generic}. Let $\sigma\in\fix(f)$ be such that $\delta_{\sigma}$ is a physical measure. Assume by contradiction that $\sigma$ is not a sink, and that $\cB(\delta_{\sigma})$ is not a nowhere dense set. Then, since $f$ is Kupka-Smale (\eqref{itemKS} of Theorem~\ref{t.generic}) and since $\sigma$ cannot be a source, we have that $\sigma$ is a a saddle-type hyperbolic fixed point. Moreover, as $\operatorname{Int}\overline{\cB(\delta_{\sigma})}\neq\emptyset$, there exists some ball ${B(x,r)}\subset M$ such that $\cB(\delta_{\sigma})$ is dense inside $B(x,r)$. By Lemma 3.3 of \cite{santiago2018dirac} the set $\{y\in B(x,r);\,\delta_{\sigma}\in\cM(y)\}$ is residual in $\overline{B}(x,y)$. Thus, there exists a point $y\in\mathcal R_f$ (the generic subset of $M$ given in \eqref{itemfour} of Theorem~\ref{t.generic}) such that $\delta_{\sigma}\in\cM(y)$ and, in particular, $\sigma\in\omega(y)$. By Theorem~\ref{t.generic}, this implies that the homoclinic class $H(\sigma)$ is a Lyapunov stable set, and then gives a dominated splitting $T_{H(\sigma)}M=E\oplus F$. 

Let $\cU$ be the $C^1$ neighbourhood of $f$ given by \eqref{locallygenericlyap} of Theorem~\ref{t.generic}. We have that for every $g\in\mathcal R\cap\cU$, the homoclinic class $H(\sigma_g)$ is a Lyapunov stable transitive set which is non-trivial (\emph{i.e.}, not reduced to a single periodic orbit). We claim that $\big|\det(Df(\sigma))|_F\big|>1$. Indeed, let us assume on the contrary that
$$\big|\det(Df(\sigma))|_F\big|\leq 1.$$
Then, as the homoclinic class is not trivial we can create a sequence of periodic orbits $\gamma_n=O(p_n)$, all of them heteroclinically related with $\sigma$, and which spend arbitrarily large portions of their orbit as close as we may please to $\sigma$ (this fact is contained in \cite[Lemma 1.10]{MR2018925}, see \cite[Claim 4.4]{santiago2018dirac} for a sketch of proof). Thus, this sequence satisfies the assumptions of Theorem~\ref{t.bobo} and we can create a sink $\gamma=\gamma_n$ for an arbitrarily small perturbation $g$ of $f$, with $\sigma\in\fix(g)$, and such that the unstable manifold of $\sigma$ intersects the basin of the sink $\gamma$. Since this is an open condition, we can require that $g\in\mathcal{R}\cap\cU$. However, this implies that $H_g(\sigma)$ is Lyapunov stable and so, by Lemma \ref{LemLyapUnst}, $W^u_g(\sigma)\subset H_g(\sigma)$. Since $g$ was created such that $\gamma\subset\overline{W^u_g(\sigma)}$, this implies that $\gamma\subset H_g(\sigma)$. But $H_g(\sigma)$, being a transitive set, contains no sinks. This gives a contradiction, and proves that $|\det(Df(\sigma)|_F)|>1$. Now, since $\delta_{\sigma}$ is a physical measure we can apply Theorem~\ref{t.eleonora} and obtain that$$h_{\delta_{\sigma}}(f)\geq\sum_{i=1}^{\dim F}\log\lambda_i,$$where the numbers $\lambda_i$ are the modulus of the eigenvalues of $Df(\sigma)$ in the subspace $F$. Since $\det(Df(\sigma)|_{F})>1$, one obtains that $h_{\delta_{\sigma}}(f)>0$. This contradiction ends the proof of Theorem \ref{mt.generic}.\qed

\section{A figure-eight attractor with affine returns}
\label{sec.figoito}

In this section we construct a suitably devised figure-eight like attractor in $\R^2$. We shall \emph{glue} a hyperbolic linear flow with a rotation, with the main feature being a perturbation that \emph{undoes} a part of the non-linearity raised by the gluing. This will gives us a diffeomorphism having a saddle type hyperbolic fixed point, whose statistical basin contains a wandering domain (an open set whose future iterates never intersects it back), and has an important technical feature: every time that this wandering domain comes close to the stable manifold of the saddle, the ``return map'' is \emph{affine} (see Proposition~\ref{buildTowers} below where we state this property in precise terms). 

More than that, after a suitable rescaling, this kind of \emph{first return map} is just a rotation by $\pi/2$. This nice property will be crucial to enable us to further perturb this diffeomorphism in order to withdraw points from the wandering domain in the subsequent sections.

Hence, despite this initial step being rather simple, we need to be very careful with the dynamics in the region that will compose our wandering domains. Some additional work is also needed to complete the dynamics appropriately with sinks and sources, a feature which will be important later for the construction of \emph{nowhere dense invariant sets}.

\subsection{Boxes and towers}\label{subsecboxes}
Fix a constant $\sigma>1$\nomenclature{$\sigma$}{Eigenvalue of $f_0$ at $O$, with $\sigma>1$} and numbers $\tilde{a},\tilde{b}$ such that $1<\tilde{a}<\tilde{b}<\sigma$. Given these constants we take 
\[a\eqdef\tilde{a}+\frac{\tilde{b}-\tilde{a}}{4}\qquad
\textrm{and}\qquad
b\eqdef\tilde{b}-\frac{\tilde{b}-\tilde{a}}{4}.\]
In particular, $b-a=\frac{\tilde{b}-\tilde{a}}{2}$ and $a+b=\tilde{a}+\tilde{b}$. We assume further that $\tilde{b}-\tilde{a}<\frac{\sigma-1}{5}$. This last requirement, which is not sharp at all, is important to give us enough space to dissolve appropriately the perturbations we are going to make in the next sections. For simplicity we denote $I=[a,b]$ and $\widetilde{I}=[\tilde{a},\tilde{b}]$. Notice that, by our choice of $a$ and $b$, these two intervals are centred around the same point. We are going to use this later. We denote by $s_v:\R^2\to\R^2$ the symmetry with respect to the vertical axis, \emph{i.e.}, the linear map $s_v(x,y)=(x,-y)$. We use $s_h:\R^2\to\R^2$ to denote similarly the linear symmetry with respect to the horizontal axis, \emph{i.e.}, $s_h(x,y)=(-x,y)$. 

\begin{definition}[Boxes]
\nomenclature{$I$}{$I = [a,b]\subset (1,\sigma)$, basis of the tower}\nomenclature{$\tilde I$}{$\tilde I = [\tilde a,\tilde b]\subset (1,\sigma)$, basis of the extended tower}\nomenclature{$\tilde a$}{Number such that $\tilde I = [\tilde a,\tilde b]$}\nomenclature{$\tilde b$}{Number such that $\tilde I = [\tilde a,\tilde b]$}\nomenclature{$a$}{Number such that $I = [a,b]$}\nomenclature{$b$}{Number such that $I = [a,b]$} For each $n\in\mathbb{N}$ the sets     
\begin{equation}\label{eqBox}
S_n=I\times\sigma^{-n} I,\quad
\tilde{S}_n=\tilde{I}\times\sigma^{-n}\tilde{I}\quad\mbox{and}\quad \tilde{U}_n=\sigma^{-n}\tilde{I}\times \tilde{I},
\end{equation}
are called, respectively, \emph{stable boxes}\nomenclature{$S_n$}{$S_n=I\times\sigma^{-n} I$, stable boxes}, \emph{extended stable boxes}\nomenclature{$\tilde S_n$}{$\tilde{S}_n=\tilde{I}\times\sigma^{-n}\tilde{I}\supset S_n$, extended stable boxes} and \emph{unstable boxes}\nomenclature{$\tilde U_n$}{$\tilde{U}_n=\sigma^{-n}\tilde{I}\times \tilde{I}$, extended unstable boxes}.	We shall use also their images by the symmetry maps $s_h$ and $s_v$, and shall denote by $S_n^e\eqdef s_v(S_n)$, which we call the \emph{exterior stable boxes}\footnote{The adjective \emph{exterior} refers to the fact the boxes $S_n^e$ will be contained in the exterior connected components of the complement of the figure-eight attractor we shall produce in this section.}.
\end{definition}
\begin{remark}
Notice that $n$ successive applications of the map $(x,y)\in\R^2\mapsto(\sigma^{-2}x,\sigma y) \in \R^2$ send the stable extended box $\tilde{S}_n$ diffeomorphicaly onto the corresponding unstable box $\tilde{U}_{2n}$. 
\end{remark}
The space between the boxes $S_n$ and $\tilde{S}_n$ will be used in the next section to dissolve some perturbations we are going to make.

\subsubsection{Rescaling boxes} For each $n\in\nt$ we consider the affine map\nomenclature{$L_n$}{Linear map mapping the stable box $S_n$ to $[-1,1]^2$}
\begin{equation}
\label{e.ln}
L_{n} \begin{pmatrix} x\\y \end{pmatrix}
= \begin{pmatrix} 
\frac{2}{b-a}x - \frac{a+b}{b-a}\\
\frac{2\sigma^{n}}{b-a}y - \frac{a+b}{b-a}\\
\end{pmatrix}
\:\text{and its inverse}\:
L_n^{-1} \begin{pmatrix} x\\y \end{pmatrix}
= \begin{pmatrix} 
\frac{b-a}{2}x + \frac{b+a}{2}\\
\sigma^{-n}\frac{b-a}{2}y + \sigma^{-n}\frac{b+a}{2}
\end{pmatrix}
\end{equation}
which identifies each stable box $S_n$ with the square $[-1,1]^2$. Notice that the map $L_n$ also identifies each extended stable box $\widetilde{S}_n$ with a square $[\alpha,\zeta]^2$, for some appropriate choice of $\alpha<-1$ and $\zeta>1$.

\subsubsection{Affine returns} Before stating precisely the main result of this section, we need a couple more definitions. To avoid introducing heavier notations, we denote the map $(x,y)\in\R^2\mapsto(\sigma^{-2}x,\sigma y)\in\R^2$ simply by $\operatorname{Diag}(\sigma^{-2},\sigma)$. Recall that a diffeomorphism of the plane is said to be compactly supported if it equals the identity outside a ball centred at the origin.

\begin{definition} Let $g\in\Diff^1(\R^2)$ with a saddle-type hyperbolic fixed point $O$. We say that $g$ has a \emph{figure eight attractor} at $O$ if $W^s(O)=W^u(O)$.
\end{definition} 

We are ready to state the main result of this section.

\begin{prop}\label{buildTowers} There exist a compactly supported diffeomorphism $f_0\in \Diff^\infty(\R^2)$ and two positive integers $n_0,k_0\in\nt$ such that the origin $O=(0,0)\in\R^2$ is a saddle-type hyperbolic fixed point for $f_0$\nomenclature{$f_0$}{Initial diffeomorphism, with a Bowen loop}, $f_0$ has a figure eight attractor at $O$ and moreover, the following seven additional properties are satisfied.
\begin{enumerate}[(i)]
\item $f_0(p,p)=(p,p)$, where $p=\sigma^4$.
\item\label{p1} There exists a neighbourhood $\mathcal V$ of $O$ (depicted in Figure~\ref{FigD}) such that $f_0|_{\mathcal V}=\operatorname{Diag}(\sigma^{-2},\sigma)$. Moreover, if $n\geq n_0$ then $f_0^{\ell}(\tilde{S}_n)\subset\mathcal V$, for all $\ell=0,\dots,n$.
\item\label{p1'} For any $(x,y)\in\mathcal{V}\cap (\R_+\times [\sigma^3,\sigma^4])$, one has $f_0^\ell (x,y)\notin\mathcal V$ for any $0<\ell<k_0-5$ and $f_0^{k_0-5} (x,y)\in\mathcal V$;
\item\label{p3}If $(x,y)\in(-\sigma^{-2n_0},\sigma^{-2n_0})\times[\tilde a,\tilde b]$, then\[f_0^{k_0}(x,y)=R(x,y),\]where $R$ is the rotation of angle $\pi/2$ and centre $(\frac{a+b}{2},\frac{a+b}{2})$, in the positive (counter-clockwise) sense. If $(x,y)\in(-\sigma^{-2n_0},\sigma^{-2n_0})\times[-\tilde{a},-\tilde{b}]$ then$$f_0^{k_0}(x,y)=s_h\,R\big(s_v(x,y)\big)\,.$$In particular, for all $n \geq n_0$, if $(x,y)\in\tilde{S}_n$, then
$$f_0^{n+k_0}(x,y)=\big(a+b-\sigma^ny\,,\,\sigma^{-2n}x\big) \in \tilde{S}_{2n}$$
and if $(x,y)\in s_h(\tilde{S}_n)$,
$$f_0^{n+k_0}(x,y)=\big(-a-b+\sigma^ny\,,\,\sigma^{-2n}x\big);$$
\item\label{p5} 
$f_0^{k_0}\left([-\sigma^{-2n_0},\sigma^{-2n_0}]\times[\sigma^{-1}\tilde{b},\tilde{a}]\right)\subset [\tilde{b},\sigma^2\tilde{a}]\times[-\sigma^{-n_0-4},\sigma^{-n_0-4}]$ and also 
\[f_0^{k_0}\left([-\sigma^{-2n_0},\sigma^{-2n_0}]\times[-\sigma^{-1}\tilde{b},-\tilde{a}]\right)\subset [-\tilde{b},-\sigma^2\tilde{a}]\times[-\sigma^{-n_0-4},\sigma^{-n_0-4}];\]

 
\item\label{p6} There exists a neighbourhood $\mathcal{W}^{\prime}$ of $O$ such that $f_0|_{\R^2\setminus\mathcal{W}^{\prime}}=\operatorname{Id}$  
Moreover, if $\mathcal{L}$ denotes the figure-eight attractor of $f_0$, then $\R^2\setminus\mathcal{L}$ has three connected components, two of which, $\mathcal{L}_r^i$ and $\mathcal{L}_\ell^i$ (the \emph{interior} components), are topological disks with $O$ at their boundary and one, $\mathcal{L}^e$, which is unbounded (the \emph{exterior} component). There exists a $C^{\infty}$ disk $Q_R$ such that, such that $\mathcal{V}$ is contained in the unbounded connected component of $\R^2\setminus \partial Q_R$.
If $z\in\mathcal{L}_{r}^{i}\setminus Q_R$ then $\alpha_{f_0}(z)=\partial Q_R$ and $\omega_{f_0}(z)\subset\mathcal{L}$. There exists also a $C^{\infty}$ disk $Q_L$ contained in $\mathcal{L}^i_\ell$ with similar properties. 
\item\label{p8} There exists a $C^{\infty}$ disk $Q_e$ whose boundary is contained in $\mathcal{L}^e$ such that if $z\in\mathcal{L}^e$ belongs to the bounded component of $\R^2\setminus \partial Q_e$ then $\alpha_{f_0}(z)\subset \partial Q_e$ and $\omega_{f_0}(z)=\mathcal{L}$.
\end{enumerate}
\end{prop}


\subsubsection{Towers} As indicated in item \eqref{p6} the global dynamics of $f_0$ is quite simple. In what concerns the perturbations of $f_0$ that we are going to perform later, the most crucial feature is \eqref{p3}. Item \eqref{p5} is a technical feature that we shall use in the next section. As we shall see in the next subsection, item \eqref{p3} implies the following.

\begin{coro}\label{c.bacia} Let $f_0\in\dif^{\infty}(\R^2)$ be given by Proposition~\ref{buildTowers}. Then $\tilde{S}_n\subset\cB(\delta_O)$, for every $n\geq n_0$.
\end{coro}

For future reference, we introduce the following definition

\begin{definition}[Towers]\label{def.towers} We define the \emph{stable tower of height $n_0$} as the orbit under $f_0$ of all the boxes $S_n$ for $n\geq n_0$
$$\mathcal{S}\eqdef\bigcup_{n\ge n_0} \bigcup_{k\in\mathbb{Z}} f_0^k(S_n).$$
We define similarly the \emph{stable extended tower of height $n_0$}, which we denote\footnote{The dependence of $\cS$ and $\tilde{\cS}$ on $n_0$ is not explicit in our notation because we shall fix once and for all the integer $n_0$.} just $\tilde{\cS}$ by replacing $S_n$ by $\tilde{S}_n$. We shall also use the \emph{exterior stable tower} 
$$\mathcal{S}^e\eqdef\bigcup_{n\ge n_0} \bigcup_{k\in\mathbb{Z}} f_0^k(S^e_n).$$
\end{definition}

Thus, Corollary~\ref{c.bacia} says that the stable extended tower is contained in the statistical basin of the saddle at the origin $O$. In the subsequent sections of the paper our strategy is to introduce new perturbations of $f_0$ which preserve both a neighbourhood of $O$ and the union of the boxes $S_n$. Roughly speaking, the complicated region will be the space between $\tilde{S}_n$ and $S_n$ where these perturbations need to be dissolved. To deal with this difficulty, Property \eqref{p3} will play a central role since it allows us to define a first return map to the tower of boxes $\tilde{S}_n$.

\begin{definition}\label{defgzero} We introduce the \emph{first return map} $g_0:\bigcup_{n=n_0}^{\infty}\tilde{S}_n\to\bigcup_{n=n_0}^{\infty}\tilde{S}_n$ defined by
\[g_0=f_{0}^{n+k_0}\quad\mbox{on $\tilde{S}_n$ for each $n \geq n_0$.}\]
\end{definition}

\begin{remark}
Note that item \eqref{p3} of Proposition~\ref{buildTowers} says that for all $n \geq n_0$, we have $g_0(\tilde{S}_n)=\tilde{S}_{2n}$ and $g_0(S_n)=S_{2n}$. Moreover, $L_{2n}\circ g_0\circ (L_n)^{-1}$ is a rotation of angle $\pi/2$, well defined in some square $[\alpha,\zeta]^2$ where $\alpha<-1$ and $\zeta>1$. 
\end{remark}

Now fix $n \geq n_0$ and some initial point $(x_0,y_0) \in S_n$. For each $d\in\nt$ let$$(x_d,y_d)=g_{0}^{4d}(x_0,y_0)\,,$$and note that $(x_d,y_d) \in S_{16^dn}$ for all $d \in \nt$. A straightforward computation, based on item \eqref{p3} of Proposition~\ref{buildTowers}, shows that$$x_d=x_0 \quad\mbox{and}\quad y_d=\sigma^{(1-16^d)n}y_0\quad\mbox{for all $d\in\nt$.}$$Therefore, the horizontal coordinate $x_d$ is constant in $d$, and moreover, if for each $d$ we define $\tau_{0}^{d}\in[a,b]$ as
\begin{equation}\label{EqTau0d}
\tau_{0}^{d}=\sigma^{16^dn}\,y_d,
\end{equation}
we see that $\tau_{0}^{d}$ is also constant in $d$ (equal to $\sigma^ny_0$). This rigidity of the first return map $g_0$ will be crucial in Section~\ref{sec.orbitexcii}, where we shall give the proof of Theorem~\ref{main.exemplonovo}.

The configuration used to construct the diffeomorphism $f_0$ (that is: to prove Proposition \ref{buildTowers}) is depicted in Figure~\ref{f.figzero}. As we said at the beginning of this section, the idea is simply to glue a rotation with a hyperbolic linear map. This gluing procedure is simpler if we work with vector fields and the first step of the construction is to define a vector field in this way (see Section \ref{seccampoF}). To obtain item \eqref{p3}, however, we need to perturb the time one map of this vector field (see Section \ref{S:connectingtowers}).

\begin{figure}
\begin{center}

\begin{tikzpicture}[scale=.9]
\draw (0,0) node{$\bullet$} node[below left]{$O$};

\draw (3,3) node{$\bullet$} node[right]{$(p,p)$};
\draw (-.5,0) -- (3,0);
\draw (0,-.5) -- (0,3);
\draw (3,0) arc (-90:180:3);
\fill[color=gray, opacity=.06] (0,0) -- (3,0) arc (-90:180:3) -- cycle;

\fill[color=yellow, opacity=.2] (4,1) -- (3,1) to[out=180, in=-90] (1,3) -- (1,4) -- (1.5,4) to[out=-100, in=190] (4,1.5) -- cycle;
\fill[color=red, opacity=.1] (3,0) -- (3,1) to[out=180, in=-90] (1,3) -- (0,3) -- (0,0) -- cycle;

\fill[opacity=.2, color=yellow] (4,1) -- (3,1) -- (3,0) arc (-90:-70.5:3) -- cycle;
\fill[opacity=.2, color=yellow] (1,4) -- (1,3) -- (0,3) arc (180:160.5:3) -- cycle;

\foreach \i in {0,...,4}
{\foreach \j in {0,...,4}
{\draw[fill=blue!20!white] (0.5^\i*.9,0.5^\j*.9) rectangle (0.5^\i*1.5*.9,0.5^\j*1.5*.9);}}
\foreach \i in {1,...,4}
{\draw[fill=blue!20!white] (0.5^\i*.9,2*.9) rectangle (0.5^\i*1.5*.9,3*.9);
\draw[fill=blue!20!white] (2*.9,0.5^\i*.9) rectangle (3*.9,0.5^\i*1.5*.9);}

\draw [->,>=latex, color=red!70!black] (-.5,1.5)node[left]{linear hyperbolic part} to[bend left] (.5,1.55);
\draw [->,>=latex, color=yellow!50!black] (-.5,3.5)node[left]{mixed part} to[bend left] (.5,3.5);
\draw [->,>=latex, color=gray!70!black] (6.2,1.7)node[right]{rotational part} to[bend left] (5.2,1.9);
\draw [->,>=latex, color=blue!80!black] (2.15,1.5)node[right]{$S_n$} to[bend left] (1.15,1.1);
\draw (6.1,3) node[right]{$W^s(O) = W^u(O)$};

\draw[->] (1.8,0) -- (1.45,0);
\draw[->] (1.8,0) -- (1.5,0);
\draw[->] (0,0) -- (0,.85);

\end{tikzpicture}
\caption{\label{f.figzero} The construction of the vector field $F$ in Section \ref{seccampoF}. The time one map of its flow will be perturbed in Section \ref{S:connectingtowers} in order to obtain the desired diffeomorphism $f_0$ of Proposition \ref{buildTowers}.}
\end{center}
\end{figure}
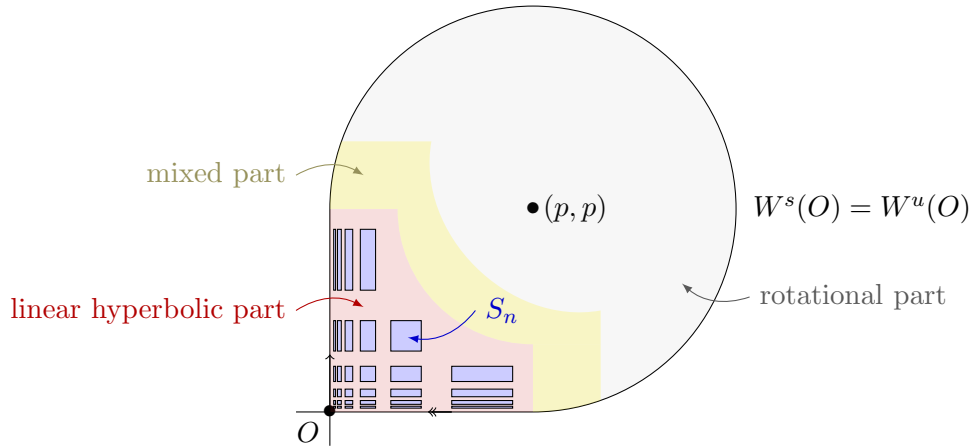

\subsection{Detecting points in the statistical basin} We shall give in this subsection a criterion to check that a point belongs to the statistical basin of a Dirac measure at a hyperbolic fixed point of saddle type, which will prove Corollary~\ref{c.bacia}. Moreover, we will apply this criterion to all subsequent constructions we are going to make. Despite the fact that our constructions are made in $\R^2$, the statement is fairly general (since it relies only on Hartman-Grobman's theorem\footnote{See \cite{palis} for instance.}) and so we state it in this more flexible context. 

\begin{lema}\label{suf} 
Let $M$ be a manifold and let $f\in \dif^1(M)$ be a $C^1$ diffeomorphism over $M$. Assume the existence of a hyperbolic fixed point of saddle type $O\in M$ and consider a Hartmann-Grobman neighbourhood\footnote{Recall that this means that, inside $V$, the map $f$ is topologically conjugate to its linear part $Df$ at $O$.} $V$ of $O$. Suppose that there exists a point $x\in V$ whose forward orbit has the following property: there exists $N\in\nt$ and a sequence of integers $0=a_0<b_0<a_1<b_1<a_2\dots$ such that $a_n\to+\infty$, $a_{n+1}-b_n<N$ and 
\begin{enumerate}
	\item if $\ell\in[a_n,b_n]$ then $f^{\ell}(x)\in V$, while if $\ell\in(b_n,a_{n+1})$ then $f^{\ell}(x)\notin V$
	\item $d(f^{a_{n}}(x), W^s_{loc}(O))\to 0$. 
\end{enumerate}
Then, $x\in\cB(\delta_O)$.
\end{lema}

Let us quickly show how Lemma~\ref{suf} implies Corollary~\ref{c.bacia}: one takes the neighbourhood $\mathcal V$ of $O$ and $k_0$, as given in Proposition~\ref{buildTowers}, and observes that the equality $f_0^{n+k_0}(\tilde{S}_n)=\tilde{S}_{2n}$ implies that all points in $\tilde{S}_n$ satisfy the assumptions of Lemma~\ref{suf}, with $N=k_0-4$, $a_d=(2^d-1)n+dk_0$ and $b_d=(2^{d+1}-1)n+dk_0$.

The proof of Lemma~\ref{suf} is fairly simple but we include it here for the sake of completeness.

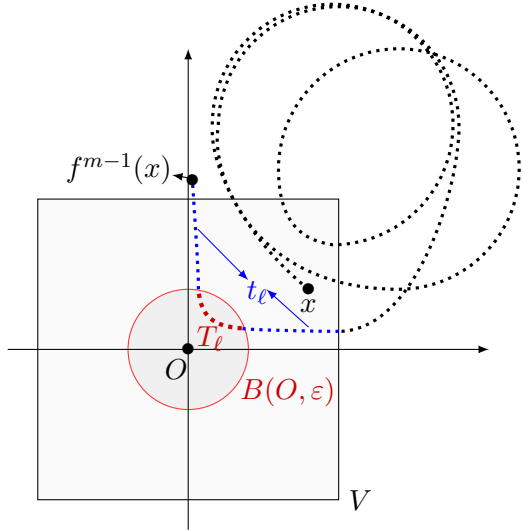
\begin{figure}[h!]
	\centering
    \begin{tikzpicture}[scale=.8]

    \filldraw[fill=black!2!white,draw=black] (-2.5,-2.5) rectangle (2.5,2.5); 
    \filldraw[fill=black!6!white, draw=red!80!white] (0,0) circle (1cm);
    \draw [->,>=latex] (-3,0) -- (5,0);
    \draw [->,>=latex] (0,-3) -- (0,5);
    \draw (0,0) node {$\bullet$} ;
    \draw (-0.2,0) node[below]{$O$};
    \draw[dotted,very thick] (2,1) .. controls (1,1.8) and (0.36,3.0) .. (0.4,3.7);
    \draw[dotted, very thick] (0.4,3.74) arc (180:-90:2) .. controls (2,1.8) and   (1.5,2)  .. (1.5,3) arc (180:-90:2) .. controls (1,1) and (0.36,3.0) .. (0.5,3.7) arc (180:0:2) .. controls (4.5,2.7) and (3.8,0.3) .. (2.5,0.3);
    \draw [blue, dotted, very thick] (2.5,0.3) .. controls (2.2,0.3) and (1,0.32).. (0.89,0.35);
    \draw[color=red!80!black, ultra thick,dotted] (0.89,0.35).. controls (0.37,0.38) and (0.2,0.6) .. (0.17,1);
    \draw[blue,dotted, very thick] (0.17,1).. controls (0.17,1.4) and (0.1,2.3) .. (0.07,2.8);
    \draw (0.37,0.55) node[below,color=red!80!black]{$T_{\ell}$};
    \draw (0.7,-0.7) node[right, color=red!80!black]{$B(O,\eps)$};
    \draw [->,>=latex,blue] (2,0.37) -- (1.3,1);
    \draw [->,>=latex,blue] (0.16,2) -- (1,1.15);
    \draw (1.15,1.3) node[below,blue]{$t_{\ell}$};
    \draw (2,1) node{$\bullet$}; 
    \draw (2,1) node[below]{$x$};
    \draw (2.5,-2.5) node[right]{$V$};
    \draw (0.07,2.8) node {$\bullet$};
    \draw (-0.1,3) node[left]{$f^{m-1}(x)$};
    \draw [->,>=latex] (0.02,2.85) -- (-0.27,2.9);
    \end{tikzpicture}
	\caption{Proof of Lemma \ref{suf}: after a very large iterate, every visit excursion of the orbit of $x$ inside $V$ has a big proportion inside $B(O,\eps)$.}
	\label{f.boxes}
\end{figure}

\begin{proof}
Let $\delta,\eps>0$ be arbitrarily chosen. For each $m\in\mathbb{N}$ denote by
$$v(\eps,m)\eqdef\#\{0\leq \ell\leq m-1;\,f^{\ell}(x)\in B(O,\eps)\},$$
the number of visits of the finite orbit segment $x,\dots,f^{m-1}(x)$ to the ball $B(O,\eps)$. Also, for each $n$ denote by $T\eqdef\#\left\{[a_n,b_n]\cap\{0\leq \ell\leq m-1;\,f^{\ell}(x)\in B(O,\eps)\}\right\}$  the amount of integers within the interval $[a_n,b_n]$ which correspond to one of these visits. For simplicity, we omit the dependence of $T$ on $n$ and on $m$.  
The fact that the dynamics is $C^0$ conjugated with the linear map $Df(O)$ inside $V$ together with condition (2) imply the existence of $n_0$ so that if $n\geq n_0$ then (see Figure~\ref{f.boxes}) 
\[\frac{(b_n-a_n)-T+N}{T}\leq\delta.\]
Thus, by condition (1), every $m>a_{n_0+1}$ can be written as
\[m=a_{n_0}+\sum_{\ell=1}^{k(m)}T_{\ell}+\sum_{\ell=1}^{k(m)}t_\ell,\]
so that $v(\eps,m)=\sum_{\ell=1}^{k(m)}T_{\ell}$ and $\frac{t_\ell}{T_{\ell}}<\delta$, for every $\ell$. 

Also, one has that $v(\eps,m)\to\infty$ as $m\to\infty$, again due to condition (2) and the conjugation with the linear part. Therefore, there exists $n_1>a_{n_0+1}$ large enough so that $m>n_1$ implies that 
\[\frac{v(\eps,m)}{m}=\frac{\sum_{\ell=1}^{k(m)}T_{\ell}}{a_{n_0}+\sum_{\ell=1}^{k(m)}T_{\ell}+\sum_{\ell=1}^{k(m)}t_{\ell}}>\frac{1}{1+2\delta}.\]
Since $\delta$ and $\eps$ are arbitrary, this proves that $x\in\cB(\delta_O)$, as desired.   	
\end{proof}

The rest of this section is devoted to the proof of Proposition~\ref{buildTowers}.

\subsection{A figure-eight attractor}\label{seccampoF} Our goal in this subsection is to build a vector field $F$ on $\R^2$ exhibiting a figure-eight attractor. We shall first explain how to to define $F$ in the first quadrant $\R_+^2$. Consider the linear vector field $X$ on $\R^2$ induced by the diagonal matrix$$\begin{bmatrix}
-\,2\log\sigma & 0\\
0 & \log\sigma
\end{bmatrix}.$$The origin is \emph{dissipative} for the vector field $X$ (the divergence of $X$ is negative, equal to $-\log\sigma$). The associated flow $\{X_t\}_{t\in\R}$ is given by
\begin{equation}\label{EqXt}
X_t(x,y)=(\sigma^{-2t}x,\sigma^{t}y)
\end{equation}
for all $t\in\R$ and $(x,y)\in\R^2$. Notice that the trajectories of the flow are contained in the level curves of $(x,y) \mapsto xy^{2}$, and that the time one map of the flow is the linear diffeomorphism $\operatorname{Diag}(\sigma^{-2},\sigma)$, that is: $X_1(x,y)=(\sigma^{-2}x,\sigma y)$. Fix $p=\sigma^4$ and some $\beta>0$, and consider the affine vector field $Y$ on $\R^2$ given by$$Y(x,y)=\begin{bmatrix}
0 & \beta\\
-\beta & 0
\end{bmatrix}(x-p,y-p)=\beta(y-p,-x+p).$$The associated flow $\{Y_t\}_{t\in\R}$ is the rotational flow around the point $(p,p)$:
$$Y_t(x,y)=(p,p)+\begin{bmatrix}
\cos(\beta t) & \sen(\beta t)\\
-\sen(\beta t) & \cos(\beta t)
\end{bmatrix}(x-p,y-p)\,,$$
for all $t\in\R$ and $(x,y)\in\R^2$. The time one map of this flow is of course the rotation of angle $\beta$ around the point $(p,p)$ in the negative (clockwise) sense. With the vector fields $X$ and $Y$ at hand, we are going to build a new vector field $F$ in the first quadrant. The sets we are about to define are shown in Figure~\ref{FigD}. 

\subsubsection{Choice of constants}\label{secregionsF} We fix $T\in (0,p/2)$ and $\varepsilon\in (0,p/3)$, and consider the two rectangles:
$$R_s=(p,p+T)\!\times\!(0,2\varepsilon)\quad\mbox{and}\quad R_u=(0,2\varepsilon)\!\times\!(p,p+T)\,.$$

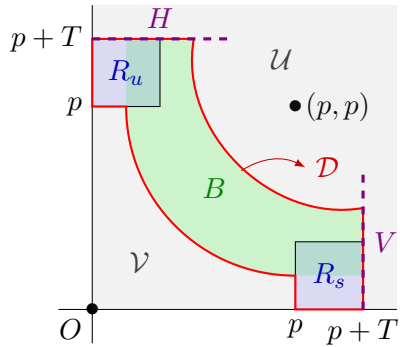
\begin{figure}[h!]
\centering
\begin{tikzpicture}[scale=.9]
\draw (0,0) node{$\bullet$} node[below left]{$O$};

\draw (3,3) node{$\bullet$} node[right]{$(p,p)$};
\draw (-.5,0) -- (4.5,0);
\draw (0,-.5) -- (0,4.5);
\fill[color=gray, opacity=.1] (0,0) rectangle (4.5,4.5);
\draw[color=gray!50!black] (.7,.7) node{$\mathcal V$};
\draw[color=gray!50!black] (2.8,3.7) node{$\mathcal U$};

\fill[color=green, opacity=.15] (4,.5) -- (3,.5) to[out=180, in=-90] (.5,3) -- (.5,4) -- (1.5,4) to[out=-100, in=190] (4,1.5) -- cycle;
\draw[color=green!50!black] (1.8,1.8) node{$\rB$};

\draw (3,0) rectangle (4,1);
\draw[color=blue!70!black] (3.5,.5) node{$R_s$};
\fill[opacity=.1, color=blue] (3,0) rectangle (4,1);
\draw (0,3) rectangle (1,4);
\draw[color=blue!70!black] (.5,3.5) node{$R_u$};
\fill[opacity=.1, color=blue] (0,3) rectangle (1,4);
\draw (3,0) node[below]{$p$};
\draw (4,0) node[below]{$p+T$};
\draw (0,3) node[left]{$p$};
\draw (0,4) node[left]{$p+T$};

\draw[color=red, thick] (4,0) -- (3,0) -- (3,.5) to[out=180, in=-90] (.5,3) -- (0,3) -- (0,4) -- (1.5,4) to[out=-100, in=190] (4,1.5) -- cycle;
\draw[color=red!80!black, <-, >=latex] (3.15,2.1) node[right]{$\mathcal D$} to[bend right] (2.2,2.05);

\draw[color=violet, very thick, dashed] (4,0) -- (4,2)node[midway, right]{$V$};
\draw[color=violet, very thick, dashed] (0,4) -- (2,4)node[midway, above]{$H$};

\end{tikzpicture}
\caption{The two rectangles $R_s$ and $R_u$, and the band $\rB$, for the construction of the vector field $F$. The region $\mathcal{D}=R_s\cup\rB\cup R_u$ is bounded by the red curves. Recall that $p=\sigma^4$.}
\label{FigD}
\end{figure}

In the square $[0 ,p+T]^2$, consider the band $\rB$ determined by the positive orbit of the point $(p+T,\varepsilon)$ under $X$, with the positive orbit of the point $(p+T,3\varepsilon)$ under $Y$. Let
$$\mathcal{D}=R_s\cup\rB\cup R_u\,,$$
and note that $\mathcal{D}$ is a topological disk (open, connected and simply connected). The complement of $\mathcal{D}$ in the first quadrant has two connected components: let $\mathcal{U}$ be the one which contains the point $(p,p)$ and let $\mathcal{V}$ be the one which contains the origin on its boundary.

\subsubsection{Bump functions} Let $\alpha^u:[0,+\infty)\to[0,1]$ be a $C^{\infty}$ bump function satisfying:
\begin{itemize}
	\item $\alpha^u(t)=1$ for $t\in[0,p]$,
	\item $(\alpha^u)'\leq 0$ in $(p,p+T)$ and
	\item $\alpha^u(t)=0$ for $t \geq p+T$.
\end{itemize}
In the same way, let $\alpha^s:[0,+\infty)\to[0,1]$ be a $C^{\infty}$ bump function satisfying:
\begin{itemize}
	\item $\alpha^s(t)=1$ for $t\in[0,p]$,
	\item $(\alpha^s)'\leq 0$ in $(p,p+T)$ and
	\item $\alpha^s(t)=0$ for $t \geq p+T$.
\end{itemize}
Finally, let $\rho:\mathcal{V}\cup\mathcal{D}\cup\mathcal{U}\to[0,1]$ be a $C^{\infty}$ function satisfying:
\[\rho(x,y) = \left\{\begin{array}{ll}
1 & \quad \text{if } (x,y)\in \mathcal{V}\\
\alpha^s(x) & \quad \text{if } (x,y)\in R_s\\
\alpha^u(y) & \quad \text{if } (x,y)\in R_u\\
0 & \quad \text{if } (x,y)\in \mathcal{U}
\end{array}\right.\]
(note that $\rho$ is not defined on $\operatorname{int}(\rB)$, we just make an arbitrary choice on this set so that $\rho(x,y)$ is of class $C^\infty$ and strictly positive in the interior of $B$). 
\subsubsection{Creating the homoclinic loop}
Now, consider the $C^{\infty}$ vector field $Z$ in the first quadrant given by
\[Z=\rho X+(1-\rho)Y.\]
Notice that $Z \equiv X$ in $\mathcal{V}$, and $Z \equiv Y$ in $\mathcal{U}$. By symmetry, we can perform the same construction on the third quadrant. In the sequel we shall prove that we can construct $Z$ in such a way that it has a homoclinic loop in the first quadrant. Therefore, by symmetry, $Z$ will present another homoclinic loop in the third quadrant.

\begin{lema}\label{l.loop} For every $T>0$ there exists a choice of $\alpha^s$ and $\alpha^u$ such that the saddle singularity of $Z$ at the origin has a homoclinic loop.
\end{lema}

For the proof of this lemma, the terminology below will be helpful.

\begin{definition}
	\label{def.hit}
Let $X$ be a vector field on $\R^2$ and let $A\subset\R^2$ be a non-empty set. We say that the positive (resp. negative) $X$-orbit of $x\in\R^2$ \emph{hits $A$ for the first time} at a point $a\in A$ if there exists $T>0$ (resp. $T<0$) such that $X_T(x)=a$ and $X_t(x)\notin A$ for every $0< t <T$ (resp. $T<t< 0$). 
\end{definition}

\begin{proof}[Proof of Lemma \ref{l.loop}] 
Let $H$ denote the unitary horizontal line segment through the point $(0,p+T)$, \emph{i.e.}, $H\eqdef\{(x,p+T);\,x\in[0,1]\}$. We consider similarly the unitary vertical segment $V$ through $(p+T,0)$. 
	
Let $\delta_0\in [0,1]$ be such that the positive $Y$-orbit of $(0,p)$ hits $H$ for the first time at $(\delta_0,p+T)$. Observe that $\delta_0>0$, and that the negative orbit of $(0,p)$ hits $V$ for the first time at $(T+p,\delta_0)$.



%

\begin{claim}
\label{claim.intermediate.value.argument}
For every $\delta^{\prime}\in(0,\delta_1)$ there exists a choice of $\alpha^s$ such that the negative $Z$-orbit of $(p,0)$ hits $V$ for the first time at $(p+T,\delta^{\prime})$, and a choice of $\alpha^u$ such that the positive $Z$-orbit of $(0,p)$ hits $H$ for the first time at $(\delta,p+T)$.
\end{claim}
	
Let us complete the proof of the lemma assuming Claim~\ref{claim.intermediate.value.argument}. Choose $0<\delta<\delta_0$. With this choice, take $\alpha^s$ and $\alpha^u$ given by Claim~\ref{claim.intermediate.value.argument}, and let $Z$ be the corresponding vector field.
Then, $(\delta,p+T)$ is in the positive orbit under $Z$ of $(0,p)$ and $(p+T,\delta)$ in the negative orbit of $(p,0)$. Since the $Y$ orbit segment joining $(\delta,p+T)$ to $(p+T,\delta)$ is contained in $\cU$ and since $Z=Y$ inside $\cU$, this proves that $(p,0)$ is in the positive $F$ orbit of $(0,p)$, and thus that $O$ exhibits a homoclinic loop.  
\medskip

We are left to prove Claim~\ref{claim.intermediate.value.argument}. For this, notice that if we choose $\alpha^u$ constant and equal to $1$ in an interval $[p,p+t^{*}]$, with $t^{*}$ very close to $T$, then the positive $Z$-orbit of $(0,p)$ will hit $H$ for the first time at a point close to $(0,p+T)$.
On the other hand, if we choose $\alpha^u$ to be constant equal to zero in an interval $[p+t^{*},p+T]$, with $t^{*}$ very close to $0$, then the positive $Z$-orbit of $(0,p)$ will hit $H$ for the first time at a point close to $(\delta_0,p+T)$. By the Intermediate Value Theorem any point between $(0,p+T)$ and $(\delta_0,p+T)$ is within reach. A similar argument applies for the choice of $\alpha^s$. This finishes the proof of the claim and the lemma.
\end{proof}

%

\subsubsection{Creating the figure-eight attractor}

We have constructed a vector field $Z$ on the closed set $\{(x,y)\in\R^2;xy\geq0\}$. It presents a homoclinic loop, which is the set
\[
\mathcal{L}\eqdef\{Z_{t}(0,p);\,t\in\R\}\cup\{Z_t(0,-p);\,t\in\R\}\cup\{O\}.
\] 
Working with bump functions on the second and fourth quadrants, in a very similar way as we did above, the definition of $Z$ can be extended to a bounded neighbourhood $\mathcal{W}$ of the origin so that $\mathcal{L}\subset\mathcal{W}$ and moreover $Z=X$ along an open set which we shall, by abuse of notation, call $\mathcal{V}$, where the boundary of $\mathcal{V}$ is made of four line segments and four hyperbola segments. More precisely:
\begin{itemize}
	\item $\mathcal{V}\subset\mathcal{W}$;
	\item $Z|_{\overline{\mathcal{V}}}=X$;
	\item for every $0<y<\eps$ the $X$ orbit of $(p,y)$ hits the horizontal line $\R\times\{p\}$ for the first time at a time $T>0$ and the orbit segment $X_{(0,T)}(p,y)$ is entirely contained in $\mathcal{V}$;
	\item the sets $\mathcal V$ and $\mathcal{W}$, as well as the field $Z$, are invariant under $-Id$ the and $\mathcal V$ is invariant under $s_h$ e $s_v$.
\end{itemize}

We remark that to perform this construction one has to work with rotations centred at $(-p,p)$ and $(p,-p)$, but in the second and in the fourth quadrants we do not care about the precise form of the vector field far from $\mathcal{V}$. In fact, in this way we define a smooth vector field in a closed set of the plane and apply Whitney's extension theorem \cite{whitney1934analytic} to define a vector field (which for simplicity we shall still denote by $Z$) on the entire plane. 


To complete the construction of our desired vector field we shall ``stop'' $Z$ far from $\mathcal{L}$, by taking a smooth bump function. More precisely, we use the fact that $Z$ has no zeros in $\partial \mathcal{W}$ to take a slightly larger open set $\mathcal{W}^{\prime}$ so that $Z$ has no zeros in the closure of $\mathcal{W}^{\prime}$ either. We can assume that $\mathcal{W}^{\prime}$ is a topological disk whose boundary is smooth. We take then the function $u:\R^2\to[0,1]$ satisfying $u^{-1}(0)=\R^2\setminus\mathcal{W}^{\prime}$ and $u^{-1}(1)=\overline{\mathcal{W}}$. Thus, the vector field $F=uZ$ has no zeros inside $W^{\prime}$ (other than $O$), but vanishes identically on its boundary.

We will prove in the next lemma that $F$ has a well described global dynamics.

Notice that $\R^2\setminus\mathcal{L}$ has three connected components, two of which, say $\mathcal{L}_{\ell}^i$ and $\mathcal{L}_r^i$, are topological disks having $O$ on their boundary, and one $\mathcal{L}^e$ which is unbounded. 

\begin{lema}
	\label{l.lemaum}
The vector field $F=uZ$ satisfies the following.
\begin{enumerate}[(a)]
\item There exists a point $q^r\in\mathcal{D}$ in the first quadrant whose orbit $O_F(q^{r})$ is periodic and satisfies 
\[O_F(q^{r})\cap\mathcal{V}=\emptyset.\]
Moreover, for every $z\in\mathcal{L}^i_{r}$ belonging to the unbounded component of $\R^2\setminus O_F(q^{r})$ it holds that $\alpha_F(z)=O_F(q^{r})$ and $\omega_F(z)\subset\mathcal{L}$. There exists also a point $q^\ell$ in the third quadrant with similar properties;  
\item There exists a topological disk which contains $\mathcal{L}$ and whose boundary is a curve $\gamma$ of class $C^{\infty}$ contained in $\overline{B}(O,R)$, and such that if $z\in\mathcal{L}^e$ belongs to the bounded component of $\R^2\setminus \gamma$ then $\omega_F(z)=\mathcal{L}$ and $\alpha_F(z)\subset\gamma$.	
\end{enumerate}
\end{lema} 

\begin{proof}
Since $F|_{\overline{\mathcal{W}}}=Z$, to prove (a) its enough to argue with $Z$ directly.

We will use the vertical segment $V'=\{(p,y);y\in[0,1]\}$. Observe that from the definition of $Z$, for every $q=(p,y)\in V'$, there exists $v=(p,y^{\prime})\in V'$ such that the $Z$ orbit of $q$ hits $V^{\prime}$ for the first time at $v$ (recall Definition~\ref{def.hit}). This allows to consider the first return map $P^Z:V'\to V'$. 

Notice that $P^Z$ is \emph{a priori} not defined at $(p,0)$. The other boundary point of the segment $V'$ is fixed by $P^Z$. We claim that there exists a smallest fixed point $q^r$ of $P^Z$, \emph{i.e.}, we claim the existence of $q^r=(p,y)\in V'$ such that if $z=(p,y^{\prime})$ with $0<y^{\prime}<y$, then $z$ is not fixed by $P^Z$.

Indeed, since $Z|_{\overline{\mathcal{V}}}=X$ and by \eqref{EqXt} we can write $Z_t(x,y)=(\sigma^{-2t}x,\sigma^ty)$ for every $(x,y)\in\mathcal{V}$. Therefore, a direct calculation shows that any point $(p,\lambda)\in \mathcal V$, with $\lambda>0$, hits $\R\times \{p\}$ for the first time at $(\lambda^2/p,p)\in\mathcal V$. Remarking that $\lambda^2/p = o(\lambda)$, as $\lambda\to 0$, and using the fact that the vector field $Z$ is continuous, we deduce that the map $P^Z$ is decreasing in some neighbourhood of the bottom boundary point of $V'$. It suffices now to take $y=\inf\{\lambda>0;(p,\lambda)\in\fix(P^Z)\}$ and the claim is established. Notice that this argument also allows us to extend continuously $P^Z$ to the point $(p,0)$ by declaring it fixed. Observe also that the point $(p,p)$ is contained in the bounded component of $\R^2\setminus O_Z(q^r)$ (as well as the whole orbit of $(p,1)$).

Take now $z\in\mathcal{L}^i_r$ and suppose it belongs to the unbounded connected component of $O_Z(q^r)$. By definition of $Z$, there must exist a point $z^{\prime}\in V'$ (below $q^r$) so that the $Z$ orbit of $z$ hits $V'$ for the first time at $z^{\prime}$. Remark that the restriction of $P^Z$ to $[(p,0),(p,y)]$ is an interval map which fixes the boundaries and whose graph is below the diagonal in the interior of its domain of definition. Thus every positive orbit accumulates on the lower boundary of its domain, and every negative orbit accumulates on the upper boundary of its domain. This proves that $\omega_Z(z^{\prime})\subset\mathcal{L}$ and $\alpha_Z(z^{\prime})=O_Z(q^r)$, and completes the proof of item (a).

To prove (b), consider a vertical segment $V^{\prime\prime}=\{(p,y);-\delta< y<0\}$. We choose $\delta$ so that $\partial V^{\prime\prime}\subset\partial\mathcal{W}^{\prime}$. We claim that there is a well defined first return map $P^F:V^{\prime\prime}\to V^{\prime\prime}$.
Assume this is not the case. Then, there must exits a point $z\in V^{\prime\prime}$ whose future orbit never hits $V^{\prime\prime}$. In particular, the orbit of $z$ do not accumulate on $\mathcal{L}$. Since $F$ has no zeros inside $W^{\prime}\cap\mathcal{L}_e$ this allows to apply Poincar\'e-Bendixon's theorem \cite{palis} to conclude that the orbit of $z$ accumulates on a periodic orbit. This periodic orbit is not allowed to cross $V^{\prime\prime}$ for otherwise $z$ would do so, and it is contained in $\mathcal{W'}$ because $F\equiv 0$ on $\mathcal{W'}^\complement$. Therefore, this periodic orbit bounds a disk contained in $\mathcal{W'}$ which do not contains $O$ in its interior. However, this implies that $F$ has another zero inside $\mathcal{W}^{\prime}$, a contradiction\footnote{We have used here the well known fact that every periodic orbit of a planar vector field bounds a disk containing a zero of the vector field inside. This can be proved combining Poincar{\'e}-Bendixon's theorem with Zorn's lemma.}.  

Now, as $Z|_{\mathcal{V}}=X$, we can argue as in item (a) to prove that $P^F:V^{\prime\prime}\to V^{\prime\prime}$ \emph{increases} the vertical coordinate of points, and thus $P^F$ either has a lowest fixed point or it fixes no point. In the latter case we declare $\gamma$ to be the boundary of $\mathcal{W}^{\prime}$ and in the former we declare $\gamma$ to be the $F$ orbit of the lowest fixed point of $P^F$. In either case the conclusion now follows, as in item (a), due to the dynamics of the one-dimensional map $P^F$. 
\end{proof}
  
\subsection{Affine returns: proof of Proposition~\ref{buildTowers}}\label{S:connectingtowers} To complete the proof of Proposition~\ref{buildTowers}, we shall now perturb the time one map of the vector field $F$ in order to ``undo" some of the raised non-linearities. This will ensure the announced properties of the first return map to the extended stable boxes (more precisely, item \eqref{p3} of Proposition~\ref{buildTowers}). With these purposes, we will use the following general fact (a connected open set in $\R^2$ will be called a \emph{domain}). We denote by $\difp$ the group of $C^{\infty}$ orientation preserving diffeomorphisms of the plane. 

\begin{lema}\label{pegandodifeos} Let\, $U$, $V$ and $W$ be bounded and convex domains in $\R^2$, whose boundaries are $C^{\infty}$-circles and such that$$\overline{U \cup V} \subset W.$$Then for any given $\phi_0\in\difp$ such that $\phi_0(U)=V$, there exists $\phi\in\difp$ satisfying$$\phi|_{U}=\phi_0|_{U}\quad\mbox{and}\quad\phi|_{\R^2\setminus\overline{W}}=\Id.$$Moreover, assume that there exists some straight line $\ell\subset\R^2$ such that $\phi_0(U\cap\mathbb{L}^{\pm})=V\cap\mathbb{L}^{\pm}$, where $\mathbb{L}^{+}$ and $\mathbb{L}^{-}$ are the two half-planes in $\R^2$ determined by $\ell$ (in particular, $\phi_0(\ell \cap U)=\ell \cap V$). Then $\phi$ can also be chosen to preserve $\mathbb{L}^{+}$ and $\mathbb{L}^{-}$ (and then $\phi(\ell)=\ell$).
\end{lema}

\begin{proof}[Proof of Lemma \ref{pegandodifeos}] Fix some convex domain $W_0$ such that $$\overline{U \cup V} \subset W_0 \subset \overline{W_0} \subset W,$$and mark some point $p \in \ell \cap U$. We will consider an orien\-tation-reversing $C^{\infty}$-diffeomorphism $\psi_{U}:W_0\setminus\overline{U} \to U\setminus\{p\}$ that maps $\ell\cap\big(W_0\setminus\overline{U}\big)$ onto $\ell \cap \big(U\setminus\{p\}\big)$ and that can be continuously extended to $\partial U$ satisfying $\psi_{U}|_{\partial U}=\Id$. One possible way to construct such a diffeomorphism is to consider the foliation of $W_0\setminus\{p\}$ given by straight rays $\ell_w$ joining $p$ with each point $w \in \partial W_0$. Since, by hypothesis, each of these lines crosses $\partial U$ only once (say, at $u_w$), we can consider an orientation-reversing one-dimensional diffeomorphism between the connected component of $\ell_w\setminus\{u_w\}$ outside $U$ and the one inside. Each of these diffeomorphisms can be chosen smooth in $w$ (because both $\partial U$ and $\partial W_0$ are $C^{\infty}$ circles) and then they jointly produce\footnote{Just as an example, fix some $R>1$ and let $p=(0,0) \in U=B(0,1) \subset W_0=B(0,R)$. Consider first the real function $g:W_0\setminus\overline{U}\to(0,1)$ given by$$g(x,y)=\frac{1}{R-1}\,\frac{1}{\sqrt{x^2+y^2}}\,\big(R-\sqrt{x^2+y^2}\big)\,,$$and then let $\psi_{U}:W_0\setminus\overline{U} \to U\setminus\{p\}$ be given by\, $\psi_{U}(x,y)=g(x,y)\,\big(x,y\big)$\,.} the desired bi-dimensional diffeomorphism $\psi_{U}:W_0\setminus\overline{U} \to U\setminus\{p\}$. In the same way, let $q=\phi_0(p) \in \ell \cap V$ and consider an orientation-reversing $C^{\infty}$-diffeomorphism $\psi_{V}:W_0\setminus\overline{V} \to V\setminus\{q\}$, mapping $\ell\cap\big(W_0\setminus\overline{V}\big)$ onto $\ell \cap \big(V\setminus\{q\}\big)$, which can be continuously extended to $\partial V$ as the identity. Now let $\phi:W_0 \to W_0$ be the homeomorphism given by
\[\phi = \left\{\begin{array}{ll}
\phi_0 & \quad \mbox{in $\overline{U}$}\\
\psi_{V}^{-1}\circ\phi_0\circ\psi_{U} & \quad \mbox{in $W_0\setminus\overline{U}$}
\end{array}\right.\]
Note that $\phi$ is an orientation-preserving $C^{\infty}$ diffeomorphism between $U$ and $V$, equal to $\phi_0$, and it is also an orientation-preserving $C^{\infty}$ diffeomorphism between $W_0\setminus\overline{U}$ and $W_0\setminus\overline{V}$. Moreover, $\phi(W_0\cap\mathbb{L}^{\pm})=W_0\cap\mathbb{L}^{\pm}$ and in particular $\phi(\ell \cap W_0)=\ell \cap W_0$. Therefore $\phi$ is almost what we want in $W_0$, but not quite because it may not be smooth at the boundary of $U$. To correct this flaw, we can use a standard \emph{isotopy extension} as in \cite[Section 8.1]{hirsch} (more precisely, see Theorem 1.9 in page 182, the \emph{smoothing theorem}, and the remark right after it). The resulting diffeomorphism, that we still denote by $\phi$, can be chosen so as to coincide with $\phi_0$ in $U$ (however, this perturbation could breakdown the fact that $\phi$ preserves $\ell$, so we may need another perturbation, post-composing $\phi$ with a suitable diffeomorphism of $W_0$ supported in a neighbourhood of $\partial V$ in the closure of $W_0 \setminus V$, to recover this property). Finally, since $\phi(W_0)=W_0$ and $\overline{W_0} \subset W$, $\phi$ can easily be extended to the whole plane satisfying $\phi|_{\R^2\setminus\overline{W}}=\Id$. Indeed, note first that $\phi$ extends to an orientation-preserving diffeomorphism $\eta$ of $\partial W_0$ to itself, that must be isotopic to $\Id|_{\partial W_0}$. Moreover, since $\phi(W_0\cap\mathbb{L}^{\pm})=W_0\cap\mathbb{L}^{\pm}$, the extension $\eta$ fixes both points of $\ell\cap\partial W_0$. Then we use the isotopy between $\eta$ and $\Id$ to deform $\phi|_{\partial W_0}$ to the identity in $\partial W$ along some smooth foliation by smooth circles in the annulus $W \setminus \overline{W_0}$. This can be done in such a way that $\phi$ preserves both segments of $\ell\cap\big(W \setminus \overline{W_0}\big)$ (because their corresponding boundary points are already fixed).
\end{proof}

\begin{remark}\label{obsmob} The statement of Lemma \ref{pegandodifeos} is quite flexible. For instance, the straight line $\ell$ could be an arc of a circle as well. Even the convexity assumption is not strictly needed. Indeed, to prove Lemma \ref{pegandodifeos} we just need some point $p \in \ell \cap U$ such that the two components of $\big(\ell\setminus\{p\}\big) \cap W_0$ belong to a smooth foliation of $W_0\setminus\{p\}$ of smooth curves joining $p$ with each point of $\partial W_0$, such that each leaf of this foliation intersects $\partial U$ at a single point (and the analogous property for the point $\phi_0(p) \in \ell \cap V$). This remark will be useful in the proof of Proposition \ref{buildTowers} below, where $\ell$ will be given by an orbit of the rotational vector field $Y$ constructed in the previous section.
\end{remark}

Recall the right extended stable boxes $\tilde{S}_n=\tilde{I}\times\sigma^{-n}\tilde{I}$, where $\tilde{I}=[\tilde{a},\tilde{b}]\subset (1,\sigma)$. Let $F_t$ be the time-$t$ map of the flow associated to the vector field $F$ constructed in the previous section. Notice that $F_t$ is a smooth diffeomorphism of the plane, isotopic to the identity, and that $F_1$ is linear on $\mathcal V$, equal to $\operatorname{Diag}(\sigma^{-2},\sigma)$, that is: $F_1(x,y)=X_1(x,y)=(\sigma^{-2}x,\sigma y)$. In particular, $O=(0,0)$ is a saddle fixed point of $F_1$. Denote by $W^{s}(O)$ and $W^{u}(O)$ its corresponding stable and unstable manifolds. By Lemma~\ref{l.loop}, $F_1$ presents a homoclinic loop associated to them. Note that $F_{n}(\tilde{S}_n)=\tilde{U}_{2n}$ for all $n\in\nt$, where $\tilde{U}_m=\sigma^{-m}\tilde{I}\times\tilde{I}$ are the unstable boxes.

\begin{proof}[Proof of Proposition~\ref{buildTowers}]
Let $n_1=n_1(\varepsilon,T)\in\nt$ be large enough so that the open rectangle $(1,\sigma^2)\times(0,\sigma^{-n_1}\tilde{b})$ is contained in $\mathcal{V}$, and consider $\widehat{W}_1=(1,\sigma^2)\times(-\sigma^{-n_1}\tilde{b},\sigma^{-n_1}\tilde{b})$. We choose $n_2 \geq n_1$ in $\nt$ and $\beta\in(0,\pi/8)$ so that
$$F_{k_0}\big([-\sigma^{-n_2}\tilde{b},\sigma^{-n_2}\tilde{b}]\times\tilde{I}\big)\subset\widehat{W}_1$$
for some $k_0\in\nt$. Notice that $k_0$ only depends on the angle of rotation of the (time one map of the) vector field $Y$, \emph{i.e.}, $k_0=k_0(\beta)$. Let $\widehat{W}_0$ be an open rectangle, compactly contained in $\widehat{W}_1$, satisfying
$$F_{k_0}\Big([-\sigma^{-n_2}\tilde{b},\sigma^{-n_2}\tilde{b}]\times\tilde{I}\Big)\cup \left(\tilde{I}\times[-\sigma^{-n_2}\tilde{b},\sigma^{-n_2}\tilde{b}]\right) \subset\widehat{W}_0\,.$$Since $\beta<\pi/8$, there exists $j_0\le k_0$ such that, denoting $W_i=F_{-j_0}(\widehat{W}_i)$ for $i\in\{0,1\}$, one has $\overline{W_1}\subset\mathcal{U}$, where $\mathcal{U}$ is the \emph{rotational region} from Section \ref{secregionsF} (see Figure \ref{f.figW}).
Moreover, denoting $i_0 = k_0-j_0$, we obtain the following two properties for all $n \geq n_2$:
\begin{itemize}
\item $\left(F_{i}(\tilde{U}_n)\cup F_i\big(s_h(\tilde{U}_n)\big)\right)\cap\overline{W_1}=\emptyset$ for $i\in\nt\cap [0,i_0)$, and $F_{i_0}\big(\tilde{U}_n \cup s_h(\tilde{U}_n)\big)\subset W_0$;
\item $\left(F_{-j}(\tilde{S}_n)\cup F_i\big(s_v(\tilde{S}_n)\big)\right)\cap\overline{W_1}=\emptyset$ for $j\in\nt\cap [0,j_0)$ and $F_{-j_0}\big(\tilde{S}_n \cup s_v(\tilde S_n)) \subset W_0$.
\end{itemize}

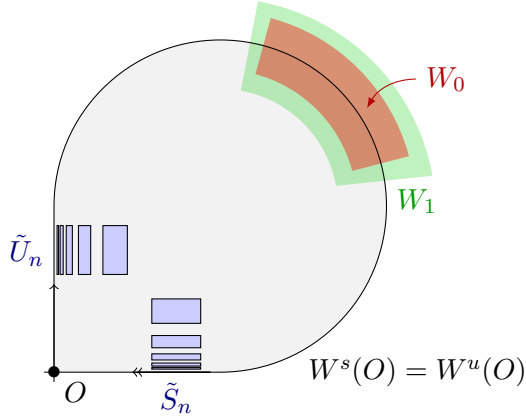
\begin{figure}
\begin{center}
\begin{tikzpicture}[scale=1.3]
\fill[color=gray, opacity=.1] (0,0) -- (1.7,0) arc (-90:180:1.7) -- cycle;
\fill[color=green!80!black, opacity=.25] (3.87,2.01) arc (10:80:2.2) -- (1.91,2.88) arc (80:10:1.2) -- cycle;
\fill[color=red, opacity=.4] (3.63,2.21) arc (15:75:2) -- (2.06,3.05) arc (75:15:1.4) -- cycle;
\draw[color=green!60!black] (3.4, 1.75) node[right]{$W_1$};
\draw [->,>=latex, color=red!70!black] (3.7, 3) node[right]{$W_0$}to[bend right] (3.2,2.7);
\draw (-.1,0) -- (1.7,0) arc (-90:180:1.7) -- (0,-.1);
\draw[->] (1.6,0) -- (0.8,0);
\draw[->] (1.6,0) -- (0.85,0);
\draw[->] (0,0) -- (0,.9);
\draw (2.5,0) node[right]{$W^s(O) = W^u(O)$};
\draw (0,0) node{$\bullet$} node[below right]{$O$};

\foreach \i in {1,...,5}				
	{\draw[fill=blue!20!white] (0.5^\i,1) rectangle (0.5^\i*1.5,1*1.5);
	\draw[fill=blue!20!white] (1,0.5^\i) rectangle (1*1.5,0.5^\i*1.5);}
	
\draw[color=blue!50!black] (1.25,0) node[below]{$\tilde S_n$};
\draw[color=blue!50!black] (0,1.25) node[left]{$\tilde U_n$};
\end{tikzpicture}
\caption{\label{f.figW} The region $W_1$, where we perturb the time one map of the flow associated to $F$, in order to obtain the diffeomorphism $f_0$ of Proposition \ref{buildTowers}.}
\end{center}
\end{figure}

Denote by $\mathcal{R}:\R^2\to\R^2$ the rotation of angle $\pi/2$ around the point $\big((\tilde{a}+\tilde{b})/2,(\tilde{a}+\tilde{b})/2\big)$ in the positive sense (counter-clockwise). In other words, $\mathcal{R}(x,y)=(\tilde{a}+\tilde{b}-y\,,x)$. Now consider the orientation preserving $C^{\infty}$ diffeomorphism of the plane defined by
\[
\phi_0 = F_{-j_{0}}\circ\mathcal{R}\circ F_{-i_{0}}\,.
\]
Then, it holds
\begin{equation}
\label{e.uvainos}
F_{j_0} \circ \phi_0 \circ F_{i_0}(x,y)=\mathcal{R}(x,y),\:\:\:\textrm{for all}\:\:\:(x,y)\in [-\sigma^{-n_2}\tilde{b},\sigma^{-n_2}\tilde{b}]\times\tilde{I}.
\end{equation}
In particular, 
\[
F_{j_0} \circ \phi_0 \circ F_{i_0}\big(\tilde{U}_n\big)=\tilde{S}_n,\:\:\:\textrm{for all}\:\:\:n\geq n_2.
\]
and 
\[
F_{j_0} \circ \phi_0 \circ F_{i_0}\big(s_h(\tilde{U}_n)\big)=s_v(\tilde{S}_n),\:\:\:\textrm{for all}\:\:\:n\geq n_2.
\]
Let $U$ be a convex domain, with smooth boundary, containing $F_{i_0}\big([-\sigma^{-n_2}\tilde{b},\sigma^{-n_2}\tilde{b}]\times\tilde{I}\big)$ and compactly contained in $W_0$. Let $V=\phi_0(U)$, a domain containing $$\phi_0\big(F_{i_0}\big([\sigma^{-n_2}\tilde{b},\sigma^{-n_2}\tilde{b}]\times\tilde{I}\big)\big)=F_{-j_0}\big(\tilde{I}\times[\sigma^{-n_2}\tilde{b},\sigma^{-n_2}\tilde{b}]\big)$$ and compactly contained in $W_0$. By taking the corresponding convex hulls of $W_0$ and $W_1$ as well, we can apply Lemma \ref{pegandodifeos} (see also Remark \ref{obsmob}) in order to obtain an orientation preserving $C^{\infty}$ diffeomorphism $\phi:\R^2\to\R^2$ such that $\phi=\phi_0$ in $F_{i_0}\big([-\sigma^{-n_2}\tilde{b},\sigma^{-n_2}\tilde{b}]\times\tilde{I}\big)$ and $\phi=\Id$ in $\R^2\setminus\overline{W_1}$. Moreover, since $\phi_0\big(W^{u}(O) \cap U\big)=W^{u}(O) \cap V$, we can choose $\phi$ preserving $W^{u}(O)$. Finally, let $f_0\in\dif^{\infty}(\R^2)$ be given by $f_0=F_1\circ\phi$, and note the following properties:
\begin{itemize}
\item $f_0=F_1$ in $\R^2 \setminus W_1$. In particular $f_0$ is linear equal to $\operatorname{Diag}(\sigma^{-2},\sigma)$ in the set $\mathcal V$, and then the origin is a saddle-type hyperbolic fixed point for $f_0$.
\item The homoclinic loop constructed in Lemma \ref{l.loop} for $F_1$ is also a homoclinic loop for $f_0$.
\item If $(x,y)\in(-\sigma^{-2n_2},\sigma^{-2n_2})\times[\tilde a,\tilde b]$, then\[f_0^{k_0}(x,y)=R(x,y),\]where $R$ is the rotation of angle $\pi/2$ and centre $(\frac{a+b}{2},\frac{a+b}{2})$, in the positive (counter-clockwise) sense. If $(x,y)\in(-\sigma^{-2n_2},\sigma^{-2n_2})\times[-\tilde{a},-\tilde{b}]$ then $f_0^{k_0}(x,y)=s_hR\big(s_v(x,y)\big)$.
\item $f_0^{k_0}\big(\{0\}\times\tilde{I}\big)=\tilde{I}\times\{0\}$ and $f_0^{k_0}(\tilde{U}_n)=\tilde{S}_n$ for every $n \geq n_2$.
\item $f_0^{k_0}\left(s_h(\tilde{U}_n)\right)=s_v(\tilde{S}_n)$, for every $n\geq n_2$.
\item If $(x,y)\in\tilde{S}_n$, then$$f_0^{n+k_0}(x,y)=\big(a+b-\sigma^ny,\sigma^{-2n}x\big)\in\tilde{S}_{2n}.$$
\end{itemize}
It only remains to establish properties \eqref{p5} and \eqref{p6}.

Let us prove now property \eqref{p5}. We consider $c_0\eqdef\|f^{k_0}_0\|_{C^1}$ (remark that the integer $k_0$ does not depend upon $n_2$). Fix arbitrarily $n_0\geq n_2$ and let us analyse the boundary of the set $f_0^{k_0}\big([0,\sigma^{-2n_0}]\times[\sigma^{-1}\tilde{b},\tilde{a}]\big)$. 

By property \eqref{p3} we deduce directly that 
\[
f_0^{k_0}\left([0,\sigma^{-2n_0}]\times\{\tilde{a}\}\right)=\{\tilde{b}\}\times[0,\sigma^{-2n_0}].
\]
Moreover, as $f_0|_{\mathcal V}$ is the map $\operatorname{Diag}(\sigma^{-2},\sigma)$, we have that $f_0\big([0,\sigma^{-2n_0}]\times\{\sigma^{-1}\tilde{b}\}\big)=[0,\sigma^{-2n_0-2}]\times\{\tilde{b}\}$. Therefore, applying \eqref{p3} again and iterating backwards once we conclude that

\[
f_0^{k_0}\left([0,\sigma^{-2n_0}]\times\{\sigma^{-1}\tilde{b}\}\right)=\{\sigma^2\tilde{a}\}\times[0,\sigma^{-2n_0-3}].
\]

This implies that $f_0^{k_0}\big(\{0\}\times[\sigma^{-1}\tilde{b},\tilde{a}]\big)$ is a smooth curve joining $(\tilde{b},0)$ to $(\sigma^2\tilde{a},0)$. Since $\{0\}\times[\sigma^{-1}\tilde{b},\tilde{a}]\subset W^{u}_{f_0}(O)$ and $W^s_{f_0}(O)=W^u_{f_0}(O)$, we deduce that such smooth curve must be contained in $W^s_{f_0}(O)$, from which we conclude that  

\begin{equation}
\label{e.fundamentaldomainreturns}
f_0^{k_0}\left(\{0\}\times[\sigma^{-1}\tilde{b},\tilde{a}]\right)=[\tilde{b},\sigma^2\tilde{a}]\times\{0\}. 
\end{equation}

This shows that the boundary of $f_0^{k_0}\big([0,\sigma^{-2n_0}]\times[\sigma^{-1}\tilde{b},\tilde{a}]\big)$ has two vertical sides, respectively over the points with abscissa $\sigma^2\tilde{a}$ and $\tilde{b}$, and a horizontal side $[\tilde b,\sigma^2\tilde a]\times\{0\}$. The remainder of this boundary is a smooth curve joining the points $(\tilde{b},\sigma^{-2n_0})$ and $(\sigma^2\tilde{a},\sigma^{-2n_0-3})$, as indicated in Figure~\ref{FigAmeba}. 

Now, $n_0\geq n_2$ is such that 
\[
f_0^{k_0}\left([0,\sigma^{-2n_0}]\times[\sigma^{-1}\tilde{b},\tilde{a}]\right)\not\subset [\tilde{b},\sigma^2\tilde{a}]\times[0,\sigma^{-n_0-4}].
\] 
Notice that we proved above that the boundary of $f_0^{k_0}\big([0,\sigma^{-2n_0}]\times[\sigma^{-1}\tilde{b},\tilde{a}]\big)$ has an overlap with the boundary of $[\tilde{b},\sigma^2\tilde{a}]\times[0,\sigma^{-n_0-4}]$. In particular, the two sets are not disjoint.     
We claim that, for every $n_0\geq n_2$ the set $f_0^{k_0}\big([0,\sigma^{-2n_0}]\times[\sigma^{-1}\tilde{b},\tilde{a}]\big)$ crosses the boundary of $[\tilde{b},\sigma^2\tilde{a}]\times[0,\sigma^{-n_0-4}]$ at a point with $(x,y)$ with $y\geq\sigma^{-n_0-4}$. In other words, we claim that there exists $z=(\tilde{x},\tilde{y})\in(0,\sigma^{-2n_0})\times[\sigma^{-1}\tilde{b},\tilde{a}]$ such that $f_0^{k_0}(z)=(x,y)$ with $y=\sigma^{-n_0-4}$.   

\begin{figure}
\centering
\begin{tikzpicture}[scale=.7]

\fill [color=gray, opacity=.08] (0,0) rectangle (7,6.2);
\draw[->] (-.7,0)--(7,0);
\draw[->] (0,-.7)--(0,6.2);
\draw [color=red!60!black, thick] (4,1)--(4,2)--(6,2)--(6,.6);
\filldraw[fill=green, opacity=.25,draw=green!50!black, very thick] (0,4) rectangle (1,6);
\draw [thick, dashed, opacity=.6] (5,2) ellipse (20pt and 10pt); 
\draw [thick, dashed, opacity=.6] (3.9,1.6) ellipse (8pt and 10pt);
\draw [thick, dashed, opacity=.6] (6.1,1.2) ellipse (8pt and 10pt);
\draw (4,0) node[below]{$\tilde{b}$};
\draw (6,0) node[below]{$\sigma^2\tilde{a}$};
\draw (1,-0.1) -- (1,0.1);
\draw (1,-.089) node[below]{$\sigma^{-2n_0}$};
\draw [dotted] (1,.1)--(1,4);
\draw (-.1,1)--(.1,1);
\draw (-.089,1.1) node[left]{$\sigma^{-2n_0}$};
\draw [dotted] (.1,1)--(4,1);
\draw (-.1,.6)--(.1,.6);
\draw (-.089,.5) node[left]{$\sigma^{-2n_0-3}$};
\draw [dotted] (.1,.6)--(6,.6);
\draw (0,4) node[left]{$\sigma^{-1}\tilde{b}$};
\draw (0,6) node[left]{$\tilde{a}$};
\draw (-.1,2)--(.1,2);
\draw (-.089,2) node[left]{$\sigma^{-n_0-4}$};
\draw[dotted] (.1,2)--(4,2);
\filldraw[fill=green, opacity=.25,draw=green!50!black, very thick]  (6,.6)--(6,0)--(4,0)-- (4,1)to[out=0,in=0] (3.9,1.4) to[out=180,in=-180] (3.9,1.8) to[out=0,in=-90](4.5,2) to[out=90,in=-270] (5.5,2) to[out=-90,in=-180] (6.1,1.4) to[out=0, in=0] (6.1,1) to[out=180,in=-180] (6,.6);
\draw (3.7, 3) node[color=green!50!black,font=\footnotesize]{$f_0^{k_0}\big([0,\sigma^{-2n_0}]\times[\sigma^{-1}\tilde{b},\tilde{a}]\big)$};
\draw [->,>=latex, color=green!50!black] (3.8,2.7) to[bend right] (5,1.4);
\end{tikzpicture}
\caption{\label{FigAmeba}Proof of Property~\eqref{p5}: if it does not hold then the image of the green rectangle by $f_0^{k_0}$ has to cross at least one of the three red segments, and each case leads to a contradiction.}
\end{figure}
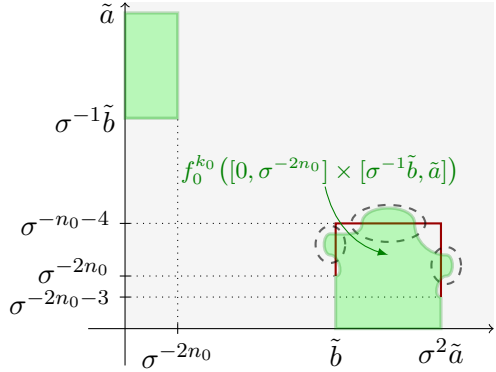

Indeed, by the conclusion about the boundary of $f_0^{k_0}\big([0,\sigma^{-2n_0}]\times[\sigma^{-1}\tilde{b},\tilde{a}]\big)$ that we saw above, the only possibility for this not to happen is if the set $f_0^{k_0}\big([0,\sigma^{-2n_0}]\times[\sigma^{-1}\tilde{b},\tilde{a}]\big)$ crosses one of the vertical segments $\{\sigma^2\tilde{a}\}\times(\sigma^{-2n_0-3},\sigma^{-n_0-4})$ or $\{\tilde{b}\}\times (\sigma^{-2n_0},\sigma^{-n_0-4})$ (see Figure~\ref{FigAmeba}). However, applying \eqref{p3} as we did above, we see that the images under $f_0^{-k_0}$ of these segments are (respectively)
$(\sigma^{-2n_0},\sigma^{-n_0-4})\times\{\sigma^{-1}\tilde{b}\}$ or $(\sigma^{-2n_0},\sigma^{-n_0-4})\times\{\tilde{a}\}$, which are both disjoint from $[0,\sigma^{-2n_0}]\times[\sigma^{-1}\tilde{b},\tilde{a}]$. 

Therefore $f_0^{k_0}\big([0,\sigma^{-2n_0}]\times[\sigma^{-1}\tilde{b},\tilde{a}]\big)$ cannot cross the vertical segments $\{\sigma^2\tilde{a}\}\times(\sigma^{-2n_0},\sigma^{-n_0-4})$ or $\{\tilde{b}\}\times(\sigma^{-2n_0},\sigma^{-n_0-4})$. But by our assumption on $n_0$ it must cross the boundary of $[\tilde{b},\sigma^2\tilde{a}]\times[0,\sigma^{-n_0-4}]$, so this should happen at the segment $\{\sigma^{-n_0-4}\}\times[\tilde{b},\sigma^2\tilde{a}]$, proving the claim. 

Thus, there must exist $z=(\tilde{x},\tilde{y})$, with $\tilde{x}\in(0,\sigma^{-2n_0})$ and $\tilde{y}\in [\sigma^{-1}\tilde{b},\tilde{a}]$ such that $f_0^{k_0}(z)=(x,y)$, with $y=\sigma^{-n_0-4}$. 
Let $z^{\prime}=(0,\tilde{y})$. Using \eqref{e.fundamentaldomainreturns}, we get $f_0^{k_0}(z^{\prime})\in[\tilde{b},\sigma^2\tilde{a}]\times\{0\}$. By the mean value inequality, one deduces then
\[
c_0\geq\frac{\sigma^{-n_0-4}}{\sigma^{-2n_0}}=\sigma^{n_0-4},
\]
for every $n_0\geq n_2$, which is absurd if $n_0$ is large enough. This proves that for every $n_0$ large, it holds 
\[
f_0^{k_0}\left([0,\sigma^{-2n_0}]\times[\sigma^{-1}\tilde{b},\tilde{a}]\right)\subset [\tilde{b},\sigma^2\tilde{a}]\times[0,\sigma^{-n_0-4}].
\] 
With a similar argument, one establishes that
\[
f_0^{k_0}\left([-\sigma^{-2n_0},0]\times[\sigma^{-1}\tilde{b},\tilde{a}]\right)\subset [\tilde{b},\sigma^2\tilde{a}]\times[-\sigma^{-n_0-4},0],
\] 
for every $n_0$ large enough. By the symmetry of $f_0$ with respect to $-Id$, the proof of property \eqref{p5} is completed.

This argument also proves that forward iterations under $f_0$ decreases the distance to the homoclinic loop $\mathcal{L}$ and thus we have perturbed the time-one map of the vector field $F$ while retaining the dynamical properties of items (e) and (f) in Lemma~\ref{l.lemaum}. This concludes the proof of Proposition~\ref{buildTowers}.  
\end{proof}

\subsection{Historic behaviour}

By a similar construction, one can connect the stable/un\-stable manifolds of two different dissipative hyperbolic saddle fixed points with the same eigenvalues, getting a similar configuration of heteroclinic connections together with a family of boxes (see Figure \ref{FigBowen}). More precisely, one connects the linear map $(x,y)\mapsto (\sigma^{-2}x,\sigma y)$ having the origin $O$ as a dissipative hyperbolic saddle fixed point, with the affine map $(x,y)\mapsto (\sigma^{-2}(x-2p),\sigma (y-2p))$ having the point $P = (2p,2p)$ as a dissipative hyperbolic saddle fixed point. The properties of the obtained diffeomorphism are summarized in the following lemma.

\begin{lema}\label{buildTowers2}
	For any $\sigma>1$, there exists $\hat f_0\in \Diff^\infty(\R^2)$\nomenclature{$\hat f_0$}{Alternative diffeomorphism for $f_0$, of Bowen-eye type} compactly supported, such that the origin $O\in\R^2$ and the point $P = (2p,2p)\in\R^2$ are hyperbolic fixed points for $\hat f_0$, with the following properties:
	\begin{enumerate}[(i)]
		\item $\hat f_0$ is linear in a neighbourhood $V$ of $O$, equal to the map $(x,y)\mapsto (\sigma^{-2}x,\sigma y)$, and linear in a neighbourhood $V'$ of $P$, equal to the map $(x,y)\mapsto (\sigma^{-2}(x-2p),\sigma (y-2p))$;
		\item  $O$ and $P$ are heteroclinically related, meaning that $W^s(O) = W^u(P)$, and $W^s(P) = W^u(O)$;
		\item  Denoting $s_{(p,p)}$ the symmetry with respect to the point $(p,p)$, there are integers $n_0,k_0\in\nt$ such that, for every $n\geq n_0$,
\[\hat f_0^{k_0}(\tilde{U}_n) = s_{(p,p)}(\tilde{S}_n)\:\:\textrm{and thus}\:\: \hat f^{n+k_0}(\tilde{S}_n) = s_{(p,p)}(\tilde{S}_{2n})\]
and
\[\hat f_0^{k_0}\big(s_{(p,p)}(\tilde{U}_n)\big) = \tilde{S}_n\:\:\textrm{and thus}\:\: \hat f^{n+k_0}\big(s_{(p,p)}(\tilde{S}_n)\big) = \tilde{S}_{2n};\]
		\item For $n\geq n_0$, there is a first return map $\hat g \eqdef f_{2}^{k_0+n}:\tilde{S}_n \to s_{(p,p)}(\tilde{S}_{2n})$ which
		satisfies $$\hat g(x,y)=s_{(p,p)}\big(a+b-\sigma^ny,\sigma^{-2n}x\big).$$
		In particular, $s_{(p,p)}\circ L_{2n}\circ \hat g\circ(L_{n})^{-1}$ is a rotation by $\pi/2$ (and the same for orbits from $s_{(p,p)}(\tilde{S}_n)$ to $\tilde{S}_{2n}$;)
		\item $\hat f_0$ is symmetric with respect to $O$, \emph{i.e.}, it commutes with $-\operatorname{Id}$.
\end{enumerate}
\end{lema}

\begin{figure}
	\begin{center}
		\begin{tikzpicture}[scale=1.5]
		\fill[color=gray, opacity=.1] (0,0) -- (1.5,0) arc (-90:0:1.5) -- (3,3) -- (1.5,3) arc (90:180:1.5) -- cycle;
		\draw (-.1,0) -- (1.5,0) arc (-90:0:1.5) -- (3,3.1);	
		\draw (3.1,3) -- (1.5,3) arc (90:180:1.5) -- (0,-.1);
		\draw[->] (1.6,0) -- (1,0);
		\draw[->] (1.6,0) -- (1.05,0);
		\draw[->] (0,0) -- (0,1);
		\draw[->] (1.5,3) -- (2,3);
		\draw[->] (1.5,3) -- (2.05,3);
		\draw[->] (3,3) -- (3,2);
		\draw (3,1) node[right]{$W^s(O) = W^u(P)$};
		\draw (0,2) node[left]{$W^u(O) = W^s(P)$};
		\draw (0,0) node{$\bullet$} node[above left]{$O$};
		\draw (3,3) node{$\bullet$} node[below right]{$P$};
		
		\foreach \i in {1,...,5}
		{\foreach \j in {1,...,5}
			{\draw[fill=blue!20!white] (0.5^\i,0.5^\j) rectangle (0.5^\i*1.5,0.5^\j*1.5);
			\draw[fill=blue!20!white] (3-0.5^\i,3-0.5^\j) rectangle (3-0.5^\i*1.5,3-0.5^\j*1.5);}}
		
		\draw (1.5,1.5) node{$\bullet$};
		\draw[<-] (1.7,1.7) arc (45:225:.29);
		\draw (1.3,1.7) node[left]{$s_{(p,p)}$};
		
		\end{tikzpicture}
		\caption{The alternative diffeomorphism $\hat f_0$, of Bowen-eye type}\label{FigBowen}
	\end{center}
\end{figure}
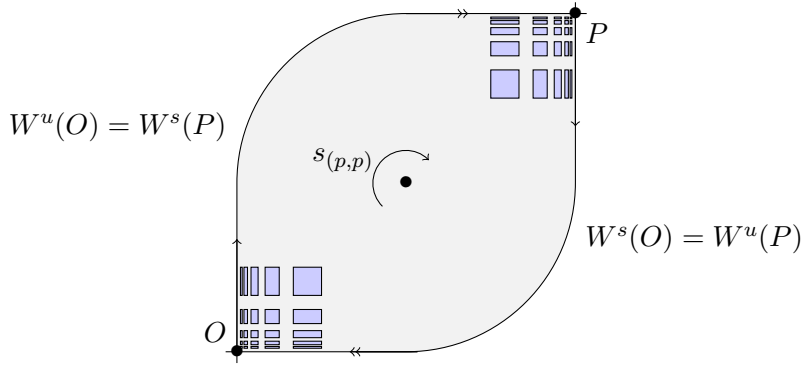

The configuration of a heteroclinic connection with two dissipative saddle fixed points is known as \emph{Bowen's eye attractor}. In this case, any point starting sufficiently close to the union of stable/unstable manifolds inside the heteroclinic loop has historic behaviour\footnote{A point $x\in M$ is said to have \emph{historic behaviour} if the sequence $\frac{1}{n}\sum_{j=0}^{n-1}\delta_{f^{j}(x)}$ does not converge in the weak-* topology.} (this can be seen by reasoning as in Lemma \ref{suf}). See the grounding papers \cite{MR1182135} and \cite{MR1274765} for a complete proof. Since then, the Bowen's eye attractor has been quite widely studied (see \cite{MR2399948}, \cite{MR2237469}, \cite{MR2947933}, \cite{MR1901073} and the references therein).

\section{Orbit exclusion I: determining the basin}\label{sec.orbitexcum}

In this section we shall perform the first step of our \emph{orbit exclusion procedure}. It consists in introducing a modification $f_1\in\difp$ of $f_0$, where $f_0\in\difp$ is the diffeomorphism given by Proposition~\ref{buildTowers}, so that $f_1=h_1\circ h_2\circ f_0$ and the composition map $h_1\circ h_2$ pushes points away from the stable manifold of $O$ while carefully preserving the tower of wandering boxes $S_n$. This will create a periodic trapping region, but the resulting map $f_1$ will still have a Dirac physical measure at $O$, so that we shall have a finer description of the statistical basin: it will be composed of the union of a ``wandering tower" (the stable tower $\cS$ defined in Definition~\ref{def.towers}), the stable manifold of $O$, plus some orbits inside a nowhere dense zero Lebesgue measure set (see Proposition \ref{p.bluetrichotomy}). The map $f_1$ will present much more complicated dynamics: it will have positive topological entropy and infinitely many periodic orbits (see Section \ref{SecGeom}). This section and the next one are certainly the most delicate and technical of the present paper.

\subsection{Description of the perturbations}\label{SecDesc}
We shall obtain the perturbation $f_1$ of $f_0$ from the composition of two diffeomorphisms: $h_1$, which pushes points horizontally towards an attractive fixed vertical segment, and $h_2$ which pushes points vertically towards an attractive fixed horizontal segment\footnote{It may be that the diffeomorphism $h_1$ is useless to get the desired result; however it simplifies significantly the proofs.}. 

\subsubsection{Notations from Proposition~\ref{buildTowers}}
Recall that we have fixed $\sigma>1$ and numbers $1<\tilde{a}<a<b<\tilde{b}$, so that the intervals $I=[a,b]\subset\tilde{I}=[\tilde{a},\tilde{b}]$ have the same center. The diffeomorphism $f_0$ is given by Proposition~\ref{buildTowers}.

\subsubsection{Choice of parameters}\label{subsub.escolhasparaavida}
We choose $\eps_1>0$ small enough so that $[a-5\eps_1,b+5\eps_1] \subset \tilde I$, and also
\begin{equation}\label{e.oitoepsilon}
\eps_1<\min\left(\frac{b-a}{8},\ \sigma^{-n_0}\frac{a-\sigma^{-1}b}{10},\,\frac{(\sigma-\tilde{b})(b-a)}{4}.\right).
\end{equation}
We also fix $\delta_1\eqdef\eps_1/10$. The functions we are going to define will depend on these parameters.

\subsubsection{Bump function notation}
As we shall need different bump functions, it will be more convenient to have a convention for them. Therefore, in the sequel, given $\ell_0<\ell_1<r_1<r_0$ we denote by $\varphi_{\ell_0,r_0}^{\ell_1,r_1} : \R\to \R$\nomenclature{$\varphi_{\ell_0,r_0}^{\ell_1,r_1}$}{Bump function} a $C^\infty$ bump function satisfying:
\begin{enumerate}
	\item $\varphi_{\ell_0,r_0}^{\ell_1,r_1}(x)\in[0,1]$ for all $x\in\R$;
	\item $\varphi_{\ell_0,r_0}^{\ell_1,r_1}(x)=1$ for all $x\in [\ell_1,r_1]$;
	\item $\varphi_{\ell_0,r_0}^{\ell_1,r_1}(x)=0$ for all $x\in \R\setminus[\ell_0,r_0]$;
	\item $\varphi_{\ell_0,r_0}^{\ell_1,r_1}$ is increasing between $\ell_0$ and $\ell_1$, and decreasing between $r_0$ and $r_1$.
\end{enumerate}

\subsubsection{More boxes} 
\label{subsub.moreboxes}

We define also the \emph{$\eps_1$-boxes}\nomenclature{$C_n$}{$\eps_1$-boxes $C_n=[a-3\eps_1,b+3\eps_1]\times\sigma^{-n}[a-3\eps_1,b+3\eps_1]$, so that $S_n\subset C_n \subset \tilde S_n$}
$$C_n\eqdef[a-3\eps_1,b+3\eps_1]\times\sigma^{-n}[a-3\eps_1,b+3\eps_1],$$
and the \emph{$\eps_1$-tower} $\mathcal{C}\eqdef\bigcup_{n\geq n_0}\bigcup_{k\in\Z}f_0^k(C_n)$.\nomenclature{$\mathcal{C}$}{$\eps_1$-tower $\mathcal{C}=\bigcup_{n\ge n_0} \bigcup_{k\in\mathbb{Z}} f_0^k(C_n)$}

\subsubsection{The horizontal push}
Our horizontal perturbation is ruled by a diffeomorphism of the line which is the identity outside the interval $[b,\sigma^2 a]$ and has a unique attractive fixed point inside this interval. For the estimations that we are going to make, we shall need some other specific properties. 

Indeed, we consider any diffeomorphism $\xi_1\in\Diff^\infty(\R)$ satisfying
\begin{enumerate}
\item[(H1)] $\xi_1(x)=x$ for every $x\in\R\setminus[b,\sigma^2 a]$.
\item[(H2)] There exists $q_1\in(\sigma,\sigma^2)$ such that if $x\in(b,\sigma^2 a)$ satisfies $\xi_2(x)=x$ then $x=q_1$. Moreover, $0<\xi_1^{\prime}(q_1)<1$.
\item[(H3)]\label{H3} $\xi_1(x)=x-\sigma^2\eps_1$ if $x\in\sigma^2[a-3\eps_1,a-\eps_1]$ and $\xi(x)=x+\eps_1$ if $x\in[b+\eps_1,b+3\eps_1]$. In particular, $\xi_1(\sigma^2(a-\eps_1))= \sigma^2(a-2\eps_1)$ and $\xi_1(b+\eps_1)=b+2\eps_1$.
\end{enumerate}

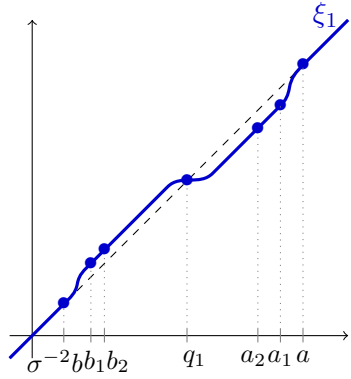
\begin{figure}
\begin{center}
\begin{tikzpicture}[scale=.6]
\draw[->] (-.5,0) -- (7,0);
\draw[->] (0,-.5) -- (0,7);
\draw[dashed] (-.5,-.5) -- (7,7);
\draw [very thick,blue!80!black] (-.5,-.5) -- (.7,.7) to [out=45,in=260] (1,1.15) to [out=80,in=225] (1.3,1.6) -- (3,3.3) to [out=45, in=225] (4,3.6) -- (5.5,5.1);
\draw[very thick, blue!80!black] (7,7) -- (6,6) to [out=225, in=80] (5.7,5.5) to [out=260, in=45] (5.5,5.1);
\draw (6,6) node[blue!80!black]{$\bullet$};
\draw[black!50!white] (6,-0.15)--(6,0.15);
\draw (6,-0.1) node[below]{\small $a$};
\draw[black!50!white,dotted] (6,0.15) -- (6,6);
\draw (5.5,5.1) node[blue!80!black]{$\bullet$};
\draw[black!50!white] (5.5,-0.15)--(5.5,0.15);
\draw (5.5,-0.1) node[below]{\small $a_1$};
\draw[black!50!white,dotted] (5.5,0) -- (5.5,5.1);
\draw[black!50!white] (5,-0.15)--(5,0.15);
\draw (4.9,-0.1) node[below]{\small $a_2$};
\draw[black!50!white,dotted] (5,0.15) -- (5,4.6);
\draw (5,4.6) node[blue!80!black]{$\bullet$};
\draw[black!50!white] (3.43,-0.15) -- (3.43,0.15);
\draw (3.6,-0.1) node[below]{\small $q_1$};
\draw[black!50!white,dotted] (3.43,0.15) -- (3.43,3.43);
\draw (3.43,3.43) node[blue!80!black]{$\bullet$};
\draw (.7,.7) node[blue!80!black]{$\bullet$};
\draw[black!50!white] (.7,-0.15) -- (.7,0.15);
\draw[black!50!white,dotted] (.7,0.15) -- (.7,.7);
\draw (.5,-0.1) node[below]{\small $\sigma^{-2}b$};
\draw[black!50!white] (1.3,-0.15) -- (1.3,0.15);
\draw (1.4,-0.1) node[below]{\small $b_1$};
\draw[black!50!white,dotted] (1.3,0.15) -- (1.3,1.6);
\draw (1.3,1.6) node[blue!80!black]{$\bullet$};
\draw[black!50!white] (1.6,-0.15) -- (1.6,0.15);
\draw (1.9,-0.1) node[below]{\small $b_2$};
\draw[black!50!white,dotted] (1.6,0.15) -- (1.6,1.9);
\draw (1.6,1.9) node[blue!80!black]{$\bullet$};
\draw (6.5,6.6) node[above,color=blue!80!black]{$\xi_1$};
\end{tikzpicture}
\caption{\label{f.grafxi1}The diffeomorphism $\xi_1$ of the line. In this drawing, $b_1=b+\eps_1$, $b_2=b+\eps_13$, $a_2=\sigma^{2}(a-3\eps_1)$ and $a_1=\sigma^{2}(a-\eps_1)$. (H3) says that $\xi_1$ is a translation when restricted to each interval $[b_1,b_2]$ and $[a_2,a_1]$.}
\end{center}
\end{figure}

Since it is rather elementary, we refrain from giving a explicit construction of a diffeomorphism satisfying the above properties (see Figure~\ref{f.grafxi1}).
Property (H3) implies that $\xi_1$ is a translation on the intervals $[a-\eps_1-\delta_1,a-\eps_1]$ and $\sigma^{-2}[b+\eps_1,b+\eps_1+\delta_1]$. This will simplify some calculations in the proof of our inclination lemma (Lemma~\ref{LemInclination}). 

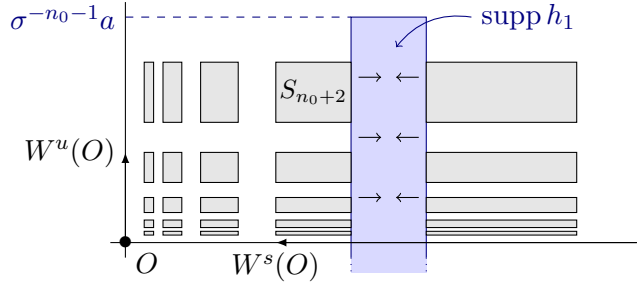
\begin{figure}
	\begin{center}
		\begin{tikzpicture}[scale=2]
		\fill[color=blue, opacity=.15] (1.5,-.2) rectangle (2,1.5);
		\draw[color=blue!50!black] (1.5,-.1) -- (1.5,1.5) -- (2,1.5) -- (2,-.1);
		\draw[color=blue!50!black, dotted] (1.5,-.1) -- (1.5,-.2);
		\draw[color=blue!50!black, dotted] (2,-.1) -- (2,-.2);
		\draw [->, color=blue!50!black] (2.3,1.5) node[right]{$\operatorname{supp} h_1$} to[bend right] (1.8,1.3);
		
		\draw[color=blue!50!black, dashed] (0,1.5) -- (1.5,1.5);
		\draw[color=blue!50!black] (0,1.5) node[left]{$\sigma^{-n_0-1}a$};
		\draw[->] (1.55,.3) -- (1.7,.3);
		\draw[->] (1.55,.7) -- (1.7,.7);
		\draw[->] (1.55,1.1) -- (1.7,1.1);
		\draw[->] (1.95,.3) -- (1.8,.3);
		\draw[->] (1.95,.7) -- (1.8,.7);
		\draw[->] (1.95,1.1) -- (1.8,1.1);
		\draw (-.1,0) -- (3.4,0);
		\draw (0,-.1) -- (0,1.6);
		\draw[->,>=latex] (3.4,0) -- (1,0) node[below]{$W^s(O)$};
		\draw[->,>=latex] (0,0) -- (0,.6) node[left]{$W^u(O)$};
		\draw (0,0) node{$\bullet$} node[below right]{$O$};
		
		\foreach \i in {0,...,4}
		{\foreach \j in {0,...,4}
			{\draw[fill=black!10!white] (2*0.5^\i,0.8*0.5^\j) rectangle (2*0.5^\i*1.5,0.8*0.5^\j*1.5);}}
		\draw  (1.25,1) node{\small $S_{n_0+2}$};
		
		\end{tikzpicture}
		\caption{\label{FigDefh1}The diffeomorphism $h_1$.}
	\end{center}
\end{figure}

With the map $\xi_1$ at hand we can now define our horizontal push. Consider $\varphi_1 = \varphi_{-\sigma^{-n_0-1}a,\sigma^{-n_0-1}a}^{-\sigma^{-n_0-2}(b+3\eps_1),\sigma^{-n_0-2}(b+3\eps_1)}$, so that $\supp \varphi_1 = [{-\sigma^{-n_0-1}a},\sigma^{-n_0-1}a]$ and
\[\varphi_1|_{[-\sigma^{-n_0-2}(b+3\eps_1),\sigma^{-n_0-2}(b+3\eps_1)]}\equiv 1.\]
Our horizontal perturbation is the map (see Figure~\ref{FigDefh1})
\[\begin{array}{rrcl}
h_1 : & \R^2 & \longrightarrow & \R^2\\
 & (x,y) & \longmapsto & \big( \varphi_1(y) \xi_1(x) + (1-\varphi_1(y))x \,,\ y\big).
\end{array}\]

Notice that, as $[\min\xi_1^{\prime},\max\xi_1^{\prime}]\subset (0,\infty)$ and since, for every $y\in\R$, the first coordinate of $h_1(x,y)$ is a convex combination between $1$ and $\xi_1(x)$, we can see that $h_1$ is injective and $Dh_1(x,y)$ is invertible for every $(x,y)\in\R^2$. Therefore, $h_1\in\dif^{\infty}(\R^2)$.  

Moreover, we have that {$\supp h_1 = [b,\sigma^2 a]\times [-\sigma^{-n_0-1}a,\sigma^{-n_0-1}a]$}, and if $(x,y)\in [b,\sigma^2 a] \times [-\sigma^{-n_0}(b+3\eps_1),\sigma^{-n_0}(b+3\eps_1)]$, then $h_1(x,y) = (\xi_1(x),y)$. Thus, $h_1$ fixes a vertical segment of points with abscissa $q_1$, and all the points with the same vertical coordinate as one of these fixed points is attracted towards it. See Figure~\ref{FigDefh1}.

\subsubsection{The vertical push}

As with the horizontal push, the core of our vertical perturbation is a special diffeomorphism of the line.
Indeed, we take any $\xi_2\in\dif^{\infty}(\R)$\nomenclature{$\xi_2$}{Diffeomorphism of $\R$ used to define $\varphi_2$; it pushes points up} such that all the following properties hold true (see Figure~\ref{f.grafxi2}):
\begin{itemize}
\item[(V1)] $\xi_2(y)=y$ if $y\notin[-\sigma^{-n_0-1}a,\sigma^{-n_0-1}a]$.
\item[(V2)] $\xi_2\left([-\sigma^{-n_0-1}(a-\eps_1),\sigma^{-n_0-1}(a-\eps_1)]\right)\subset\left(\sigma^{-n_0-2}(b+4\eps_1),\sigma^{-n_0-1}(a-3\eps_1)\right)$.
\item[(V3)] There exists $0<\beta<1$ and $c\in\R$ such that $\xi_2(y)=\beta y+c$ for any $y\in[-\sigma^{-n_0-1}(a-\eps_1),\sigma^{-n_0-1}(a-\eps_1)]$,
such that $\xi_2$ has $q_2 = \sigma^{-n_0-1}(a+\sigma^{-1}b)/2$ as a fixed point.
\end{itemize} 

\begin{figure}
\begin{minipage}[c]{.45\linewidth}
	\begin{center}
		\begin{tikzpicture}[scale=.35]
		\draw[->] (-7,0)--(9,0);
		\draw[->](0,-7)--(0,9);
		\draw[dashed, opacity=.6] (-7,-7)--(9,9);
		\draw[very thick, red!80!black] (-7,-7)--(-6,-6);
		\draw[very thick, red!80!black] (8,8)--(9,9);
		\draw[very thick, red!80!black] (-5,3)--(6,5.2);
		\draw[very thick, red!80!black] (-6,-6) .. controls (-5,-5) and (-6,2.8).. (-5,3);
		\draw[very thick, red!80!black] (6,5.2) .. controls (7,5.4) and (6,6) .. (8,8);
		\foreach \x in {-6,-5,5,6,8}
		{
			\draw (\x,-.1)--(\x,.1);
		}
		\foreach \x in {-5,5,6}
		{
			\draw[dotted, opacity=.6] (\x,.1)--(\x,0.2*\x+4); 
			\draw (\x,0.2*\x+4) node[color=red!80!black]{$\bullet$};
		} 
		\foreach \x in {-6,8}
		{
			\draw[dotted, opacity=.6] (\x,0)--(\x,\x);
			\draw (\x,\x) node[color=red!80!black]{$\bullet$};
		}
		\draw (-7.7,0) node[below, font=\footnotesize]{$-\sigma^{-n_0}a$};
		\draw (-2.5,0) node[below, font=\footnotesize]{$-\sigma^{-n_0}(a-\eps_1)$};
		\draw (5,0) node[below, font=\footnotesize]{$q_2$};
		\draw (6.2,-1.4) node[below, font=\footnotesize]{$\sigma^{-n_0-1}(a-\eps_1)$};
		\draw (8,0) node[below, font=\footnotesize]{$\sigma^{-n_0-1}a$};
		\foreach \x in {2.8,7}
		\draw (-.1,\x)--(.1,\x);
		\draw[->] (0,2.8) to[bend left] (-4,4)node[above, font=\footnotesize]{$\sigma^{-n_0-2}(b+4\eps_1)$};
		\draw (0,7) node[left, font=\footnotesize]{$\sigma^{-n_0-1}(a-3\eps_1)$};
		\end{tikzpicture}
		\caption{\label{f.grafxi2} The diffeomorphism $\xi_2$ of the line.}
	\end{center}
\end{minipage}
\hfill
\begin{minipage}[c]{.52\linewidth}
\begin{tikzpicture}[scale=2.4]
\draw[->] (-.2,0) -- (2.25,0);
\draw[->] (-.2,0) -- (-.2,1.2);
\draw (-.25,1) node[above left]{\footnotesize $1$} -- (-.15,1);
\draw[dotted, color=gray] (-.15,1) -- (.5,1) -- (.5,0);
\draw[dotted, color=gray] (-.15,.9) -- (1.6,.9);
\draw[dotted, color=gray] (1.5,1) -- (1.5,0);
\draw (-.25,.9) node[left]{\small $\frac{b+3\eps_1}{b+4\eps_1}$} -- (-.15,.9);
\draw (0,0) node{$|$} node[below]{\footnotesize $b+2\eps_1$};
\draw (.5,0) node{$|$};
\draw (.4,-.15) node[below]{\footnotesize $b+2\eps_1+\delta_1$};
\draw (1.5,0) node{$|$};
\draw (1.6,-.15) node[below]{\footnotesize $\sigma^{2}(a-2\eps_1-\delta_1)$};
\draw (2,0) node{$|$} node[below]{\footnotesize $\sigma^{2}(a-2\eps_1)$};

\draw[color=blue!80!black, thick] (-.2,0) -- (0,0)to[out=0,in=-90] (.4,.9)to[out=90,in=180](.5,1) -- (1.5,1) to[out=0,in=90] (1.6,.9)to[out=-90,in=180] (2,0) -- (2.2,0);
\draw (1,1) node[above, color=blue!80!black]{$\varphi_2$};

\draw [decorate,decoration={brace,amplitude=5pt}] (-.25,0) -- (-.25,.9) node [midway,xshift=-0.3cm, left] {\footnotesize $\varphi_2$ cvx.};

\end{tikzpicture}
\caption{\label{FigPhi2}The map $\varphi_2$. }
\end{minipage}
\end{figure}

The key dynamical feature about $\xi_2$ is (V2). Note that \eqref{e.oitoepsilon} implies that $8\eps_1 < a+\sigma^{-1}b$, and hence that $q_2 = \sigma^{-n_0-1}(a+\sigma^{-1}b)/2$ satisfies
\[\sigma^{-n_0-2}(b+4\eps_1) < q_2 < \sigma^{-n_0-1}(a-3\eps_1).\]

To form our vertical push over a point $(x,y)$ we shall choose a smooth bump function $\varphi_2$ so that the $x$-coordinate of a point will determine, according to the value $\varphi_2(x)$, the ``amount'' of $\xi_2$ that is going to be applied to $y$. This balance between the $x$-coordinate and the ``intensity'' of the vertical push will play a key role in our arguments. 
Hence, to perform important future estimations, we need to impose some technical assumptions on this bump function, which we now describe.


We choose $\varphi_2 = \varphi_{b+2\eps_1,\sigma^{2}(a-2\eps_1)}^{b+2\eps_1+\delta_1,\sigma^{2}(a-2\eps_1-\delta_1)}$ (recall that by \ref{subsub.escolhasparaavida}, $\delta_1=\eps_1/10$).
Moreover, we assume that $\varphi_2^{\prime\prime}>0$ in restriction to both intervals
\[\left(b+2\eps_1\,,\ (\varphi_2|_{[b+2\eps_1, b+2\eps_1+\delta_1]})^{-1}\big(\frac{b+3\eps_1}{b+4\eps_1}\big)\right)\]
and
\[\left((\varphi_2|_{[\sigma^{2}(a-\eps_1),\sigma^{2}(a-2\eps_1-\delta_1)]})^{-1}\big(\frac{b+3\eps_1}{b+4\eps_1}\big)\,,\ \sigma^{2}(a-2\eps_1)\right)\]
(see Figure~\ref{FigPhi2}).

Our vertical perturbation $h_2$\nomenclature{$h_2$}{Second perturbation of $f_0$, ``vertical push''} is then defined as (see Figure \ref{FigDefh2}) 
\begin{equation}\label{EqDefH2}
h_2(x,y) = \Big(x\,,\ \varphi_2(x) \xi_2(y) + (1-\varphi_2(x)) y\Big).
\end{equation}
By a similar reasoning as we did above for $h_1$, one can see that $h_2\in\dif^{\infty}(\R^2)$. 

\begin{remark}
\label{rem.convex}
The following argument will allow us to profit from the convexity assumption about $\varphi_2$.
If
\[(x,y)\in \Big(\big[b+2\eps_1, b+2\eps_1+\delta_1\big]\cup \big[\sigma^{2}(a-2\eps_1-\delta_1), \sigma^{2}(a-2\eps_1)\big]\Big)  \times \big[0,\sigma^{-n_0-2}(b+3\eps_1)\big],\]
then $\xi_2(y)\geq \sigma^{-n_0-2}(b+4\eps_1)$. Hence, if $\varphi_2(x)\geq \frac{b+3\eps_1}{b+4\eps_1}$ then the second coordinate of $h_2(x,y)$ is bigger than $\sigma^{-n_0-2}(b+3\eps_1)$. Thus, if the second coordinate of $h_2(x,y)$ is strictly smaller than $\sigma^{-n_0-2}(b+3\eps_1)$ then $\varphi_2(x)<\frac{b+3\eps_1}{b+4\eps_1}$, and so by our convexity assumption, the restriction of $\varphi_2$ to the interval $[x,\sigma^2(a-2\eps_1)]$ (or $[b+2\eps_1,x]$, depending on the context) is convex.  
\end{remark}

%

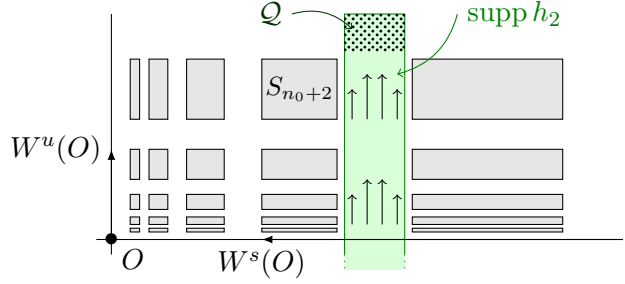
\begin{figure}
\begin{center}
\begin{tikzpicture}[scale=2]
\fill[color=green, opacity=.15] (1.55,-.2) rectangle (1.95,1.5);
\draw[color=green!50!black] (1.55,-.1) -- (1.55,1.5) -- (1.95,1.5) -- (1.95,-.1);
\draw[color=green!50!black, dotted] (1.55,-.1) -- (1.55,-.2);
\draw[color=green!50!black, dotted] (1.95,-.1) -- (1.95,-.2);
\fill[pattern=crosshatch dots] (1.55,1.25) rectangle (1.95,1.5);
\draw[->] (1.6,.1) -- (1.6,.3);
\draw[->] (1.7,.1) -- (1.7,.4);
\draw[->] (1.8,.1) -- (1.8,.4);
\draw[->] (1.9,.1) -- (1.9,.3);
\draw[->] (1.6,.8) -- (1.6,1);
\draw[->] (1.7,.8) -- (1.7,1.1);
\draw[->] (1.8,.8) -- (1.8,1.1);
\draw[->] (1.9,.8) -- (1.9,1);
\draw (-.1,0) -- (3.4,0);
\draw (0,-.1) -- (0,1.5);
\draw[->,>=latex] (3.4,0) -- (1,0) node[below]{$W^s(O)$};
\draw[->,>=latex] (0,0) -- (0,.6) node[left]{$W^u(O)$};
\draw (0,0) node{$\bullet$} node[below right]{$O$};

\foreach \i in {0,...,4}
{\foreach \j in {0,...,4}
{\draw[fill=black!10!white] (2*0.5^\i,0.8*0.5^\j) rectangle (2*0.5^\i*1.5,0.8*0.5^\j*1.5);}}
			
\draw [->, color=green!50!black] (2.3,1.5)node[right]{$\operatorname{supp} h_2$} to[bend left] (1.9,1.1);
\draw [->, color=green!20!black] (1.2,1.5)node[left]{$\mathcal{Q}$} to[bend left] (1.6,1.4);
\draw  (1.25,1) node{\small $S_{n_0+2}$};
\end{tikzpicture}
\caption{\label{FigDefh2}The diffeomorphism $h_2$ and its attracting region $\mathcal{Q}$.}
\end{center}
\end{figure}

Our perturbation map $h_2\circ h_1$ has an attracting region, as shown by the following lemma.

\begin{lema}
\label{rem.hdoisnoq}	
The diffeomorphism $h_2\circ h_1$ maps $[b+\eps_1+\delta_1,\sigma^2(a-\eps_1-\delta_1)] \times [-\sigma^{-n_0-1}(a-\eps_1),\sigma^{-n_0-1}(a-\eps_1)]$ inside 
\begin{equation}
\label{e.q}
\mathcal{Q} \eqdef [b+2\eps_1+\delta_1,\sigma^2(a-2\eps_1-\delta_1)] \times [\sigma^{-n_0-2}(b+4\eps_1), \sigma^{-n_0-1}(a-3\eps_1)].
\end{equation}
In particular, $\mathcal{Q}$ is a trapping region for $h_2\circ h_1$.
\end{lema}

\begin{proof}
Take a point $(x,y)\in [b+\eps_1+\delta_1,\sigma^2(a-\eps_1-\delta_1)]\times [0,\sigma^{-n_0}(a-\eps_1)]$. Then, by definition of $h_1$ and by (H3) we have that $h_1(x,y)=(\bar{x},y)$, with $b+2\eps_1+\delta_1\leq\bar{x}\leq\sigma^2(a-2\eps_1-\delta_1)$. But this implies that $\varphi_2(\bar{x})=1$ and so $h_2(\bar{x},y)=(\bar{x},\xi_2(y))$. Now, the conclusion follows from (V2) .
\end{proof}

\bigskip

%
%

\subsection{Description of the statistical basin}

We are now in position to define the new diffeomorphism $f_1\in\difp$ by
\begin{equation}
f_1 = h_2\circ h_1\circ f_0\label{Def.h}.
\end{equation}
Since both diffeomorphisms $h_1$ and $h_2$ have their supports disjoint from the origin $O$, $f_0$ and $f_1$ coincide in a neighbourhood of $O$. In particular, $O$ is a hyperbolic fixed point of saddle type for $f_1$. Also, observe that $h_1|_{\cS}=h_2|_{\cS}=\Id$, and so $f_1|_{\mathcal{S}}=f_0|_{\mathcal{S}}$. In particular, the stable tower $\mathcal{S}$ is totally invariant under $f_1$ and moreover, as it follows from Lemma~\ref{suf}, we have the inclusion $\mathcal{S}\subset\cB_{f_1}(\delta_O)$. In this section and the next one, we will prove that actually $\cB_{f_1}(\delta_O)$ \emph{coincides} with $\mathcal{S}$, up to a nowhere dense zero Lebesgue measure set (this is what we mean by describing the statistical basin $\cB_{f_1}(\delta_O)$). In the following statement, which is the main result of this section, we shall use the notation
\[W^s(\mathcal{Q})\eqdef\left\{p\in\R^2;\,\exists\:n>0;f_1^n(p)\in\mathcal{Q}\right\}.\]
Moreover, recall that we denote by $\mathcal{L}_r^i\subset\R_+^2$ the interior of the figure-eight attractor of $f_0$ in the first quadrant. Finally, recall that the stable boxes $S_n$ were defined in Section \ref{subsecboxes}, while the $\varepsilon_1$-boxes $C_n$ were defined in Section \ref{subsub.moreboxes}.

From this moment, we will only deal with points of $\mathcal{L}_r^i$, the dynamics of points of $\mathcal{L}_\ell^i$ being identical up to a symmetry, and the dynamics of points of $\mathcal{L}^e$ being similar. The specific moments where the arguments differ will be pointed out in Subsection~\ref{SubsecExt}.

\begin{prop}\label{p.bluetrichotomy}
There exists a set\, $\Gamma\subset\bigcup_{n \geq n_0}\big(C_n \setminus S_n\big)$ (described in Definition~\ref{defpersistent}) such that if the positive orbit of a point $p\in\mathcal{L}_r^i$ under $f_1$ meets $[a-3\eps_1,\,\sigma^2(a-3\eps_1)]\times [0,\sigma^{-n_0-2}b]$, then there are only three possible cases:
\begin{enumerate}
\item $p\in\mathcal S$;
\item $p\in W^s(\mathcal Q)$;
\item $p\in\bigcup_{n\in\Z}f_1^{n}(\Gamma)$.
\end{enumerate}
Moreover,
\[\mathcal{L}_r^i\cap(W_{f_1}^s(O)\cup\mathcal{S})\,\,\subset\,\,\cB_{f_1}(\delta_O)\cap\mathcal{L}_r^i\,\,\subset\,\,\ W_{f_1}^s(O)\cup\mathcal{S}\cup\bigcup_{n\in\Z} f_1^{n}(\Gamma).\]
\end{prop}

Both inclusions in the final line of the statement of Proposition \ref{p.bluetrichotomy} are symmetric with respect to the first and third quadrant, as all constructions along the paper. In Section \ref{SecGeom} we will prove that the whole orbit $\mathcal{O}_{f_1}(\Gamma)=\bigcup_{n\in\Z}f_1^{n}(\Gamma)$ of $\Gamma$ under $f_1$ has zero Lebesgue measure and is nowhere dense (see Proposition \ref{l.persmedidanula}). Note that, having established this fact, we do not need to worry about the intersection $\cB_{f_1}(\delta_O)\cap\mathcal{O}_{f_1}(\Gamma)$ in order to prove Theorem \ref{main.exemplonovo}. The conclusion about $\cB_{f_1}(\delta_O)$ in the statement of Proposition \ref{p.bluetrichotomy} will follow from the fact (that will be proved in the next subsection) that $\mathcal{Q}$ is a trapping region for some iterate of $f_1$, \emph{i.e.}, $f_1^N(\mathcal{Q})\subset\interior(\mathcal{Q})$ for some $N>0$. This will make the three cases of the proposition mutually exclusive. Let us give a rough sketch of the proof of Proposition~\ref{p.bluetrichotomy}. First we prove that every point $p$ as in the statement which is outside the $\eps_1$-tower $\mathcal C$ (see paragraph \eqref{subsub.moreboxes}) belongs either to $W^s(\mathcal Q)$, or the pre-orbit of $\mathcal C$. The dynamics inside the $\eps_1$-tower $\mathcal C$ is more intricate. We shall subdivide $\mathcal C \setminus\mathcal S$ into several regions all of which but one is wandering. The remaining region is also a tower of smaller rectangles, which is very close to the places where our perturbations $h_1$ and $h_2$ are being dissolved. We shall prove that the sole possibility for a point in the pre-orbit of the $\eps_1$ tower not belonging to $W^s(\mathcal Q)$ is that it has a future iterate which enters in this tower of smaller rectangles (the blue tower $\Ba$ defined below in \eqref{e.bluetower}) and never leave it again for future iterations. A point with this property is what we shall call a \emph{persistent point} (see Definition~\ref{defpersistent}). The set $\Gamma$ in the statement of Proposition \ref{p.bluetrichotomy} is precisely this set of persistent points.

\subsection{A trapping region for $f_1$}
Following the above sketch, the first property about $f_1$ that we are going to prove is that $\mathcal{Q}$, defined in \eqref{e.q}, is a periodic trapping region. In the next lemma (both on its statement and its proof) we use the notation and results from Proposition \ref{buildTowers}.  

\begin{lema}
\label{l.criandopoco}
The set $\mathcal{Q}$ is a periodic trapping region for $f_1$, that is:\[f_1^{n_0+k_0+1}(\mathcal{Q})\subset\interior(\mathcal{Q}).\]
\end{lema}

\begin{figure}
\begin{center}
\begin{tikzpicture}[scale=5]
\clip (-.35,-.2) rectangle (2.1,1.9);

\fill[color=blue!10!white] (0.98,0.98) rectangle (0.39, -.15);
\draw (0,1.6) -- (0,0) -- (1.6,0);
\draw[->,>=latex] (1.6,0) -- (.17,0) node[below]{$W^s(O)$};
\draw[->,>=latex] (0,0) -- (0,.44) node[left]{$W^u(O)$};
\draw (0,0) node{$\bullet$} node[below left]{$O$};
\draw[color=blue!50!black] (0.62,.65) node{\small $\operatorname{supp}(h_2)$};

\foreach \i in {0,...,3}
{\foreach \j in {0,...,6}
{\draw[fill=black!10!white] (0.25^\i,0.5^\j) rectangle (0.25^\i*1.5,0.5^\j*1.5);}}

\fill[pattern=crosshatch dots] (0.98,.98) rectangle (0.39, 0.76);
\draw[color=red, very thick] (0.98,0.98) -- (0.98,0.76);
\draw[color=blue, very thick] (0.98,0.76) -- (0.39,0.76);
\draw[color=green!70!black, very thick] (0.39,0.76) -- (0.39, 0.98);
\draw[color=yellow!90!black, very thick] (0.39, 0.98) -- (0.98,0.98);

\fill[pattern=crosshatch dots] (.24,1.8) rectangle (0.1,1.51);
\draw[color=red, very thick] (.24,1.8) -- (.24,1.51);
\draw[color=blue, very thick] (.24,1.51) -- (0.1,1.51);
\draw[color=green!70!black, very thick] (0.1,1.51) -- (0.1,1.8);
\draw[color=yellow!90!black, very thick] (0.1,1.8) -- (.24,1.8);

\fill[pattern=crosshatch dots] (.98,.24) to[out=180,in=-0] (0.39,.5) -- (0.39,0.2) to[in=180, out=-0] (.98,0.1) -- cycle;
\draw[color=yellow!90!black, very thick] (0.39,.5) -- (0.39,0.2);
\draw[color=red, very thick] (.98,.24) to[out=180,in=-0] (0.39,.5);
\draw[color=green!70!black, very thick] (0.39,0.2) to[in=180, out=-0] (.98,0.1);
\draw[color=blue, very thick] (.98,.24) -- (.98,0.1);

\draw[->] (0.38,0.85) to[bend left] (0.18,1.49);
\draw[->] (0.27,1.62) .. controls +(1.8,.8) and +(1.4,0.3) .. (1,0.2);
\draw[->] (0.85,0.3) to[bend right] node[midway, left]{$h_2\circ h_1$} (0.85,0.72);
\draw (0.17, 0.92) node{$f_0^{n_0}$};
\draw (1.97,1) node{$f_0^{k_0+1}$};

\draw (0.7,1.05) node{$\mathcal Q$};
\draw (1.25,1.25) node{$f_0^{-1}(S_{n_0})$};
\draw (1.25,.625) node{$f_0^{-1}(S_{n_0+1})$};

\draw (1,.02) -- (1,-.03) node[below]{$\sigma^2 a$};
\draw (.375,.02) -- (.375,-.03) node[below]{$b$};
\draw (.02,1) -- (-.03,1) node[left]{$\sigma^{-n_0-1}a$};
\draw (.02,.75) -- (-.03,.75) node[left]{$\sigma^{-n_0-2}b$};

\end{tikzpicture}
\caption{\label{FigTrapped}Trajectory of the trapped region $\mathcal Q$: one has $f_1^{n_0+k_0}(\mathcal Q) = h_2\circ h_1(f_0^{n_0+k_0+1}(\mathcal Q))\subset\mathcal{Q}$.}
\end{center}
\end{figure}

\begin{proof}
The reader may refer to Figure \ref{FigTrapped}. The idea is that the first $n_0+k_0$ iterates of $\mathcal{Q}$ under $f_0$ lie outside the supports of $h_1$ and $h_2$, while $f_0^{n_0+k_0+1}(\mathcal{Q})$ lies within the rectangle which is mapped by $h_2\circ h_1$ inside the interior of $\mathcal{Q}$ (as in Lemma~\ref{rem.hdoisnoq}).
	
Let us give the precise argument. First, as $\mathcal Q\subset\mathcal V$, by \eqref{p1} of Proposition~\ref{buildTowers}, one has 
\[f_0^{n_0+1}(\mathcal Q) = \big[\sigma^{-2n_0-2}(b+2\eps_1+\delta_1),\,\sigma^{-2n_0}(a-2\eps_1-\delta_1)\big] \times \big[\sigma^{-1} (b+4\eps_1),\, a-3\eps_1\big],\]
and the $n_0+1$ first iterates of $\mathcal Q$ do not meet the supports of $h_1$ and $h_2$. By a similar argument, combined with \eqref{p1'} of Proposition~\ref{buildTowers}, the next iterates up to time $n_0+k_0$ do not meet the supports of $h_1$ and $h_2$ either.

We  claim that 
\[f_0^{n_0+k_0+1}(\mathcal Q) \subset [b+3\eps_1,\sigma^2(a-4\eps_1)]\times (0,\sigma^{-n_0-4}).\]
Once this claim is established the lemma is proved for we have seen above that $f_0=f_1$ along the first $n_0+k_0$ iterates of $\mathcal{Q}$, and thus the claim implies that 
\[
f_1^{n_0+k_0+1}\big(\mathcal{Q}\big)\subset h_2\circ h_1\big([b+3\eps_1,\sigma^2(a-4\eps_1)]\times (0,\sigma^{-n_0-4})\big),
\] 
and the right-hand side above is included in $\interior(\mathcal{Q})$ due to Lemma~\ref{rem.hdoisnoq}.

To prove the claim, we use \eqref{p3} of Proposition~\ref{buildTowers}, which says that the dynamics near the two horizontal sides of $f_0^{n_0+1}(\mathcal Q)$ by $f_0^{k_0}$ is a rotation of $\pi/2$ centred at $(\frac{a+b}{2},\frac{a+b}{2})$, to conclude that the boundary of the set $f_0^{n_0+k_0+1}(\mathcal Q)$ contains the vertical segments $\{b+3\eps_1\}\times \big[\sigma^{-2n_0-2}(b+2\eps_1+\delta_1),\,\sigma^{-2n_0}(a-2\eps_1-\delta_1)\big]$ and  $\{\sigma^2(a-4\eps_1)\}\times \big[\sigma^{-2n_0-5}(b+2\eps_1+\delta_1),\,\sigma^{-2n_0-3}(a-2\eps_1-\delta_1)\big]$ (the yellow and the blue segments in Figure~\ref{FigTrapped}, respectively).   

Therefore, applying now \eqref{p5} of Proposition~\ref{buildTowers}, we deduce that any vertical interval inside $f_0^{n_0+1}(\mathcal Q)$ is mapped by $f_0^{k_0}$ to a curve contained in $[b+3\eps_1,\sigma^2(a-4\eps_1)]\times (0,\sigma^{-n_0-4})$, concluding. 
\end{proof}

\subsection{Dynamics outside the $\eps_1$-tower}

Recall Subsection~\ref{subsub.moreboxes}, where we have defined the $\eps_1$-tower $\mathcal{C}$, which contains the whole orbit under $f_0$ of the stable boxes $S_n$ and the regions where we dissolve $h_1$ and $h_2$. In particular, $\mathcal{C}$ is an $f_0$-invariant set. The dynamics on the complement of $\mathcal{C}$ under $f_1$, as we shall see below, is quite simple: every point there which comes close to $O$ eventually hits the trapping region $\mathcal{Q}$. 

Recall that $\mathcal{V}$ is the neighbourhood of $O$ inside of which the dynamics of $f_0$ is linear (see Proposition~\ref{buildTowers}). 

\begin{lema}
\label{l.bigbacia2}
Let $\tilde p=(x,y)\in\R^2$ with $a-3\eps_1\leq x\leq\sigma^2 (a-3\eps_1)$ and $0<y<\sigma^{-n_0-2}(a-3\eps_1)$. Consider $p = f_0^{-1}(\tilde p)$. Then at least one of the following holds:
\begin{enumerate}
	\item $f_1(p)\in\mathcal{C}$ or $f_1^2(p)\in\mathcal{C}$;
	\item $p\in W^s(\mathcal{Q})$.
\end{enumerate}
\end{lema}

Remark that the set $[a-3\eps_1,\,\sigma^2 (a-3\eps_1)] \times [0,\, \sigma^{-n_0-2}(a-3\eps_1)]$ is a fundamental domain for the action of $f_0$: each point $p\in \mathcal{L}_r^i\subset\R_+^2$ that comes close enough to the stable manifold of $O$ has an iterate by $f_1$ whose image by $f_0$ crosses this set (and this happens only once at each return in $\mathcal V$).

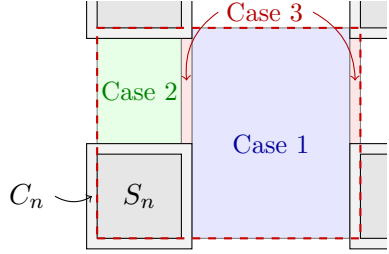
\begin{figure}
\begin{center}
\begin{tikzpicture}[scale=.7]
\clip (-2.5,-.5) rectangle (5.7,4.7);

\filldraw[color=gray, fill=blue!10!white] (2,.2) rectangle (5,4.2);
\draw[color=blue!60!black] (3.5,2) node{\small Case 1};
\filldraw[color=gray, fill=green!10!white] (.2,2) rectangle (1.8,4);
\draw[color=green!50!black] (1,3) node{\small Case 2};

\filldraw[color=gray, fill=red!10!white] (5,2) rectangle (5.2,4);
\filldraw[color=gray, fill=red!10!white] (1.8,2) rectangle (2,4);
\draw[color=red!70!black,<-] (5.1,3.2) to[bend right] (4.3,4.5) node[left]{\small Case 3};
\draw[color=red!70!black,<-] (1.9,3.2) to[bend left] (2.5,4.5);

\draw[fill=gray!10!white] (0,0) rectangle (2,2);
\draw[fill=gray!10!white] (5,0) rectangle (7,2);
\draw[fill=gray!10!white] (0,4) rectangle (2,6);
\draw[fill=gray!10!white] (5,4) rectangle (7,6);

\draw[fill=gray!20!white] (.2,.2) rectangle (1.8,1.8);
\draw[fill=gray!20!white] (5.2,.2) rectangle (6.8,1.8);
\draw[fill=gray!20!white] (.2,4.2) rectangle (1.8,5.8);
\draw[fill=gray!20!white] (5.2,4.2) rectangle (6.8,5.8);
\draw[color=red!80!black, thick, dashed] (.2,.2) rectangle (5.2,4.2);

\draw(1,1) node{$S_n$};
\draw[<-] (.1,1) to[bend left] (-.6,1) node[left]{$C_n$};
\end{tikzpicture}
\caption{\label{FigCasesLemBacia}The different cases of the proof of Lemma~\ref{l.bigbacia2}}
\end{center}
\end{figure}

\begin{proof}
There exists some (unique) integer $n\ge n_0+2$ such that $\sigma^{-n-1}a\le y < \sigma^{-n}a$ (see the red dashed rectangle of Figure~\ref{FigCasesLemBacia}). We suppose that $p\notin \mathcal C$, and break down the argument into several cases (depicted in Figure~\ref{FigCasesLemBacia}). 
\medskip

\paragraph{\textit{Case 1: $x\in[b+3\eps_1, \sigma^2(a-3\eps_1)]$}}
In this case, by Lemma \ref{rem.hdoisnoq}, one has $h_2\circ h_1(\tilde p)\in\mathcal Q$ and hence $f_1(p) \in \mathcal Q$.
\medskip

\paragraph{\textit{Case 2: $x\in[a, b]$ and $y\in [\sigma^{-n-1}(b+3\eps_1), \sigma^{-n}(a-3\eps_1)]$}}
In this case, $\tilde p$ is not in the support of $h_2\circ h_1$. Using Proposition~\ref{buildTowers} (as in the proof of Lemma~\ref{l.criandopoco}), one can see that the first return of $p$ in $[a-3\eps_1,\,\sigma^2 (a-3\eps_1)]\times [0,\sigma^{-n_0-2}(a-3\eps_1)]$ belongs to the set of Case 1, so $p\in W^s(\mathcal{Q})$.
\medskip

\paragraph{\textit{Case 3: $x\in[\sigma^2 (a-3\eps_1),\sigma^2 a]$ or $x\in[\sigma^2 b,\sigma^2 (b+3\eps_1)]$, and $y\in [\sigma^{-n-1}(b+3\eps_1), \sigma^{-n}(a-3\eps_1)]$}}
This case is a bit more complicated, as $\tilde p$ belongs to the support of $h_2\circ h_1$, but we do not know if $h_2(h_1(\tilde p))$ belongs to $\mathcal Q$ or not. Let us treat the case $x\in[\sigma^2 (a-3\eps_1),\sigma^2 a]$, the other being similar.

We know by the definitions of $h_2$ and $h_1$ that the first coordinate of $h_2(h_1(\tilde p))$ belongs to $[\sigma^2 (a-4\eps_1),\sigma^2 a]$, and that the second coordinate, denoted $\bar y$, is smaller that $\sigma^{-n_0-1} a$. We have three cases:
\begin{itemize}
\item $\bar y\ge \sigma^{-n_0-2}(b+4\eps_1)$. In this case, $h_2(h_1(\tilde p)) = f_1(p) \in\mathcal Q$, and hence $p\in W^s(\mathcal{Q})$.
\item $\bar y \in [\sigma^{-n}(a-3\eps_1),\, \sigma^{-n}(b+3\eps_1)]$ for some $n\ge n_0-1$. In this case, $h_2(h_1(\tilde p)) = f_1(p) \in f_0^{-1}(C_n)$, and hence $f_1^2(p)\in \mathcal C$.
\item $\bar y \in [\sigma^{-n-1}(b+3\eps_1),\, \sigma^{-n}(a-3\eps_1)]$ for some $n\ge n_0-1$. In this case, reasoning as in the proof of Lemma~\ref{l.criandopoco}, one can see that there is some $m\in\nt$ such that $f_0(f_1^m(p)) \in  [a-3\eps_1,\,\sigma^2 (a-3\eps_1)] \times [0,\, \sigma^{-n_0-2}(a-3\eps_1)]$, and that this point falls in Case 1 of the proof. Hence, $p\in W^s(\mathcal{Q})$.
\end{itemize}
\end{proof}

\subsection{Dynamics inside the $\eps_1$-tower and persistent points} In this subsection we prove the existence of a special subset $\Gamma\subset\bigcup_{n \geq n_0}\big(C_n \setminus S_n\big)$ so that any point which is not in the stable tower $\cS$ nor in the stable set $W^s(\mathcal{Q})$ of the trapping region belongs to the orbit of $\Gamma$. Similarly to in Definition~\ref{defgzero}, we shall introduce a first return map $g_1$ to $\bigcup_{n\in\nt} C_n$; the set $\Gamma$ will be invariant under this first return map. The set $\Gamma$ does contain points of the statistical basin $\cB_{f_1}(\delta_O)$, and that is why we \emph{cannot neglect it}: it is a natural technical difficulty arising from the strategy we adopted with the perturbations $h_1$ and $h_2$. In the subsequent section of the paper, we then move to study the dynamics, topology and measure of the set $\Gamma$ and its iterates under $f_1$ (see Propositions \ref{l.persmedidanula} and \ref{PropCoding}). The global picture of the dynamics inside the $\eps_1$-tower is described in Figure~\ref{FigShapesImgBox}.

As can be seen in Figure~\ref{FigShapesImgBox}, the box $f_0^{-1}(\tilde S_{n_0+1})$ has the special property that its left boundary meets the support of the perturbation $h_2\circ h_1$ (and is mapped inside $\mathcal Q$ by $h_2\circ h_1$), while its right boundary does not meet it (and its iterate by $f_1$ is equal to its iterate by $f_0$). This is the only box having such property, the other ones having both their left boundary and the image of their right boundary by $f_0$ meeting $\supp(h_2\circ h_1)$. This is the reason why we will define the number $n_1(h)$ below. We will do all the proofs for the boxes $\tilde S_n$ with $n\ge n_0+2$; the reader may check that these proofs still work in this special case $n=n_0+1$.

\begin{figure}
\begin{center}
\begin{tikzpicture}[scale=3.8]

\fill[color=blue!10!white] (0.1875, -.15) rectangle (.5,1);
\draw (2,.02) -- (2,-.03) node[below]{$\sigma^4 a$};
\draw (3,.02) -- (3,-.03) node[below]{$\sigma^4 b$};
\draw (.5,.02) -- (.5,-.03) node[below]{$\sigma^2 a$};
\draw (.75,.02) -- (.75,-.03) node[below]{$\sigma^2 b$};
\draw (0.1875,.02) -- (0.1875,-.03) node[below]{$b$};
\draw (.02,1) -- (-.03,1) node[left]{$\sigma^{-n_0-1}a$};
\draw (.02,.75) -- (-.03,.75) node[left]{$\sigma^{-n_0-2}b$};

\draw[color=blue!70!black,->] (.55,-.25) node[right]{$\operatorname{supp}(h_2\circ h_1)$} to[bend left] (.35,-.07);
\draw[->] (1.85,.13) to[bend left] (.8,.26);
\draw (1.25,.12) node {$f_1$};
\draw[->] (.45,.27) to[bend left] (.25,.47);
\draw (.35,.4) node {$f_1$};
\draw (.35,1.05) node{$\mathcal Q$};
\draw (1.35,.55) node{$f_0^{-1}(\mathcal Q)$};
\draw (0.25,1.7) node{$f_0(\mathcal Q)$};

\draw (0,2) -- (0,0) -- (3.1,0);
\draw[->,>=latex] (2.6,0) -- (1.7,0) node[below]{$W^s(O)$};
\draw[->,>=latex] (0,0) -- (0,.25) node[left]{$W^u(O)$};
\draw (0,0) node{$\bullet$} node[below left]{$O$};

\foreach \i in {0,...,4}
{\foreach \j in {0,...,7}
{\draw[fill=black!10!white] (0.25^\i*2,0.5^\j) rectangle (0.25^\i*1.5*2,0.5^\j*1.5);}}

\fill[color=white] (1.9,.9) rectangle (3.1,1.6);

\fill[pattern=crosshatch dots] (0.77, 0.405) rectangle (1.97,.47);
\fill[pattern=crosshatch dots] (0.2, .77) rectangle (.48,.98);
\fill[pattern=crosshatch dots] (0.06, 1.52) rectangle (.11,1.97);

\fill[color=red, opacity=.4] (1.9,0.11) rectangle (3.1,0.2);
\fill[color=red, opacity=.4] (0.46,.9) .. controls +(.02,0) and +(0,.2) .. (0.49,.7) .. controls +(0,-.2) and +(-.02,0) .. (0.5,0.4) -- (0.78,0.4) -- (0.78,0.22) -- (0.5,0.22) .. controls +(-.02,0) and +(0,-.2) .. (0.48,.6)
.. controls +(0,.2) and +(.02,0) ..
(0.46,0.85) -- cycle; 

\fill[color=red, opacity=.4] (0.1,1.9) .. controls +(.02,0) and +(-.02,0) .. (0.125,0.78) -- (0.19,0.78) .. controls +(.02,0) and +(-.02,0) .. (0.22,0.95) -- (0.22,0.9) .. controls +(-.02,0) and +(.02,0) .. (0.19,0.48) -- (0.125,0.48) .. controls +(-.03,0) and +(.02,0) .. (0.1,1.7) -- cycle ;

\draw[<-](.625,1.25)to[bend left] (.9,1.3) node[right]{$f_0^{-1}(S_{n_0})$};
\draw[<-](.625,.62)to[bend left] (.9,.8) node[right]{$f_0^{-1}(S_{n_0+1})$};
\draw[color=red!70!black] (1.92,.2) node[left] {$f_0^{-2}(C_{n_0+2})$};
\draw[color=red!70!black] (.78,.33) node[right] {$f_1f_0^{-2}(C_{n_0+2})$};
\draw[color=red!70!black, <-] (.15,.62) to[bend left] (-.05,.55) node[left] {$f_1^2f_0^{-2}(C_{n_0+2})$};

\end{tikzpicture}
\caption{\label{FigShapesImgBox}Dynamics of the $\eps_1$-boxes $C_n$.}
\end{center}
\end{figure}

\subsubsection{Decomposition of the $\eps_1$-tower}\label{subsubsec.lemadascores}
Let us recall some notations introduced in Proposition~\ref{buildTowers}: for $n\ge n_0$, $L_n$ is the affine identification between $S_n$ and $[-1,1]^2$ (see \eqref{e.ln}). Let $\tilde\eps_1 = 2\eps_1/(b-a)$\nomenclature{$\tilde \eps_1$}{Counterpart of $\eps_1$ in $[-1,1]^2$, that is, $\tilde\eps_1 = 2\eps_1/(b-a)$}, so that $L_n(C_n) = [-1-3\tilde\eps_1,1+3\tilde\eps_1]^2\subset L_n(\widetilde{S}_n)$. Similarly, we denote $\tilde\delta_1 = 2\delta_1/(b-a)$. We shall take adventage of these identifications to look at all boxes $C_n$ in a single drawing, as in Figure~\ref{FigSn}. Property (\ref{p3}) of Proposition \ref{buildTowers} implies that the restriction of $L_{2n} \circ f_0^{n+k_0} \circ L_n^{-1}$ to $L_n(C_n)$ is in fact the linear rotation of angle $\pi/2$ centred at the origin. Our next goal is to understand the action of $f_1^{n+k_0}$ on each box $C_n$. For this purpose, we first divide each $C_n$ into regions as depicted in Figure \ref{FigSn}, where the inside grey squares represent the image of the set $S_n$ under $L_n$. These regions are defined by the vertical lines of respective abscissa $a$, $a-\tilde\eps_1$, $a-2\tilde\eps_1$, $a-2\tilde\eps_1-\tilde\delta_1$, $a-3\tilde\eps_1$, $b$, $b+\tilde\eps_1$, $b+2\tilde\eps_1$, $b+2\tilde\eps_1-\tilde\delta_1$ and $b+3\tilde\eps_1$, and their images by the map $L_{2n} f_0^{n+k}(L_n)^{-1}$.

We shall subdivide this set into 5 groups, and we will analyse their dynamics separately. Along the next paragraphs, for the sake of simplicity, we shall consistently make reference to Figure~\ref{FigSn}. Let us give the precise definition of each one of these regions. We denote them $\cA^{C}_{h,v}(n)\subset C_n\setminus S_n$\nomenclature{$\cA^{C}_{h,v}$}{Regions of $L_n(C_n\setminus S_n)$}\nomenclature{$\cA^{C}_{h,v}(n)$}{Regions of $C_n\setminus S_n$, inverse images of the regions $\cA^{C}_{h,v}$ by $L_n$}, and they are defined from the sets $\cA^{C}_{h,v}$ by the formula
\[\cA^{C}_{h,v}(n) = L_n^{-1}\big(\cA^{C}_{h,v}\big) .\]
The upper script letter $C$ stands for the region's colour: $R$ for red, $B$ for blue, $G$ for green, $O$ for orange and $P$ for pink. For its part, $h\in \{\ell,r\}$ is either left or right, and $v\in\{t,b\}$ is either top or bottom. Finally, $n\ge n_0$ is the box's number.

These regions have special symmetry properties: denoting $s_v$ and $s_h$ the axial symmetries with axes respectively the $x$ and the $y$ axes, one has,
\begin{itemize}
\item $\cA^{B}_{r,t} = s_h(\cA^{B}_{\ell,t})$, $\cA^{B}_{\ell,b} = s_v(\cA^{B}_{\ell,t})$ and $\cA^{B}_{r,b} = s_v(\cA^{B}_{r,t})$;
\item for $C\in\{R,G\}$, $\cA^{C}_{v} = s_h(\cA^{C}_{v})$ and $\cA^{C}_{t} = s_v(\cA^{C}_{b})$;
\item $\cA^{P}_{h} = s_v(\cA^{P}_{h})$ and $\cA^{P}_{l} = s_h(\cA^{P}_{r})$;
\item $\cA^{O} = s_v(\cA^{O}) = s_h(\cA^{O})$.
\end{itemize} 
So it suffices to define the following boxes, as follows.

\medskip

\noindent\textbf{The red box} (cross-hatched rectangles in Figure~\ref{FigSn}):
\[\cA^{R}_{t} = [-1-3\tilde\eps_1,1+3\eps_1]\times[1+\tilde\eps_1+\tilde\delta_1, 1+3\tilde\eps_1].\]
\textbf{The blue box} (rectangles with tilted hatches in Figure~\ref{FigSn}):
\[\cA^B_{\ell,t} = [-1-3\tilde\eps_1,-1-2\tilde\eps_1]\times[1+\tilde\eps_1,1+\tilde\eps_1+\tilde\delta_1].\]
\textbf{The green box} (gridded rectangles in Figure~\ref{FigSn}):
\[\cA^{G}_{t} = [-1-2\tilde\eps_1,1+2\tilde\eps_1]\times[1+\tilde\eps_1,1+\tilde\eps_1+\tilde\delta_1].\]
\textbf{The orange region} (dotted region in Figure~\ref{FigSn}):
\[\cA^{O} = \Big([-1-2\tilde\eps_1,1+2\tilde\eps_1]\times[-1-\tilde\eps_1,1+\tilde\eps_1]\Big)\setminus\tilde S_n.\]
\textbf{The pink box} (hatched squares in Figure~\ref{FigSn}):
\[\cA^{P}_{\ell} = [-1-3\tilde\eps_1,-1-2\tilde\eps_1]\times[-1-\tilde\eps_1,1+\tilde\eps_1].\]
\medskip

\begin{figure}
\captionsetup{width=.48\linewidth}
\begin{minipage}[c]{.48\linewidth}
\resizebox{\textwidth}{!}{
\begin{tikzpicture}[scale=.5]

\draw[color=white] (4.5,-8) node [below]{\small $1+1.5\tilde\eps_1$}; 
\draw [color=white](-4.5,-8) node [below]{\small $-1-1.1\tilde\eps_1$} ;

\foreach \i in {-1,1}{
\foreach \j in {-1,1}{
\fill[color=blue, opacity=.4] (6*\i,4*\j) rectangle (5*\i,4.5*\j);
\fill[pattern=north east lines] (6*\i,4*\j) rectangle (5*\i,4.5*\j);
}}
\foreach \j in {-1,1}{
\fill[color=red, opacity=.4] (-6,4.5*\j) rectangle (6,6*\j);
\fill[pattern=crosshatch] (-6,4.5*\j) rectangle (6,6*\j);
\fill[color=orange, opacity=.4] (-5,3*\j) rectangle (5,4*\j);
\fill[pattern=crosshatch dots] (-5,3*\j) rectangle (5,4*\j);
\fill[color=orange, opacity=.4] (-5*\j,-3) rectangle (-3*\j,3);
\fill[pattern=crosshatch dots] (-5*\j,-3) rectangle (-3*\j,3);
\fill[color=magenta, opacity=.4] (5*\j,-4) rectangle (6*\j,4);
\fill[pattern=horizontal lines] (5*\j,-4) rectangle (6*\j,4);
\fill[color=green, opacity=.4] (-5,4*\j) rectangle (5,4.5*\j);
\fill[pattern=grid] (-5,4*\j) rectangle (5,4.5*\j);
}

\fill[color=black!10!white] (-3,-3) rectangle (3,3);

\draw[thick] (-7,-3) -- (7,-3);
\draw[thick] (-7,3) -- (7,3);
\draw[thick, color=red!80!black] (-7,-4) -- (7,-4);
\draw[thick, color=red!80!black] (-7,4) -- (7,4);
\draw[thick, color=green!60!black] (-7,-5) -- (7,-5);
\draw[thick, color=green!60!black] (-7,5) -- (7,5);
\draw[thick, color=blue!90!black] (-7,-6) -- (7,-6);
\draw[thick, color=blue!90!black] (-7,6) -- (7,6);
\draw[dotted] (-7,-4.5) -- (7,-4.5);
\draw[dotted] (-7,4.5) -- (7,4.5);

\draw[thick] (-3,-7) -- (-3,7);
\draw[thick] (3,-7) -- (3,7);
\draw[thick, color=red!80!black] (-4,-7) -- (-4,7);
\draw[thick, color=red!80!black] (4,-7) -- (4,7);
\draw[thick, color=green!60!black] (-5,-7) -- (-5,7);
\draw[thick, color=green!60!black] (5,-7) -- (5,7);
\draw[thick, color=blue!90!black] (-6,-7) -- (-6,7);
\draw[thick, color=blue!90!black] (6,-7) -- (6,7);
\node[draw,circle] at  (0,0) {$1$};

\draw(-3,-7) node [below]{$-1$};
\draw[color=red!80!black](-3,7) node [above]{\small $-1-\tilde\eps_1$};
\draw[color=green!60!black](-5.5,-7) node [below]{\small $-1-2\tilde\eps_1$} ;
\draw[color=blue!90!black](-5.8,7) node [above]{\small $-1-3\tilde\eps_1$};
\draw(3,-7) node [below]{$1$};
\draw[color=red!80!black](3.5,7) node [above]{\small $1+\tilde\eps_1$};
\draw[color=green!60!black](5.5,-7) node [below]{\small $1+2\tilde\eps_1$} ;
\draw[color=blue!90!black](5.8,7) node [above]{\small $1+3\tilde\eps_1$};

\end{tikzpicture}}
\caption{\label{FigSn} Regions of the set $L_n(C_n)$. \textcolor{white}{Lorem ipsum}}
\end{minipage}
\hfill
\begin{minipage}[c]{.48\linewidth}
\resizebox{\textwidth}{!}{
\begin{tikzpicture}[scale=.5]

\foreach \i in {-1,1}{
\foreach \j in {-1,1}{
\fill[color=blue, opacity=.4] (4*\i,6*\j) rectangle (4.5*\i,5*\j);
\fill[pattern=north east lines] (4*\i,6*\j) rectangle (4.5*\i,5*\j);
}}
\foreach \i in {-1,1}{
\fill[color=red, opacity=.4] (4.5*\i,-6) rectangle (6*\i,6);
\fill[pattern=crosshatch] (4.5*\i,-6) rectangle (6*\i,6);
\fill[color=orange, opacity=.4] (3*\i,-5) rectangle (4*\i,5);
\fill[pattern=crosshatch dots] (3*\i,-5) rectangle (4*\i,5);
\fill[color=orange, opacity=.4] (-3,-5*\i) rectangle (3,-3*\i);
\fill[pattern=crosshatch dots] (-3,-5*\i) rectangle (3,-3*\i);
\fill[color=magenta, opacity=.4] (-4,5*\i) rectangle (4,6*\i);
\fill[pattern=horizontal lines] (-4,5*\i) rectangle (4,6*\i);
\fill[color=green, opacity=.4] (4*\i,-5) rectangle (4.5*\i,5);
\fill[pattern=grid] (4*\i,-5) rectangle (4.5*\i,5);

}

\fill[color=black!10!white] (-3,-3) rectangle (3,3);

\draw[thick] (-7,-3) -- (7,-3);
\draw[thick] (-7,3) -- (7,3);
\draw[thick, color=red!80!black] (-7,-4) -- (7,-4);
\draw[thick, color=red!80!black] (-7,4) -- (7,4);
\draw[thick, color=green!60!black] (-7,-5) -- (7,-5);
\draw[thick, color=green!60!black] (-7,5) -- (7,5);
\draw[thick, color=blue!90!black] (-7,-6) -- (7,-6);
\draw[thick, color=blue!90!black] (-7,6) -- (7,6);

\draw[thick] (-3,-7) -- (-3,7);
\draw[thick] (3,-7) -- (3,7);
\draw[thick, color=red!80!black] (-4,-7) -- (-4,7);
\draw[thick, color=red!80!black] (4,-7) -- (4,7);
\draw[thick, color=green!60!black] (-5,-7) -- (-5,7);
\draw[thick, color=green!60!black] (5,-7) -- (5,7);
\draw[thick, color=blue!90!black] (-6,-7) -- (-6,7);
\draw[thick, color=blue!90!black] (6,-7) -- (6,7);
\draw[dotted] (-4.5,-7) -- (-4.5,7);
\draw[dotted] (4.5, -7) -- (4.5,7);
\node[draw,circle] at  (0,0) {$2$};

\draw[color=white](5.8,7) node [above]{\small $1+3\tilde\eps_1$};
\draw[color=white](-5.8,7) node [above]{\small $-1-3\tilde\eps_1$};

\draw (-4.5,-8) node [below]{\small $-1-\tilde\eps_1-\tilde\delta_1$} ;
\draw[->,>=latex](-4.5,-8) -- (-4.5,-7.1);
\draw (4.5,-8) node [below]{\small $1+\tilde\eps_1+\tilde\delta_1$}; 
\draw[->,>=latex](4.5,-8) -- (4.5,-7.1);
\end{tikzpicture}}
\caption{\label{Fig.umaiteracao}Image of the regions after the application of $f_0^{n+k_0}$.}
\end{minipage}

\vspace{10pt}

\begin{minipage}[c]{.48\linewidth}
\resizebox{\textwidth}{!}{
\begin{tikzpicture}[scale=.5]

\clip (-7,-9.1) rectangle (7,7);

\foreach \i in {-1,1}{
\foreach \j in {-1,1}{
\fill[color=blue, opacity=.4] (5*\i,6*\j) rectangle (5.5*\i,5*\j);
\fill[pattern=north east lines] (5*\i,6*\j) rectangle (5.5*\i,5*\j);
}}

\foreach \i in {-1,1}{
\fill[color=orange, opacity=.4] (3*\i,-5) rectangle (5*\i,5);
\fill[pattern=crosshatch dots] (3*\i,-5) rectangle (5*\i,5);
\fill[color=orange, opacity=.4] (-3,3*\i) rectangle (3,5*\i);
\fill[pattern=crosshatch dots] (-3,3*\i) rectangle (3,5*\i);
\fill[color=magenta, opacity=.4] (-5,5*\i) rectangle (5,6*\i);
\fill[pattern=horizontal lines] (-5,5*\i) rectangle (5,6*\i);
\fill[color=green, opacity=.4] (5*\i,-5) rectangle (5.5*\i,5);
\fill[pattern=grid] (5*\i,-5) rectangle (5.5*\i,5);
\fill[color=red, opacity=.4] (5.5*\i,-6) rectangle (7*\i,6);
\fill[pattern=crosshatch] (5.5*\i,-6) rectangle (7*\i,6);
}

\fill[color=black!10!white] (-3,-3) rectangle (3,3);

\draw[thick] (-7,-3) -- (7,-3);
\draw[thick] (-7,3) -- (7,3);
\draw[thick, color=red!80!black] (-7,-4) -- (7,-4);
\draw[thick, color=red!80!black] (-7,4) -- (7,4);
\draw[thick, color=green!60!black] (-7,-5) -- (7,-5);
\draw[thick, color=green!60!black] (-7,5) -- (7,5);
\draw[thick, color=blue!90!black] (-7,-6) -- (7,-6);
\draw[thick, color=blue!90!black] (-7,6) -- (7,6);

\draw[thick] (-3,-7) -- (-3,7);
\draw[thick] (3,-7) -- (3,7);
\draw[thick, color=red!80!black] (-4,-7) -- (-4,7);
\draw[thick, color=red!80!black] (4,-7) -- (4,7);
\draw[thick, color=green!60!black] (-5,-7) -- (-5,7);
\draw[thick, color=green!60!black] (5,-7) -- (5,7);
\draw[thick, color=blue!90!black] (-6,-7) -- (-6,7);
\draw[thick, color=blue!90!black] (6,-7) -- (6,7);
\node[draw,circle] at  (0,0) {$3$};

\draw[dotted] (-5.5,-7) -- (-5.5,7);
\draw[dotted] (5.5, -7) -- (5.5,7);
\draw (-5.5,-8) node [below]{\small $-1-2\tilde\eps_1-\tilde\delta_1$} ;
\draw[->,>=latex](-5.5,-8) -- (-5.5,-7.1);
\draw (5.3,-8) node [below]{\small $1+2\tilde\eps_1+\tilde\delta_1$}; 
\draw[->,>=latex](5.5,-8) -- (5.5,-7.1);
\end{tikzpicture}}
\caption{\label{Fig.afterh1} Image of the regions after the application of $(h_1\circ f_0)^{n+k_0}$.}
\end{minipage}
\hfill
\begin{minipage}[c]{.48\linewidth}
\resizebox{\textwidth}{!}{
\begin{tikzpicture}[scale=.5]

\clip (-7,-9.1) rectangle (7,7);

\foreach \i in {-1,1}{
\fill[color=orange, opacity=.4] (3*\i,-5) rectangle (5*\i,5);
\fill[pattern=crosshatch dots] (3*\i,-5) rectangle (5*\i,5);
\fill[color=orange, opacity=.4] (-3,3*\i) rectangle (3,5*\i);
\fill[pattern=crosshatch dots] (-3,3*\i) rectangle (3,5*\i);
\fill[color=magenta, opacity=.4] (-5,5*\i) rectangle (5,6*\i);
\fill[pattern=horizontal lines] (-5,5*\i) rectangle (5,6*\i);
\fill[color=blue, opacity=.4] (5*\i,5) -- (5*\i,6) ..controls +(\i,0) and +(0,-5).. (5.5*\i,21) -- (5.5*\i,20) ..controls +(0,-5) and +(\i,0).. (5*\i,5);
\fill[pattern=north east lines] (5*\i,5) -- (5*\i,6) ..controls +(\i,0) and +(0,-5).. (5.5*\i,21) -- (5.5*\i,20) ..controls +(0,-5) and +(\i,0).. (5*\i,5);
\fill[color=blue, opacity=.4] (5*\i,-5) -- (5*\i,-6) ..controls +(\i,0) and +(0,-5).. (5.5*\i,10) -- (5.5*\i,11) ..controls +(0,-5) and +(\i,0).. (5*\i,-5);
\fill[pattern=north east lines] (5*\i,-5) -- (5*\i,-6) ..controls +(\i,0) and +(0,-5).. (5.5*\i,10) -- (5.5*\i,11) ..controls +(0,-5) and +(\i,0).. (5*\i,-5);
\fill[color=green, opacity=.4] (5*\i,5) -- (5*\i,-5) ..controls +(\i,0) and +(0,-5).. (5.5*\i,10) -- (5.5*\i,20) ..controls +(0,-5) and +(\i,0).. (5*\i,5);
\fill[pattern=grid] (5*\i,5) -- (5*\i,-5) ..controls +(\i,0) and +(0,-5).. (5.5*\i,10) -- (5.5*\i,20) ..controls +(0,-5) and +(\i,0).. (5*\i,5);
}

\fill[color=black!10!white] (-3,-3) rectangle (3,3);

\draw[thick] (-7,-3) -- (7,-3);
\draw[thick] (-7,3) -- (7,3);
\draw[thick, color=red!80!black] (-7,-4) -- (7,-4);
\draw[thick, color=red!80!black] (-7,4) -- (7,4);
\draw[thick, color=green!60!black] (-7,-5) -- (7,-5);
\draw[thick, color=green!60!black] (-7,5) -- (7,5);
\draw[thick, color=blue!90!black] (-7,-6) -- (7,-6);
\draw[thick, color=blue!90!black] (-7,6) -- (7,6);

\draw[thick] (-3,-7) -- (-3,7);
\draw[thick] (3,-7) -- (3,7);
\draw[thick, color=red!80!black] (-4,-7) -- (-4,7);
\draw[thick, color=red!80!black] (4,-7) -- (4,7);
\draw[thick, color=green!60!black] (-5,-7) -- (-5,7);
\draw[thick, color=green!60!black] (5,-7) -- (5,7);
\draw[thick, color=blue!90!black] (-6,-7) -- (-6,7);
\draw[thick, color=blue!90!black] (6,-7) -- (6,7);
\node[draw,circle] at  (0,0) {$4$};

\draw[color=white] (-4.5,-8) node [below]{\small $-1-1.5\tilde\eps_1$} ;
\draw[color=white] (4.5,-8) node [below]{\small $1+1.5\tilde\eps_1$}; 
\end{tikzpicture}}
\caption{\label{Fig.afterh2}Image of the regions after the application of $(h_2\circ h_1\circ f_0)^{n+k_0}$.}
\end{minipage}
\end{figure}

We define the \emph{blue region}, as the union of the blue boxes that intersect the support of the perturbation $h_2\circ h_1$. More precisely, we set, for $h=\ell,r$,
\[n_1(h) = \begin{cases}
n_0+1 & \quad \text{if } h=\ell\\
n_0+2 & \quad \text{if } h=r,
\end{cases}\]
and
\begin{equation}
\label{e.bluetower}
\Ba = \bigcup_{\substack{h\in\{\ell,r\}\\ v\in\{t,b\}}}\, \bigcup_{n\ge n_1(h)}\cA^B_{h,v}(n)
\end{equation}

As it has a special place in our arguments, we shall also call the blue region as the \textit{blue tower}.

\subsubsection{Returns and persistent points}
Before moving to the main statement of this paragraph, we need some definitions.

\begin{definition}[First return map]\label{defg1} We define the \emph{first return map} $g_1$ of $f_1$ as the map $g_1:\bigcup_{n\geq n_0}\tilde{S}_n\to\R^2$ with $g_1|_{\tilde{S}_n}=f_1^{n+k_0}$. 
\end{definition}

This definition plays, for $f_1$, the role of Definition~\ref{defgzero} for $f_0$. The main difference is that some points in $\tilde{S}_n$ do not actually return to the extended stable tower\footnote{This phenomenon has already been shown in Lemma~\ref{l.bigbacia2}.}. This is just a reflect of the fact that the dynamics of $f_1$ is intrinsically more complicated than that of $f_0$. The next definition is devised precisely to deal with this.  

\begin{definition}[Persistent points]\label{defpersistent}
A point $p\in\mathcal{A}_{h,v}^{B}(n)$ in the blue tower $\Ba$ (meaning that $n\ge n_1(h)$) is called \emph{$1$-persistent} if $g_1(p)\in\Ba$.
We denote the set of 1-persistent points by $P_1(h,v,n)\subset \cA^B_{\ell,t}(n)  \subset\Ba$, and define the set $P_{k}(h,v,n)$ of \emph{$k$-persistent points} by induction:
$$P_{k}(h,v,n) = \big\{p\in \mathcal{A}_{h,v}^{B}(n) \,;\ \textrm{the}\:1^{st}\:\textrm{return of $p$ is $k-1$-persistent}\big\}.$$
Finally, we consider $\Gamma\eqdef\bigcup_{h,v,n}\bigcap_{k\in\nt}P_k(h,v,n)$ the set of \emph{persistent points}.
\end{definition}


By its very definition, $\Gamma$ is invariant under the first return map $g_1$.

\begin{definition}
	We define the map $\overline g$ on the triplets $(h,v,n)$ with $n\ge n_1(h)$ (or equivalently on the sets $\cA^{B}_{h,v}(n)$) that corresponds to the application of $g_1$ (see Figures~\ref{FigSn} to \ref{Fig.afterh2}):
	\[\overline g(h,v,n) = \left\{\begin{array}{ll}
	(\ell,t, 2n) & \text{ if } (h,v) = (r,t)\\
	(\ell,b, 2n) & \text{ if } (h,v) = (\ell,t)\\
	(r,t, 2n) & \text{ if } (h,v) = (\ell,b)\\
	(r,b, 2n) & \text{ if } (h,v) = (r,b).
	\end{array}\right.\]
Denoting $\overline g(h,v,n) = (\overline h,\overline v, \overline n)$, this allows to define a relation on the triplets
\[\big\{(h,v,n) \in \{l,r\}\times\{t,b\}\times \nt;\, n\ge n_1(h)\big\}\]
(or equivalently on the sets $\cA^{B}_{h,v}(n)$) by 
\[(h',v',n') \prec (h,v,n) \iff h'=\overline h \text{ and } \big(\overline n<n' \text{ or } n'=\overline n \text{ and } \overline v= b \big).\]
\end{definition}

This relation morally means ``$(h',v',n') \prec (h,v,n)$ if the image $f_1^{n+k_0}(\cA^{B}_{h,v}(n))$ intersects the rectangle $\cA^{B}_{h',v'}(n')$''. More precisely, we have the following lemma.

\begin{lema}\label{LemIntersecBlue}
For any $(h,v,n)$ and $(h',v',n')$, one has 
\[(h',v',n') \prec (h,v,n) \quad \iff\quad
\cA^{B}_{h',v'}(n') \cap f_1^{n+k_0}(\cA^{B}_{h,v}(n))\neq\emptyset.\]
\end{lema}

\begin{proof}
We still denote $\overline g(h,v,n) = (\overline h,\overline v, \overline n)$. It can be easily checked, using (iv) of Proposition~\ref{buildTowers} that the set $f_0^{n+k_0}(\cA^{B}_{h,v}(n))$ is a rectangle. For instance (see Figure~\ref{Fig.umaiteracao}), 
\[f_0^{n+k_0}(\cA^{B}_{\ell,b}(n)) = [b+\eps_1,b+\eps_1+\delta_1] \times [\sigma^{-2n}(a-3\eps_1),\sigma^{-2n}(a-2\eps_1)].\]
Hence, by the definition of $h_1$ (see Figure~\ref{Fig.afterh1}), 
\[(h_1\circ f_0)^{n+k_0}(\cA^{B}_{\ell,b}(n)) = [b+2\eps_1,b+2\eps_1+\delta_1] \times [\sigma^{-2n}(a-3\eps_1),\sigma^{-2n}(a-2\eps_1)].\]
The application of $h_2$ is a bit more complicated. The image of any horizontal sub-segment of this set by $h_2$ is a graph over $[b+2\eps_1, b+2\eps_1+\delta_1]$: for any $y_0\in [\sigma^{-2n}(a-3\eps_1),\sigma^{-2n}(a-2\eps_1)]$, and any $x\in [b+\eps_1, b+\eps_1+\delta_1]$, by \eqref{EqDefH2},
\begin{equation*}
h_2(x,y_0) = \Big(x\,,\  y_0 + \varphi_2(x) \big(\xi_2(y_0)-y_0\big) \Big).
\end{equation*}
But in restriction to $[b+2\eps_1, b+2\eps_1+\delta_1]$, the map $\varphi_2$ increases from 0 to 1, hence when $x$ goes from $b+2\eps_1$ to $b+2\eps_1+\delta_1$, the second coordinate of $h_2(x,y_0)$ increases from $y_0$ to $\xi_2(y_0)$.

Hence, the image by $h_2$ of any horizontal sub-segment of $(h_1\circ f_0)^{n+k_0}(\cA^{B}_{\ell,b}(n))$ meets all the rectangles $\cA^{B}_{h',v'}(n')$ for $(h',v',n') \prec (h,v,n)$, but none of the rectangles $\cA^{B}_{h'',v''}(n'')$ for $(h'',v'',n'') \not\prec (h,v,n)$. The proof of the lemma is similar in the other cases for $h$ and $v$.
\end{proof}

Most of this section and next one is devoted to a detailed understanding of the set of persistent points. We establish below a set-theoretic equation that comes naturally from the recursive character of the definition. Later in the paper we shall refine this lemma in topological and geometrical terms.

\begin{lema}
\label{l.perkmais1}
For any $k\in\nt$, any $(h,v)\in\{\ell,r\}\times \{t,b\}$ and any $n\ge n_1(h)$, denoting $\theta = (h,v,n)$, one has
\begin{equation*}
\label{e.didactique}
P_{k+1}(\theta)=\bigcup_{\eta\prec\theta}f_1^{-n-k_0}\left(P_k(\eta)\cap f_1^{n+k_0}(\cA^B_{h,v}(n))\right).
\end{equation*}
\end{lema}

\begin{proof}
Assume that $y\in f_1^{-n-k_0}\left(P_k(\eta)\cap f_1^{n+k_0}(\cA^B_{h,v}(n))\right)$, for some $\eta\prec\theta$. Notice that $y\in \cA^B_{h,v}(n)$ and that $g_1(y)=f_1^{n+k_0}(y)\in P_k(\eta)$, and so $g_1(y)$ is $k$ persistent. This proves that $y$ is $k+1$ persistent. Reciprocally, assume $y\in P_{k+1}(\theta)$. Then, by definition $g_1(y) = f_1^{n+k_0}(y)$ is $k$ persistent. In particular, there must exist some $\eta=(h',v',n')$ such that $g_1(y)\in P_k(\eta)\subset\cA^B_{h',v'}(n')$. Since moreover, $g_1(y)\in f_1^{n+k_0}(\cA^B_{h,v}(n))$, we deduce by Lemma~\ref{LemIntersecBlue} that $\eta\prec\theta$ and $g_1(y)\in P_k(\eta)\cap f_1^{n+k_0}(\cA^B_{h,v}(n))$, concluding.  
\end{proof}

\subsubsection{Dynamics of non-persistent points} The main result of this paragraph is the lemma below, which completely determines the dynamics inside the $\eps_1$-tower $\mathcal C$. The final conclusion is that the only points in this tower which are not in $W^s(\mathcal Q)$ are the persistent points. The complete statement is summarized in Figure~\ref{GraphTrans}: it describes all possible transitions between the ``coloured" regions under forward iteration of $f_1$. 

The proof consists in analysing separately the iterations of each ``coloured" region, and is depicted in Figures~\ref{FigSn}, \ref{Fig.umaiteracao}, \ref{Fig.afterh1} and \ref{Fig.afterh2}. 

\begin{lema}\label{l.allerverslesbleu} Every point in the extended tower $\mathcal C$ which does not belongs to $W^s(\mathcal Q)$ nor to the stable tower $\mathcal S$ has a forward iterate which is a persistent point:
$$\mathcal{C}\setminus\mathcal{S}\ \subset\ W^s(\mathcal Q)\cup\Big[\bigcup_{n\in\nt}f_1^{-n}(\Gamma)\Big].$$More precisely, for every $h\in\{\ell,r\}$ and $v\in\{t,b\}$, if $n\geq n_0$ then (see Figure \ref{GraphTrans}):
\begin{enumerate}
\item  
\emph{The red region is contained in $W^s(\mathcal{Q})$, i.e., for any $n\ge n_0$,}
$$\cA^R_{v}(n)\subset W^s(\mathcal{Q}).$$ 
\item 
\emph{Every element of the pink region is eventually mapped inside the red region, and in particular is contained in $W^s(\mathcal{Q})$: for any $n\ge n_0$,}
$$\cA^P_{h}(n)\subset\bigcup_{\tilde{v}\in\{t,b\}} g_1^{-1}\left(\cA^R_{\tilde{v}}(2n)\right)\subset W^s(\mathcal{Q}).$$
\item
\emph{Every element of the orange region is eventually mapped either inside the green region or inside the red region:}
$$\cA^O(n)\subset\bigcup_{\tilde{v}\in\{t,b\}}\bigcup_{m\in\mathbb{N}} g_1^{-m}\left(\cA^G_{\tilde{v}}(2^m n)\cup \cA^R_{\tilde{v}}(2^m n)\right).$$ 
\item 
\emph{Every element of the green region is eventually mapped either inside the blue tower, or in $W^s(\mathcal{Q})$}:
$$\cA^G_{v}(n)\subset W^s(\mathcal{Q})\cup\bigcup_{m>0}f_1^{-m}(\Ba).$$
\item\label{azullemadascores} \emph{Every element of the blue tower which is not in $W^s(\mathcal Q)$ is a persistent point, i.e.}
$$\Ba\subset W^s(\mathcal{Q})\cup\Gamma.$$
\end{enumerate}
\end{lema}

\begin{figure}[ht]
\begin{center}
\begin{tikzpicture}[scale=.55]
\node[draw,circle,minimum height=1.2cm, fill=gray!25!white] (A) at (8,4) {$W^s(\mathcal{Q})$};
\node[draw,circle,minimum height=1.2cm, fill=orange!25!white] (B) at (0,0) {$\cA^{O}$};
\node[draw,circle,minimum height=1.2cm, fill=green!25!white] (C) at (4,0) {$\cA^{G}$};
\node[draw,circle,minimum height=1.2cm, fill=blue!25!white] (D) at (8,0) {$\cA^{B}$};
\node[draw,circle,minimum height=1.2cm, fill=magenta!20!white] (E) at (0,4) {$\cA^{P}$};
\node[draw,circle,minimum height=1.2cm, fill=red!30!white] (F) at (4,4) {$\cA^{R}$};

\draw[->,>=latex,shorten >=3pt, shorten <=3pt] (B) to (C);
\draw[->,>=latex,shorten >=3pt, shorten <=3pt] (C) to (D);
\draw[->,>=latex,shorten >=3pt, shorten <=3pt] (E) to (F);
\draw[->,>=latex,shorten >=3pt, shorten <=3pt] (F) to (A);
\draw[->,>=latex,shorten >=3pt, shorten <=3pt] (D) to (E);
\draw[->,>=latex,shorten >=3pt, shorten <=3pt] (D) to (F);
\draw[->,>=latex,shorten >=3pt, shorten <=3pt] (D) to (A);
\draw[->,>=latex,shorten >=3pt, shorten <=3pt] (C) to (E);
\draw[->,>=latex,shorten >=3pt, shorten <=3pt] (C) to (F);
\draw[->,>=latex,shorten >=3pt, shorten <=3pt] (C) to (A);
\draw[->,>=latex,shorten >=3pt, shorten <=3pt] (B) to (F);
\draw[->,>=latex,shorten >=3pt, shorten <=3pt] (B) to[loop left,looseness=8,min distance=10mm]node[midway, left]{finite} (B);
\draw[->,>=latex,shorten >=3pt, shorten <=3pt] (D) to[loop right,looseness=8,min distance=10mm] (D);
\end{tikzpicture}
\end{center}
\caption{Lemma \ref{l.allerverslesbleu}: any point eventually falls in $W^s(\mathcal{Q})$, or remains in the union of the blue rectangles $\cA^B$. The arrows represent the action of the first return map $g_1$.}\label{GraphTrans}
\end{figure}
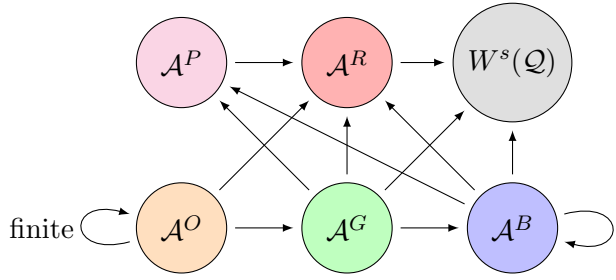

\begin{proof}[Proof of Lemma~\ref{l.allerverslesbleu}]
As the proof is similar for some regions, we will only give it for points (1), (3) and (4).
\begin{enumerate}
\item[(1)] Let us give a proof for a set $\cA^{R}_{t}(n)$, with $n\ge n_0$. As in the proof of Lemma~\ref{l.criandopoco}, we see that the sets $f_1^j(\cA^{R}_{t}(n))$ are disjoint from the supports of $h_1$ and $h_2$, for $j=1,\dots,n-2+k_0$. As can be seen in Figure~\ref{Fig.umaiteracao}, one has
\[f_0f_1^{n-2+k_0}(\cA^{R}_{t}(n)) = [\sigma^2(a-3\eps_1),\sigma^2(a-\eps_1-\delta_1)] \times [\sigma^{2n}(a-3\eps_1),\sigma^{2n}(b+3\eps_1)].\]Lemma~\ref{rem.hdoisnoq} then implies that 
\[h_2h_1f_0f_1^{n-2+k_0}(\cA^{R}_{t}(n)) = f_1^{n-1+k_0}(\cA^{R}_{t}(n)) \subset\mathcal Q.\]
The argument is analogous for $\cA^{R}_b(n)$.
\item[(2)] Reasoning as above one sees that $f_1^{n_0+k_0}\big(\cA^P_{h}(n)\big)\subset \cA^R_t(2n)\cup\cA^R_b(2n)$ (see Figure~\ref{Fig.afterh2}) and thus item (1) implies item (2).
\item[(3)] By the same arguments as before, as can be seen in Figures~\ref{FigSn} to~\ref{Fig.afterh2}, the orange region is mapped by $g_1$ into the union of the orange region with the green and the red regions. So it suffices to prove that any point of the orange region eventually leaves the orange region. Note that this is the only region which is not closed: any point of it is at a positive distance to $S_n$. Take $(x,y)\in \cA^O$ and $n\ge n_0$, and suppose by contradiction that the positive orbit of $L_n^{-1} (x,y)$ under $g_1$ stays forever in the orange region. In particular, it implies that $g_1^k(L_n^{-1} (x,y)) \in C_{2^k n}$ for every $k\ge 0$. Let us treat the case where $y<-1$, the case $y>1$ being similar and the remaining case $|x|>1$ being deduced from this one by an application of the return map $g_0$ associated to $f_0$.
We have
\[L_{2n} g_0 L_n^{-1} (x,y) = (-y,x) \eqdef L_{2n}(\tilde y,\tilde x).\]
By the definition of $h_1$, we deduce that 
\[L_{2n} g_1 L_n^{-1} (x,y) - L_{2n} g_0 L_n^{-1} (x,y) = (\xi_1(\tilde y)-\tilde y,0).\]
Iterating this process, we can easily prove that for any $j\ge 0$, the first coordinate of $L_{2^{1+4j}n} g_1^{1+4j} L_n^{-1} (x,y) - L_{2^{1+4j}n} g_0^{1+4j} L_n^{-1} (x,y)$ is bigger than $\xi_1^j(\tilde y)-\tilde y$. This is because every 4 returns in a box $C_m$, the point suffers the same perturbation as we have just seen, and at every return in such a box, the perturbation do not decrease neither the $x$-distance nor the $y$-distance to $S_m$. However, since $g_0$ (after rescaling) only rotates by $\pi/2$ and since $h_1$ only slides horizontally pushing points away from the stable box, we deduce from $g_1^k(L_n^{-1} (x,y)) \in C_{2^k n}$ for every $k\ge 0$, that the first coordinate of$$L_{2^{1+4j}n} g_1^{1+4j} L_n^{-1} (x,y) - L_{2^{1+4j}n} g_0^{1+4j} L_n^{-1} (x,y)$$ must be smaller than $2\tilde{\eps}_1=4\eps_1/(b-a)$, which by our choice of $\eps_1$ in \eqref{e.oitoepsilon}, is smaller than $\sigma-\tilde{b}<q_1-\tilde{y}$, where $q_1\in(\sigma,\sigma^2)$ is the unique attractive fixed point of $\xi_1$ (see hypothesis (H2) page \pageref{H3}). Therefore, we have proved that 
\[
\xi_1^j(\tilde{y})-\tilde{y}<\sigma-\tilde{b}<q_1-\tilde{y},
\] 
for every $j>0$, which is a contradiction because $\xi_1^j(\tilde{y})\to q_1$.

\item[(4)] By the same arguments as above in this lemma, and as in the proof of Lemma~\ref{LemIntersecBlue} (see also Figure~\ref{Fig.afterh2}), one can see that 
\[f_1^n(\cA^G_{b}(n)) \subset \bigcup_{m\ge n_0} \bigcup_{v=t,b}\left(\cA^R_{v}(m) \cup \cA^P_{r}(m) \cup \cA^B_{r,v}(m)\right) \ \cup\ \mathcal C^\complement.\]
The same arguments as in the proof of Lemma~\ref{l.bigbacia2} show that
\[f_1^n(\cA^G_{b}(n)) \cap C^\complement \subset W^s(\mathcal Q).\]
Note that even if the points of this intersection seem to enter Case 3 of the proof of Lemma~\ref{l.bigbacia2}, this lemma deals with what happens before the application of the perturbation $h_2\circ h_1$ while here this perturbation has already been applied, so that the first return of a point of this intersection in the region analysed by Lemma~\ref{l.bigbacia2} enters Case 1.

The same kind of arguments work for $\cA^G_{t}(n)$ and point (5) of the lemma.
\end{enumerate}	
\end{proof}

Lemma~\ref{l.allerverslesbleu} is the final piece which allows us to conclude Proposition~\ref{p.bluetrichotomy}.

\begin{proof}[Proof of Proposition~\ref{p.bluetrichotomy}]
Let $p\in\mathcal{L}_r^i\setminus\big(\mathcal{S}\cup W^s(\mathcal Q)\big)$ and assume that its positive orbit under $f_1$ meets the set $[a-3\eps_1,\sigma^2(a-3\eps_1)]\times[0,\sigma^{-n_0-2}b]$. By Lemma~\ref{l.bigbacia2} we deduce that $p$ has a positive iterate under $f_1$ inside the $\eps_1$-tower $\mathcal{C}$. Since $p\notin W^s\big(\mathcal{Q}\big)$ we deduce that $p$ has a positive iterate which is a persistent point. Thus, $p$ belongs to the orbit of $\Gamma$, concluding. 
\end{proof}

\subsection{Dynamics of points of $\mathcal{L}^e$\label{SubsecExt}}

We now explain how results of this section actually translate for points of the exterior $\mathcal{L}^e$ of the figure-eight attractor.

A statement similar to Proposition~\ref{p.bluetrichotomy} also holds for the points $p\in \mathcal{L}^e$ meeting $[a-3\eps_1,\,\sigma^2(a-3\eps_1)]\times [-\sigma^{-n_0-2}b,0]$, replacing the tower $\mathcal{S}$ by the exterior tower $\mathcal {S}^e$ and $\Gamma$ by another set $\Gamma^e$ (whose definition we give below): there is the trichotomy
\begin{enumerate}
\item $p\in\mathcal S^e$;
\item $p\in W^s(\mathcal Q) \cup W^s(-\mathcal Q)$;
\item $p\in\bigcup_{n\in\Z}f_1^{n}(\Gamma^e)$.
\end{enumerate}
Remark that we had to add the symmetric of $\mathcal{Q}$ with respect to $O$, because the points of $\mathcal{L}^e$ that come in the neighbourhood $\mathcal{V}$ of $O$ return alternatively to $\mathcal{V} \cap (\R_+\times \R_-)$ and $\mathcal{V} \cap (\R_-\times \R_+)$. From now on we will make the abuse of language that every point always returns to $\mathcal{V} \cap (\R_+\times \R_-)$, and we will quotient implicitly by the relation $x\sim -x$ on $\R^2$. 

We now define the set $\Gamma^e$ appearing in the previous trichotomy. First, we define the set $\Ba^e$ as the symmetric of $\Ba$ with respect to the horizontal axis. However, the dynamics on it is a bit different: one has to replace the order $\prec$ by the order $\preceq$ defined as follows. Consider the map $L^e_n$ defined by $L^e_n = L_n \circ \operatorname{diag}(1,-1)$, and the blue exterior boxes
\[\cA^{B}_{h,v}(n)^e = (L_n^e)^{-1}\big(\cA^{C}_{h,v}\big)\]
(remark that with this definition, the top rectangles are below the bottom rectangles). Then, $\preceq$ is defined by
\[(h',v',n') \preceq (h,v,n) \iff h'=\overline h \text{ and } \big(\overline n>n' \text{ or } n'=\overline n \text{ and } \overline v= t \big).\]
This allows to define the set $\Gamma^e$, similarly to the set $\Gamma$, as the set of points whose whole positive orbit under the first return map is included in $\Ba$. 

Nevertheless, with this new order, one gets a lemma similar to Lemma~\ref{l.allerverslesbleu} for regions of $\mathcal{L}^e$, the points exiting the extended towers $C_n^e\notin\Gamma^e$ being mapped either in $W^s_{f_1}(O)$, or in $\mathcal {L}^i_r \cup \mathcal {L}^i_\ell$ (and in particular, some of these points being mapped in $\mathcal{Q}$). Altogether, we conclude that the statistical basin $\mathcal{B}_{f_1}(O)$ is contained in the union of the stable tower with the orbit of all persistent points (those in the interior of $\mathcal{L}$, the set $\Gamma$ and its symmetric copy in the third quadrant) and those in the exterior of $\mathcal{L}$ (the set $\Gamma^e$). In the next section, we devote our attention to the geometry, topology and dynamics of the sets of persistent points.

\section{The set of persistent points}\label{SecGeom}

As its title indicates, this section is devoted to understand in some detail the set $\Gamma$ of \emph{persistent points} (see Definition \ref{defpersistent}). The main reason for this study is that some points that belong to $\Gamma$ also belong to the statistical basin $\cB_{f_1}(\delta_O)$ (but not all of them, for $\Gamma$ has infinitely many periodic points as we will see in Proposition \ref{PropCoding} below). Therefore, a difficulty we face is that the orbit under $f_1$ of this family of points (both forward and backward) could be dense in some open set of the plane. In this section we will show that this is not the case, by proving that the whole orbit of $\Gamma$ is a nowhere dense set in $\R^2$ of zero Lebesgue measure. In addition, we will be able to describe the action of $f_1$ in $\Gamma$ by means of symbolic dynamics (see Proposition \ref{PropCoding}). The main result of this section is the following. 

\begin{prop}\label{l.persmedidanula}
The set\, $\mathcal{O}_{f_1}(\Gamma)=\bigcup_{n\in\Z}f_1^{n}(\Gamma)$ has zero Lebesgue measure and is nowhere dense.
\end{prop}

The same statement holds for the symmetric copy of $\Gamma$ in the third quadrant and for $\Gamma^e$. In both cases, the proof is similar to the one we present in this section. For the set $\Gamma^e$, only one step needs adjustment, which is the Inclination Lemma~\ref{LemInclination} and we point out precisely how to obtain the Inclination Lemma for the set $\Gamma^e$.

 Let us give a rough sketch of why Proposition~\ref{l.persmedidanula} is true. Remember that a persistent point is a point which never leaves the blue tower under forward iteration of the \emph{first return map}. The main idea, which is depicted in Figure~\ref{fig.interesting}, is that the positive iteration of the rectangles that compose the blue tower has a Markovian structure, so that a point which never leaves it lies in a decreasing intersection of subsets of some rectangle. Section \ref{SubsecGeom} is then devoted to make this statement rigorous; it involves (among some other reasoning) an estimation of the inclination of the images of the small rectangles by the perturbation. This gives a topological description of the set of persistent points. Additional work is performed in Section \ref{SubsecSize} to estimate the size of the set of persistent points. More precisely, we obtain distortion estimates in order to show simultaneously that this nested sequence has measure which goes to zero and converges to a nowhere dense set. Finally, in Section \ref{secnowhere} we prove that $\bigcup_{n\in\Z}f_1^{n}(\Gamma)$ is nowhere dense. Incidentally, as a by-product of our arguments, we obtain a semi-conjugacy of the dynamics in the set of persistent points with a shift ``of finite type'' over an infinite alphabet, as given by the following result.

\begin{prop}[Coding]\label{PropCoding}
The set $\Gamma$ has zero Lebesgue measure and its closure
\[\overline\Gamma = \Gamma \cup \big([a-3\eps_1,a-2\eps_1] \times\{0\}\big) \cup \big([b+2\eps_1,b+3\eps_1] \times\{0\}\big)\]
is homeomorphic to the product of a Cantor set with a segment. Moreover, there is a continuous map:
\[\Phi : \Gamma \longrightarrow \big\{(h,v,n);\, h\in\{\ell,r\},v\in\{t,b\},n\ge n_1(h)\big\}\]
such that for every $x\in \Gamma$, if we denote $\Phi(x) = (h_k,v_k,n_k)_{k\in\nt}$, one has $g_1^k(x) \in \cA^{B}_{h_k,v_k}(n_k)$ for any $k\in\nt$ (recall that $g_1$ is the first return map of $f_1$ in $\Gamma$, see Definition~\ref{defpersistent}). This coding $\Phi$ semi-conjugates $g_1$ with the one-sided subshift $\sigma$ over the alphabet $\{\ell,r\}\times\{t,b\}\times\nt$ given by the (finite) transitions:
	\[(h,v,n) \to (h',v',n') \quad \text{iff}\quad (h',v',n')\prec (h,v,n).\]
Finally, for any $(h_k,v_k,n_k)_{k\in\nt}$ in this subshift, the preimage $\Phi^{-1}(h_k,v_k,n_k)$ is homeomorphic to a segment, and is the graph of a Lipschitz map over $[a-3\eps_1,a-2\eps_1] \times\{0\}\big)$ or $\big([b+2\eps_1,b+3\eps_1] \times\{0\}$ (with Lipschitz constant smaller than $1/2$ when renormalizing by $L_n$).
\end{prop}

In particular, we obtain the following interesting result.

\begin{coro} The map $f_1|_{\Gamma}$ has positive topological entropy and infinitely many periodic orbits.
\end{coro}

A statement similar to Proposition~\ref{PropCoding} holds for points of the exterior set $\Gamma^e$, replacing the order $\prec$ by the order $\preceq$ (see Subsection~\ref{SubsecExt}). Note that, contrary to what happens for $\Gamma$, the set $\Gamma^e$ contains no periodic point. In fact, the dynamics of $f_1$ on $\Gamma^e$ is wandering: the return by $g_0$ plus the perturbation $h_2\circ h_1$ make the orbit closer and closer to the stable manifold $W^{s}_{f_1}(O)$. In particular, $\Gamma^e\subset \mathcal{B}_{f_1}(\delta_O)$; this is the reason why this set cannot be neglected in our study (in the same way $\Gamma$ cannot be neglected).  
\medskip

The proof of Proposition~\ref{PropCoding} will be given in Section~\ref{SubsecSize}, as a consequence of the arguments we use in the proof of Proposition~\ref{l.persmedidanula}.  


\begin{figure}[bht]

\begin{minipage}[c]{.48\linewidth}
\resizebox{\textwidth}{!}{
\begin{tikzpicture}[scale=.6]
	
\clip (-7,-7) rectangle (7,7);

\fill[color=blue, opacity=.4] (5,4) rectangle (6,4.5);
\fill[pattern=north east lines] (5,4) rectangle (6,4.5);

\fill[color=green, opacity=.4] (5,5) -- (5,6) ..controls +(1,0) and +(0,-5).. (5.5,21) -- (5.5,20) ..controls +(0,-5) and +(1,0).. (5,5);
\fill[pattern=north west lines] (5,5) -- (5,6) ..controls +(1,0) and +(0,-5).. (5.5,21) -- (5.5,20) ..controls +(0,-5) and +(1,0).. (5,5);

\fill[color=red, opacity=.4] (-5,4) rectangle (-6,4.5);
\fill[pattern=north east lines] (-5,4) rectangle (-6,4.5);

\fill[color=yellow, opacity=.4] (5,-5) -- (5,-6) ..controls +(1,0) and +(0,-5).. (5.5,10) -- (5.5,11) ..controls +(0,-5) and +(1,0).. (5,-5);
\fill[pattern=north west lines] (5,-5) -- (5,-6) ..controls +(1,0) and +(0,-5).. (5.5,10) -- (5.5,11) ..controls +(0,-5) and +(1,0).. (5,-5);

\fill[color=yellow, opacity=.4] (-5,-4) rectangle (-6,-4.5);
\fill[pattern=north east lines] (-5,-4) rectangle (-6,-4.5);

\fill[color=blue, opacity=.4] (-5,5) -- (-5,6) ..controls +(-1,0) and +(0,-5).. (-5.5,21) -- (-5.5,20) ..controls +(0,-5) and +(-1,0).. (-5,5);
\fill[pattern=north west lines] (-5,5) -- (-5,6) ..controls +(-1,0) and +(0,-5).. (-5.5,21) -- (-5.5,20) ..controls +(0,-5) and +(-1,0).. (-5,5);

\fill[color=green, opacity=.4] (5,-4) rectangle (6,-4.5);
\fill[pattern=north east lines] (5,-4) rectangle (6,-4.5);

\fill[color=red, opacity=.4] (-5,-5) -- (-5,-6) ..controls +(-1,0) and +(0,-5).. (-5.5,10) -- (-5.5,11) ..controls +(0,-5) and +(-1,0).. (-5,-5);
\fill[pattern=north west lines] (-5,-5) -- (-5,-6) ..controls +(-1,0) and +(0,-5).. (-5.5,10) -- (-5.5,11) ..controls +(0,-5) and +(-1,0).. (-5,-5);

\fill[color=black!10!white] (-3,-3) rectangle (3,3);

\draw[thick] (-7,-3) -- (7,-3);
\draw[thick] (-7,3) -- (7,3);
\draw[thick, color=red!80!black] (-7,-4) -- (7,-4);
\draw[thick, color=red!80!black] (-7,4) -- (7,4);
\draw[thick, color=green!60!black] (-7,-5) -- (7,-5);
\draw[thick, color=green!60!black] (-7,5) -- (7,5);
\draw[thick, color=blue!90!black] (-7,-6) -- (7,-6);
\draw[thick, color=blue!90!black] (-7,6) -- (7,6);

\draw[thick] (-3,-7) -- (-3,7);
\draw[thick] (3,-7) -- (3,7);
\draw[thick, color=red!80!black] (-4,-7) -- (-4,7);
\draw[thick, color=red!80!black] (4,-7) -- (4,7);
\draw[thick, color=green!60!black] (-5,-7) -- (-5,7);
\draw[thick, color=green!60!black] (5,-7) -- (5,7);
\draw[thick, color=blue!90!black] (-6,-7) -- (-6,7);
\draw[thick, color=blue!90!black] (6,-7) -- (6,7);
\end{tikzpicture}}
\captionsetup{width=.95\linewidth}
\caption{\label{fig.interesting}Markovian structure implying the coding described in Proposition~\ref{PropCoding}.}
\end{minipage}
\hfill
\begin{minipage}[c]{.48\linewidth}
\resizebox{\textwidth}{!}{
\begin{tikzpicture}[scale=1]
\usetikzlibrary{calc}

\filldraw[fill=yellow!30!white] (-2,-1) rectangle (2,1);

\fill[color=green, opacity=.4, cm={0 ,-0.5 ,2 ,0 ,(0 cm,0 cm)}] (-1.4,1) -- (-1.1,1) to[in=93, out=-60] (-.8,-1) -- (-1.1,-1) to[out=93, in=-60] (-1.4,1);
\fill[pattern=north east lines, cm={0 ,-0.5 ,2 ,0 ,(0 cm,0 cm)}] (-1.4,1) -- (-1.1,1) to[in=93, out=-60] (-.8,-1) -- (-1.1,-1) to[out=93, in=-60] (-1.4,1);
\fill[color=blue, opacity=.4, cm={0 ,-0.5 ,2 ,0 ,(0 cm,0 cm)}] (-.4,1) -- (-.2,1) to[in=91, out=-89] (-.1,-1) -- (-.3,-1) to[out=91, in=-89] (-.4,1);
\fill[pattern=north east lines, cm={0 ,-0.5 ,2 ,0 ,(0 cm,0 cm)}] (-.4,1) -- (-.2,1) to[in=91, out=-89] (-.1,-1) -- (-.3,-1) to[out=91, in=-89] (-.4,1);

\fill[color=green, opacity=.4, cm={0 ,-0.5 ,2 ,0 ,(0 cm,0 cm)}] (0.5,1) -- (.55,1) -- (.6,-1) -- (.5,-1) -- (0.5,1);
\fill[pattern=north east lines, cm={0 ,-0.5 ,2 ,0 ,(0 cm,0 cm)}] (0.5,1) -- (.55,1) -- (.6,-1) -- (.5,-1) -- (0.5,1);
\fill[color=blue, opacity=.4, cm={0 ,-0.5 ,2 ,0 ,(0 cm,0 cm)}] (0.7,1) rectangle (.75,-1);
\fill[pattern=north east lines, cm={0 ,-0.5 ,2 ,0 ,(0 cm,0 cm)}] (0.7,1) rectangle (.75,-1);

\fill[color=green, opacity=.4, cm={0 ,-0.5 ,2 ,0 ,(0 cm,0 cm)}] (0.9,1) rectangle (.93,-1);
\fill[pattern=north east lines, cm={0 ,-0.5 ,2 ,0 ,(0 cm,0 cm)}] (0.9,1) rectangle (.93,-1);
\fill[color=blue, opacity=.4, cm={0 ,-0.5 ,2 ,0 ,(0 cm,0 cm)}] (0.95,1) rectangle (.97,-1);
\fill[pattern=north east lines, cm={0 ,-0.5 ,2 ,0 ,(0 cm,0 cm)}] (0.95,1) rectangle (.97,-1);

\fill[color=green, opacity=.4, cm={0 ,-0.5 ,2 ,0 ,(0 cm,0 cm)}] (1.05,1) rectangle (1.06,-1);
\fill[pattern=north east lines, cm={0 ,-0.5 ,2 ,0 ,(0 cm,0 cm)}] (1.05,1) rectangle (1.06,-1);
\fill[color=blue, opacity=.4, cm={0 ,-0.5 ,2 ,0 ,(0 cm,0 cm)}] (1.08,1) rectangle (1.09,-1);
\fill[pattern=north east lines, cm={0 ,-0.5 ,2 ,0 ,(0 cm,0 cm)}] (1.08,1) rectangle (1.09,-1);

\draw (-2,-1) rectangle (2,1);

\end{tikzpicture}}
\captionsetup{width=.95\linewidth}
\caption{\label{FigP1}Intersection between $\cA^{B}_{\ell,b}(n)$ (yellow) and $f^{-n-k_0}(\bigcup_{m\ge 2n,v\in\{t,b\}}\cA^{B}_{r,v}(m))$ (the preimages of top rectangles in blue and of bottom rectangles in green, see the right of Figure \ref{fig.interesting})}
\end{minipage}
\end{figure}

\subsection{Geometry of the set of $k$-persistent points}\label{SubsecGeom}
The goal of this subsection is to understand the geometry of the sets $P_k(\theta)$ of $k$-persistent points (defined in Definition~\ref{defpersistent}). A key idea for this is given in Lemma~\ref{l.perkmais1} which furnishes an inductive procedure: starting with the ``$0$-persistent'' points, which are just the points in some blue rectangle $\cA^B_{h,v}(n)$ one forms $P_1(\theta)$, for $\theta=(h,v,n)$, by taking the pre-images under $f_1^{n+k_0}$ of each intersection of $f_1^{n+k_0}(\cA^B_{h,v}(n))$ with some blue rectangle. We shall prove that these pre-images are precisely the connected components of $P_1(\theta)$, which will also be proven to have a special geometry (after the rescaling $L_n$ they are what we call quasi-rectangles, see Definition~\ref{DefQuasiRect}). We then show that the connected components of $2$-persistent points are the pre-images under $f_1^{n+k_0}$ of the intersections of the image $f_1^{n+k_0}(\cA^B_{h,v}(n))$ of a blue rectangle with some connected component of the set of $1$-persistent points, and the same geometric characterization (being a quasi-rectangle after rescaling) holds. We then proceed by induction. This inductive procedure is depicted in Figure~\ref{Fig.lefttop} below, where we complete our analysis by estimating the size of the connected components of $k$-persistent points     

\subsubsection{Quasi-rectangles}
Let us begin by introducing the notion of quasi-rectangle.

\begin{definition}\label{DefQuasiRect}
A \emph{quasi-rectangle} is a topological disk $R\subset\R^2$ whose boundary is made of 4 smooth curves: two that are parallel to the vertical axis, and two ``\emph{horizontal}'' that are graphs of the form $y=\gamma_j(x)$ ($x\in[x_m,x_M]$ and $j\in\{1,2\}$), with $\gamma_j : [x_m,x_M] \to \R$ a $C^1$ map and $\gamma_1(x)<\gamma_2(x)$ for all $x$. The \emph{maximal inclination} of such a quasi-rectangle is defined as 
\[\operatorname{maxincl}(R) \eqdef \max \big\{|\gamma_j'(x)|\,;\, j\in\{1,2\}, x\in [x_m,x_M]\big\}.\]
The \emph{height} of such a rectangle is the maximal vertical distance between the two pieces of the boundary defined by the $\gamma_j$:
\[\operatorname{height}(R) \eqdef \sup_{x\in [x_m,x_M]} \big|\gamma_1(x) - \gamma_2(x)\big|.\]
\end{definition}

Notice that the supremum is attained at some point $x_R\in[x_m,x_M]$. We denote $I_R=[\gamma_1(x_R),\gamma_2(x_R)]$. This choice is going to be used later in the proof of Lemma~\ref{LemEstSize}.

The result below will be used in the proof of Lemma~\ref{LemEstSize} (where we will estimate the measure of the set of persistent points) to bound the effects of the non-linearities of both the map $h_2$ and the quasi-rectangles that will appear (see Figure \ref{FigLemNonLin}).

\begin{lema}\label{LemNonLin}
	Consider $R$ a quasi-rectangle such that $\operatorname{maxincl}(R) \le 1/2$ and $\operatorname{height}(R) \le d $, whose bottom and top sides are respectively curves $\gamma_1$ and $\gamma_2$ (see Definition \ref{DefQuasiRect}). Let also be $\zeta \in C^1([x_m,x_M] , \R)$ be a map such that $|\zeta'(x)|\ge 2$ for all $x\in [x_m,x_M]$. Denote $\gr \gamma_1 \cap \gr \zeta = \{(x^1,y^1)\}$ and $\gr \gamma_2 \cap \gr \zeta = \{(x^2,y^2)\}$ (in particular, suppose that these intersections are nonempty). Then 
	\[|x^1-x^2|\le \frac{2d}{3}\quad \text{and} \quad |y^1-y^2| \le \frac{4d}{3}.\]
\end{lema}


\begin{proof}[Proof of Lemma \ref{LemNonLin}]
	
	\begin{figure}
		\begin{tikzpicture}[scale=1.6]
		\draw[dotted](0,0) rectangle (1,2);
		\fill[color=blue, opacity=.05] (0,2) -- (1,1.6) -- (1,0) -- (0,0.2);
		\draw[color=blue!60!black] (0,2) -- (1,1.6);
		\draw[color=blue!60!black] (1,0) -- (0,0.2);
		\draw[thick] (0,2) -- (1,0);
		\draw[color=red!50!black] (1,0) .. controls (0.6,1)  and (0.5,1) .. (0.2,2);
		\draw[<->] (1.1,1.6) -- (1.1,0);
		\draw (1.1,0.8) node[right]{$d$};
		\draw[<->] (0,-.1) -- (1,-.1);
		\draw (0.5,-.1) node[below]{$d'$};
		\draw[<->] (-.1,2) -- (-.1,0);
		\draw (-.1,1) node[left]{$\delta$};
		\draw[color=red!50!black, ->] (1.1,1.9)node[right]{$\gr \zeta$} to[bend right] (0.5,1.3);
		\draw[color=green!50!black] (1,0) node{$\bullet$} node[below right] {$(x^1,y^1)$};
		\draw[color=green!50!black] (0.24,1.89) node{$\bullet$} node[above right] {$(x^2,y^2)$};
		\draw (0,2) node[above left]{$(a,b)$};
		\end{tikzpicture}
		\caption{\label{FigLemNonLin}Notations of Lemma \ref{LemNonLin}.}
	\end{figure}
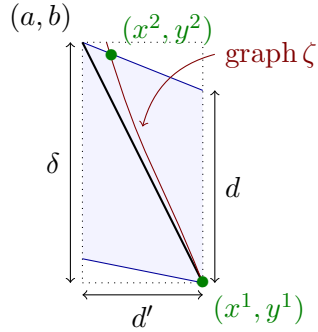
	
Consider the line passing through $(x^1,y^1)$ of slope $-2$ and the line passing through $(x^2,y^2)$ of slope $-1/2$ (see Figure~\ref{FigLemNonLin}). These two lines meet at a point $(a,b)$. Let us denote $\delta=|b-y^1|$ and $d' = |a-x^1|$. As $|\zeta'|\geq 2$ we have that $|y^1-y^2|\le \delta$ and $|x^1-x^2| \le d'$. 
Easy geometry then leads to 
	\[2d' = \delta \leq d+\frac{d'}{2},\]
	from which we obtain the bound $d'\leq d/(2-2^{-1})$. Combining the inequality $\delta\leq d+d'/2$ with the bound for $d'$ that we just obtained, one obtains 
\[\delta\leq d\left(1+\frac{1}{2^2-1}\right),\]
concluding.
\end{proof}

The usefulness of the notion of a quasi-rectangle to our arguments is two-fold. First, as we shall see in the next lemma, the connected components of $k$-persistent points are quasi-rectangles when rescaled by the maps $L_n$, and this will provide a clean strategy to estimate their Lebesgue measure: a simple Fubini argument shows that the area of quasi-rectangle $R$ is smaller than $\operatorname{height}(R)\times|x_M-x_m|$. Secondly, the inductive structure of the sets $P_k(\theta)$ will allow to estimate the height of the components of $P_{k+1}(\theta)$ from the estimation of the heights of all $P_{k}(\eta)$ with $\eta\prec\theta$.

Let us state in precise terms the fact that the connected components of $P_{k}(\theta)$ are quasi-rectangles after rescaling. Given $\theta=(h,v,n)$, with $h\in\{\ell,r\}$, $v\in\{t,b\}$ and $n\geq n_1(h)$, recall that $P_k(\theta)$ is the set of $k$-persistent points (see Definition \ref{defpersistent}). We will consider $\{\tilde{P}_k(\theta)_j\}_{j\in J_{\theta}(k)}$ the decomposition of $P_k(\theta)$ into connected components.

\begin{definition}\label{DefLP}
For any $k\in\nt$, $n\ge n_0$ and $(h,v)\in\{\ell,r\}\times \{t,b\}$, writing $\theta = (h,v,n)$, we shall denote 
\[LP_k(\theta)\eqdef L_n(P_k(\theta)).\]
\end{definition}

Notice that we can write $\{L\tilde P_k(\theta)_j\}_{j\in J_\theta(k)}$ for the decomposition of $LP_k(\theta)$ into connected components, since $L_n$ is a homeomorphism. Our goal in this subsection is to establish the following.

\begin{lema}\label{lemQuasiRec}
For any $k\in\nt$, and $\theta=(h,v,n)$, with $(h,v)\in\{\ell,r\}\times\{t,b\}$ and $n\ge n_1(h)$, the set $J_{\theta}(k)$ is finite, and the map $f_1^{n+k_0}$ induces a bijection between the collection $\{\tilde{P}_{k+1}(\theta)_j\}_{j\in J_{\theta}(k+1)}$ and the set of all intersections $\tilde{P}_k(\eta)_i\cap f_1^{n+k_0}(\cA^B_{h,v}(n))$, for $i\in J_{\eta}(k)$ and $\eta\prec\theta$. Each of these intersections $\tilde{P}_k(\eta)_i\cap f_1^{n+k_0}(\cA^B_{h,v}(n))$ is nonempty and Markovian.

Moreover, each connected component $L\tilde P_k(\theta)_j$ is a quasi-rectangle, with vertical sides that are subintervals of the vertical sides of $\cA^B_{h,v}$, and with maximal inclination smaller than $1/2$.
\end{lema}

An example of what can look like the set $L P_k(\theta)$ is depicted in Figure~\ref{FigP1}.


As mentioned at the beginning of this subsection, this lemma is a strong refinement of Lemma~\ref{l.perkmais1}: while we could only say which rectangles the trajectory of a $k$-persistent point can visit (giving ``admissible'' sequences of rectangles), Lemma \ref{lemQuasiRec} states precisely that to any of these ``admissible'' sequences of rectangles is associated a non-empty set of $k$-persistent points, and that this set is a quasi-rectangle (and in particular, is connected). The next subsection will be devoted to the estimation of the height of these quasi-rectangles.

\begin{remark} It is important for this lemma to ``rescale" the connected components of $P_k(\theta)$ by the map $L_n$. That said, nonetheless, to avoid an overload of notations in some proofs, we shall indiscriminately work with either $P_k(\theta)$ or $LP_k(\theta)$, whichever is more convenient. The important point is that, since we eventually want to estimate the \textit{ratio between the height of the sets $\tilde P_k(\theta)_j$ and the height of the rectangles $\cA^B_{h,v}(\theta)$}, our results will be invariant under the maps $L_n$.        
\end{remark}

\subsubsection{An inclination estimate} Notice that $g_0(\cA^B_{h,v}(\theta))$ is also a rectangle, since $g_0$ is an affine map (see Figure~\ref{Fig.umaiteracao}). Also, observe that $h_1|_{g_0(\cA^B_{h,v}(\theta))}$ acts as an horizontal translation (as depicted in Figure~\ref{Fig.afterh1}). The core of the proof of Lemma~\ref{lemQuasiRec} is to bound the distortion caused by the application of the perturbation map $h_2$, as illustrated in Figure~\ref{Fig.afterh2}. As a result of this we shall prove that the \emph{rescaled} first return map of $f_1^{-1}$ sends almost horizontal vectors to almost horizontal vectors.

For the statement, we recall the affine maps $L_n:\R^2\to\R^2$, defined in \eqref{e.ln}, which provide an identification between the extended stable box $\tilde{S}_n$ and a square $[\alpha,\zeta]^2$, where $\alpha<-1$ and $\zeta>1$. For every $(x,y)\in\R^2$, we have 
\[
DL_n(x,y)=\begin{bmatrix}
\frac{2}{b-a}&0\\
0&\frac{2\sigma^n}{b-a}
\end{bmatrix}.
\]
We will denote $s_v$ and $s_h$ the orthogonal symmetries with respect to respectively the $x$ (horizontal) and $y$ (vertical) axes.

\begin{lema}[Inclination]\label{LemInclination}
For every $m,n \geq n_0+1$ the following holds:
\begin{enumerate}
\item the first return map of $f_1^{-1}$ sends almost horizontal vectors to almost horizontal vectors: Let $(x,y)\in \cA^B_{\ell,b}(n) \cap f_1^{m+k_0}(\cA^B_{\ell,t}(m))$. Then, the linear map 
\[
DL_m\circ Df_1^{-m-k_0}(x,y)\circ DL^{-1}_n
\]
preserves the cone $\{(\lambda,\mu)\in\R^2;\,|\lambda|\geq 2|\mu|\}$ of vectors with inclination smaller than $1/2$.
\item the perturbation map $h_2\circ h_1$ sends horizontal vectors to almost vertical ones: Suppose that $(x,y)\in f_0\circ h_1\circ f_0^{m+k_0-1}\left(\cA^B_{\ell,t}(m)\right)$ and $(x,y^{\prime})= f_0\circ h_2\circ f_0^{-1}(x,y)\in\cA^B_{\ell,v}(n)$, for some choice of $(\ell,v,n)\prec (\ell,t,m)$. Then, the linear map
\[
DL_n\circ D(f_0\circ h_2\circ f_0^{-1})(x,y)\circ DL_{2m}^{-1}
\]
sends the vector $e_1=(1,0)$ strictly inside the cone $\{(\lambda,\mu);\,|\mu|\geq 2|\lambda|\}$ of vectors with inclination bigger than $2$.
\end{enumerate}
The same statement holds for other combinations of regions $\cA^{B}_{h,v}(n)$.

Moreover, the same statement as (1) holds for boxes of $\mathcal{L}^e$:  Let $(x,y)\in s_h\big(\cA^B_{\ell,b}(n)\big) \cap f_1^{m+k_0}\big(s_v(\cA^B_{\ell,t}(m))\big)$. Then, the linear map 
\[DL_m\circ Ds_v \circ Df_1^{-m-k_0}(x,y)\circ Ds_h \circ DL^{-1}_n =
DL_m \circ Df_1^{-m-k_0}(x,y) \circ DL^{-1}_n\]
preserves the cone $\{(\lambda,\mu)\in\R^2;\,|\lambda|\geq 2|\mu|\}$.
\end{lema}

Let us first show how to deduce Lemma \ref{lemQuasiRec} from Lemma \ref{LemInclination}.

\begin{proof}[Proof of Lemma \ref{lemQuasiRec}]
For simplicity of notation we assume $h=\ell$ and $v=t$ through this proof. Clearly the other cases are analogous. In the proof of the lemma, we will use implicitly and repeatedly the fact that the map $h_1$ is a translation in restriction to the set we are presently interested in.
We argue by induction. For $k=0$, there is nothing to say.

By the induction hypothesis, any connected component of some $LP_k(h',v',m)$ is a quasi-rectangle of inclination $\le 1/2$. Moreover, by the form of the perturbations $h_1$ and $h_2$ (see Figure \ref{fig.interesting}), the boundary of the set $f_1^{n+k_0}(\cA^B_{\ell,b}(n))$ is made of two small vertical curves, and two others which are convex graphs above the interval $[b+2\eps_1,b+2\eps_1+\delta_1]$. Also, by part 2 of Lemma~\ref{LemInclination}, after rescaling, the intersection of these graphs with each rectangle $\cA^B_{h,v}(m)$, for $\eta=(h,v,m)\prec\theta = (\ell,b,n)$, have inclination $>2$. Using the fact that two curves of respective inclinations $<1$ and $>1$ have at most one intersection, we deduce that there is exactly one connected component of
$$L_m\Big(f_1^{n+k_0}(\cA^B_{\ell,b}(n)) \cap P_k(h,v,m)\Big)$$
for each connected component of $P_k(h,v,m)$ (see Figure \ref{FigP1}). 

The boundary of each connected component of 
$$f_1^{n+k_0}(\cA^B_{\ell,b}(n)) \cap P_k(h,v,m),$$
is made of four curves, two of which are part of the respective top and bottom boundary curves of a connected component of $P_k(h,v,m)$; after rescaling these components become almost horizontal with inclination smaller than $1/2$, by induction hypothesis. By part 1 of Lemma~\ref{LemInclination}, the map $f_1^{-n-k_0}$ sends these almost horizontal curves into almost horizontal curves (after rescaling), with the same bound on the inclination. Moreover, the pre-images of the other two parts of the boundary are vertical intervals contained in the vertical boundary components of $\cA^B_{\ell,b}(n)$. This proves that the pre-image of each intersection 
$$f_1^{n+k_0}(\cA^B_{\ell,b}(n)) \cap P_k(h,v,m),$$ 
is a quasi-rectangle. This establishes the induction and finishes the proof.
\end{proof}

\begin{proof}[Proof of Lemma~\ref{LemInclination}]
We shall begin by proving the first point.
Notice that, as in the proof of Lemma~\ref{l.allerverslesbleu}, $f_1^{m+k_0}|_{\cA^B_{\ell,t}(m)}=f_0\circ h_2\circ h_1\circ f_0^{m+k_0-1}|_{\cA^B_{\ell,t}(m)}$. So, consider the point $(\bar{x},\bar{y})=f_0^{-1}(x,y)$. As $f_0^{m+k_0-1}\left(\cA^B_{\ell,t}(m)\right)=f_0^{-1}\circ g_0\left(\cA^B_{\ell,t}(m)\right)$ we can write (see Figure~\ref{Fig.umaiteracao})  	
\[
f_0^{m+k_0-1}\left(\cA^B_{\ell,t}(m)\right)=[\sigma^2(a-\eps_1-\delta_1),\sigma^2(a-\eps_1)]\times \sigma^{-2m-1}[a-3\eps_1,a-2\eps_1].
\]
Recall the bump function $\varphi_1 = \varphi_{-\sigma^{-n_0-1}a,\sigma^{-n_0-1}a}^{-\sigma^{-n_0-2}(b+3\eps_1),\sigma^{-n_0-2}(b+3\eps_1)}$ and the diffeomorphism $\xi_1:\R\to\R$ used in the definition of $h_1$. Since $m\geq n_0$, by (H3) of the definition of $\xi_1$ we have that $h_1|_{f_0^{m+k_0-1}\left(\cA^B_{\ell,t}(m)\right)}$ is a horizontal translation (by $-\sigma^2\eps_1$). This enables us to write $(\bar{x},\bar{y})=h_2(\bar{x},y^{\prime})$ and $f_1^{-1}(x,y)=(\bar{x}-\sigma^2\eps_1,y^{\prime})$. In particular, we have that
\begin{equation}
\label{e.deltaum}
|\bar{x}-\sigma^2(a-2\eps_1)|<\sigma^2\delta_1.
\end{equation}
The observation that $f_1^{m+k_0}|_{\cA^B_{\ell,t}(m)}=f_0\circ h_2\circ h_1\circ f_0^{m+k_0-1}|_{\cA^B_{\ell,t}(m)}$ also allows us to write 
\begin{eqnarray}
& DL_m\ Df_1^{-m-k_0}(x,y)\ DL^{-1}_n = & \nonumber\\
&\left(DL_n\ Df_0(\bar{x},\bar{y})\  Dh_2(\bar{x},y^{\prime})\  Dh_1(\bar{x}-\sigma^2\eps_1,y^{\prime})\  Df_0^{m+k_0-1}(f_1^{-m-k_0}(x,y))\  DL^{-1}_{m}\right)^{-1}\nonumber
\end{eqnarray}
We proceed now to compute the matrix of this linear map. Notice that $Dh_1(\bar{x}+\sigma^2\eps_1,y^{\prime})=\Id$ and consider $R_{\pi/2}:\R^2\to\R^2$ the counter-clockwise rotation of angle $\pi/2$. From Proposition~\ref{buildTowers} we can write $L_{2m}\circ f_0^{m+k_0}\circ L_m^{-1}=R_{\pi/2}$ and thus in restriction to $[\alpha,\zeta]^2$,
\[
f_0^{m+k_0-1}\circ L_m^{-1}=f_0^{-1}\circ L_{2m}\circ R_{\pi/2}.
\]
Observe that, in this equality, $f_0^{-1}$ is the linear map $(x,y)\mapsto\operatorname{Diag}(\sigma^{2}x,\sigma^{-1}y)$. Let us compute $Dh_2(\bar{x},y^{\prime})$.  
 
With this purpose denote along this proof
	\begin{equation}\label{EqDefalphaBeta}
	\gamma = \varphi_2'(\bar x) \big(\xi_2(y^{\prime})- y^{\prime}\big) \qquad \text{and}\qquad \beta = \varphi_2(\bar x) \xi_2'(y^{\prime}) + 1 - \varphi_2(\bar x).
	\end{equation} 
	Recall from the definition of $h_2$ (\eqref{EqDefH2} page \pageref{EqDefH2}) that $$h_2(\bar x,y^{\prime})=\left(\bar x,\varphi_2(\bar x) \xi_2(y^{\prime}) + (1 - \varphi_2(\bar x))y^{\prime}\right)$$and\[Dh_2(\bar x,y^{\prime}) = \begin{pmatrix} 
	1 & 0\\
	\gamma & \beta
	\end{pmatrix}
	\qquad \text{and} \qquad
	Dh_2^{-1}(\bar x,\bar y) = \begin{pmatrix} 
	1 & 0\\
	-\frac{\gamma}{\beta} & \frac{1}{\beta}
	\end{pmatrix}.\]Therefore, the matrix we are seeking to compute is (we omit the base points as they are now clear):
	\begin{align*}
	R_{-\frac{\pi}{2}}\,DL_{2m}\,f_0\,Dh_2^{-1}\,f_0^{-1}\,DL_n^{-1} 
	& =
	R_{-\frac{\pi}{2}}\begin{pmatrix} 
	\frac{2}{b-a} & 0\\
	0 & \frac{2\sigma^{2m}}{b-a}\\
	\end{pmatrix}
	f_0
	\begin{pmatrix} 
	1 & 0\\
	-\frac{\gamma}{\beta} & \frac{1}{\beta}
	\end{pmatrix}
	\begin{pmatrix} 
	\frac{\sigma^2(b-a)}{2} & 0\\
	0 & \frac{b-a}{2\sigma^{n+1}}\\
	\end{pmatrix}
	\\
	& =
	R_{-\frac{\pi}{2}}\begin{pmatrix} 
	\frac{2}{b-a} & 0\\
	0 & \frac{2\sigma^{2m}}{b-a}\\
	\end{pmatrix}
	\begin{pmatrix}
	\sigma^{-2}&0\\
	0&\sigma
	\end{pmatrix}
	\begin{pmatrix} 
	\frac{\sigma^2(b-a)}{2} & 0\\
	-\frac{\sigma^2(b-a)}{2}\frac{\gamma}{\beta} & \frac{b-a}{2\sigma^{n+1}}\frac{1}{\beta}
	\end{pmatrix}\\
	\\
	& = 
	R_{-\frac{\pi}{2}}\begin{pmatrix} 
	1 & 0\\
	-\sigma^{2m+3} \frac{\gamma}{\beta} & \frac{\sigma^{2m-n}}{\beta}
	\end{pmatrix}
	\\
	& =
	\begin{pmatrix} 
	-\sigma^{2m+3} \frac{\gamma}{\beta} & \frac{\sigma^{2m-n}}{\beta}\\
	-1 & 0
	\end{pmatrix}.
	\end{align*}
	It remains to prove that this matrix preserves the cone of vectors with inclination smaller than $1/2$. Thus take $(\lambda,\mu)\in\R^2$ such that $|\lambda|\geq 2|\mu|$. 
The ratio between components of the vector $R_{-\frac{\pi}{2}}\,DL_{2m}\,f_0\,Dh_2^{-1}\,f_0^{-1}\,DL_n^{-1} (\lambda,\mu)$ (\emph{i.e.}, its inclination) is given by
	\[
	\eta\eqdef\left|\frac{-\lambda\sigma^{2m+3} \frac{\gamma}{\beta} + \sigma^{2m-n}\frac{\mu}{\beta}}{-\lambda}\right|,
	\]
and thus we have to control the terms $\gamma$ and $\beta$. This estimation is based on the localization of the points $\bar x,\bar y$ and $y^{\prime}$, and the relation $h_2(\bar x,y^{\prime})=(\bar x,\bar y)$. Indeed, by the definition of $h_2$ this equality implies that 
\begin{equation}
\label{e.cefet}
\bar y- y^{\prime}=\varphi_2(\bar x)(\xi_2(y^{\prime})-y^{\prime}) 
\end{equation}

As $f^{-m-k_0}(x,y)\in\cA_{\ell,t}(m)$, we have that $y^{\prime}\in\sigma^{-2m-1}[a-3\eps_1,a-2\eps_1]$, which implies $0<\xi_2^{\prime}(y^{\prime})<1$ (recall the graph of $\xi_2$ in  Figure~\ref{f.grafxi2}). As $\varphi_2(\bar x)\in [0,1]$ (see Figure~\ref{FigPhi2}) we conclude that  $0<\beta\leq 1$.

Also, Lemma~\ref{LemIntersecBlue} gives $(\ell,b,n)\prec(\ell,t,m)$ and so we must have $2m\geq n$. 
These considerations are enough to give us the bound
\begin{equation}\label{EqMinEta}
\eta\geq\frac{1}{\beta}\left|\sigma^{2m+3} \gamma - \sigma^{2m-n}\frac{\mu}{\lambda}\right|
\ge \frac{\sigma^{2m-n}}{\beta}\left(\sigma^{n+3} |\gamma|- \left|\frac{\mu}{\lambda}\right|\right).
\end{equation}

We proceed now to estimate the term $\sigma^{n+3}\gamma$. Since $(\bar x,\bar y)=f_0^{-1}(x,y)$, and since 
\[
f_0^{-1}(\cA^B_{\ell,b}(n))=[\sigma^2(a-3\eps_1),\sigma^2(a-2\eps_1)]\times \sigma^{-n-1}[a-\eps_1-\delta_1,a-\eps_1],
\]
we can estimate that (recall that $2m\geq n$)
\[
\bar y - y^{\prime}\geq \sigma^{-n-1}(a-\eps_1-\delta_1)-\sigma^{-2m-1}(a-2\eps_1)\geq \sigma^{-n-1}(\eps_1-\delta_1).
\]
Since $\delta_1=\frac{\eps_1}{10}$, we conclude that
\begin{equation}\label{Eqyybar}
y-\bar{y}\geq 9\delta_1\sigma^{-n-1}.
\end{equation}
Notice now that $\xi_2(y^{\prime})-y^{\prime}>0$, because $y^{\prime}>0$ (see Figure~\ref{f.grafxi2}). Then, \eqref{e.cefet} with the above inequality yield
	\[\varphi_2(x)\ge \frac{9\delta_1\sigma^{-n-1}}{\xi_2(y^{\prime})-y^{\prime}}.\]
	
Now, we apply the convexity of $\varphi_2$. Observe that $\bar{y}\leq\sigma^{-n-1}(a-\eps_1)$, and since $n\geq n_0+1$, we have $\bar{y}<\sigma^{-n_0-2}(b+3\eps_1)$. By Remark~\ref{rem.convex}, the restriction $\varphi_2|_{[\bar x,\sigma^2(a-2\eps_1)}$ is convex. With the previous bound on $\varphi_2(\bar x)$ this allows to say that	
\begin{equation}\label{EqVarPhi2Prim}
|\varphi_2'(\bar x)| \ge \frac{9\delta_1\sigma^{-n-1}}{\left(\xi_2(y^{\prime})-y^{\prime}\right)|\bar x-\sigma^2(a-2\eps_1)|}.
\end{equation}
By the definition \eqref{EqDefalphaBeta} of $\gamma$ and by \eqref{e.deltaum} we then have
\begin{equation}
\label{e.negodrama}
|\sigma^{n+3}\gamma| \geq \sigma^{n+3}\frac{9\sigma^{-n-1}\delta_1}{|\bar x-\sigma^2(a-2\eps_1)|}\geq \sigma^{n+3}\frac{9\sigma^{-n-1}\delta_1}{\sigma^2\delta_1} \ge 9.
\end{equation}
By assumption, $|\frac{\mu}{\lambda}|\leq\frac{1}{2}$, and therefore estimations \eqref{e.negodrama} and \eqref{EqMinEta} combined give $\eta> 2$, as desired.
\bigskip

Let us now explain hot to get the result about points of $\mathcal{L}^e$. By composing left and right by the symmetries $s_h$ and $s_v$, one can see that it amounts to prove the case \textit{(1)}, where the map $\xi_2$ is replaced by the map $\xi_2^e : y\mapsto -\xi_2(-y)$ (the map whose graph is obtained from the graph of $\xi_2$ by a symmetry around the origin). We now explain how this change affects the proof.

As in our case the order $\prec$ is replaced by $\preceq$, the condition $(\ell,b,n)\preceq(\ell,t,m)$ implies $2m\leq n$, and hence the bound \eqref{EqMinEta} becomes
\[\eta\geq\left|\frac{1}{\beta}\left(\sigma^{2m+3} \gamma - \sigma^{2m-n}\frac{\mu}{\lambda}\right)\right|
\ge\frac{1}{|\beta|}\left( |\sigma^{2m+3}\gamma| -|\frac{\mu}{\lambda}|\right).\]

The same computation as before shows that $-1\le\beta \le 1$, hence 
\[\eta\geq |\sigma^{2m+3}\gamma| -|\frac{\mu}{\lambda}|.\]
Because now $2m \le n$, the inequality \eqref{Eqyybar} becomes
\begin{equation*}
|y-\bar{y}|\geq 9\delta_1\sigma^{-2m-1},
\end{equation*}
and the convexity estimation holds as before: \eqref{EqVarPhi2Prim} becomes
\begin{equation*}
|\varphi_2'(\bar x)| \ge \frac{9\delta_1\sigma^{-2m-1}}{\left|\xi_2(y^{\prime})-y^{\prime}\right||\bar x-\sigma^2(a-2\eps_1)|},
\end{equation*}
and implies the counterpart of \eqref{e.negodrama}:
\begin{equation*}
|\sigma^{2m+3}\gamma| \geq \sigma^{2m+3}\frac{9\sigma^{-2m-1}\delta_1}{|\bar x-\sigma^2(a-2\eps_1)|}\geq \sigma^{2m+3}\frac{9\sigma^{-2m-1}\delta_1}{\sigma^2\delta_1} \ge 9,
\end{equation*}
As $|\frac{\mu}{\lambda}|\leq\frac{1}{2}$ this proves as before that $\eta>2$.
\bigskip

For the second point of the lemma, one computes, similarly to what has been done for the first point,
	\[DL_{n}Df_0Dh_2Df_0^{-1}DL_{2m}^{-1} 
	= \begin{pmatrix} 
	1 & 0\\
	\sigma^{n+3}\gamma & \beta \sigma^{n-2m}
	\end{pmatrix},\]
	thus using \eqref{e.negodrama} we deduce that the image of $e_1$ under this matrix is a vector whose inclination is 
\[
\gamma\sigma^{n+3}\geq 9,
\]
concluding the proof.
\end{proof}

\begin{remark}
The proof of the inclination lemma in the case $\cA^B_{r,b}(m)$ is identical to the above. The cases $\cA^B_{\ell,b}(m)$ and $\cA^B_{r,t}(m)$ are analogous but even simpler: one can compute directly with the points $(x,y)$, and there is no need to consider $(\bar x,\bar y)$. The reason is that in these cases the perturbations $h_1,h_2$ act only \emph{after} the first return $g_0$. In the case we explained above, they act one iteration before the first return, and this introduces some additional terms in the matrix computation which is necessary for the proof.   
\end{remark}

\subsection{Size estimation of the set of persistent points}\label{SubsecSize}

We shall now conclude the main step of our analysis of the set of persistent points. We have seen that the connected components $\{P_{k}(\theta)_j\}_{i\in J_{\theta}(k)}$ of $k$-persistent points, for $\theta=(h,v,m)$, after been rescaled by $L_m$, become the quasi-rectangles $\{L\tilde{P}_k(\theta)_j\}_{j\in J_{\theta}(k)}$. Fix $j\in J_{\theta}(k)$ and denote $\tilde{d}_{\theta}(k)_j\eqdef\operatorname{height}(L\tilde{P}_k(\theta)_j)$. Fix also a vertical interval $I_{L\tilde{P}_k(\theta)_j}\subset L\tilde{P}_k(\theta)_j$ so that $\tilde{d}_{\theta}(k)_j=\operatorname{lenght}(I_{L\tilde{P}_k(\theta)_j})$, as we explained right after Definition~\ref{DefQuasiRect}. We denote by $d_{\theta}(k)_j$ the length of $L_m^{-1}({I_{L\tilde{P}_k(\theta)_j}})$. Notice that, being the pre-image of a quasi-rectangle by $L_m$, the set $P_{k}(\theta)_j$ can be described as a union of vertical segments and the number $d_{\theta}(k)_j$ is the maximal length of such intervals.  
To prove that $\Gamma=\bigcup_\theta\bigcap_{k}P_k(\theta)$ has zero Lebesgue measure, the main step will be to establish that the proportion of the \emph{height of the blue rectangle $\cA^B_{h,v}(m)$} occupied by the sum of the numbers $d_{\theta}(k)_j$ decreases exponentially to zero as $k\to\infty$, with rate 2. Notice that $\operatorname{height}(\cA^B_{h,v}(m))=\sigma^{-m}\delta_1$. For simplicity we denote this numbers by $h_{\theta}$. 

\begin{lema}\label{LemEstSize}
For every $k\in\nt$, $m\ge n_0$ and $(h,v)\in\{\ell,r\}\times \{t,b\}$, denoting $\theta = (h,v,m)$, 
we have
\begin{equation}\label{EqLemEstSize}
\sum_{j\in J_\theta(k)} \frac{d_\theta(k)_j}{h_{\theta}}\leq 2^{-k}.
\end{equation}
\end{lema}

Notice that this lemma implies that $\bigcap_{k=1}^{\infty}P_k(\theta)$ has zero Lebesgue measure: it suffices to combine the estimation \eqref{EqLemEstSize}, the fact that each $P_k(\theta)$ is a quasi-rectangle and a Fubini argument. A statement analogous to Lemma~\ref{LemEstSize} (with the same proof) holds for the \emph{exterior} persistent points and thus $m(\Gamma^e)$ is also zero, using the part of Lemma~\ref{LemInclination} relative to exterior boxes.

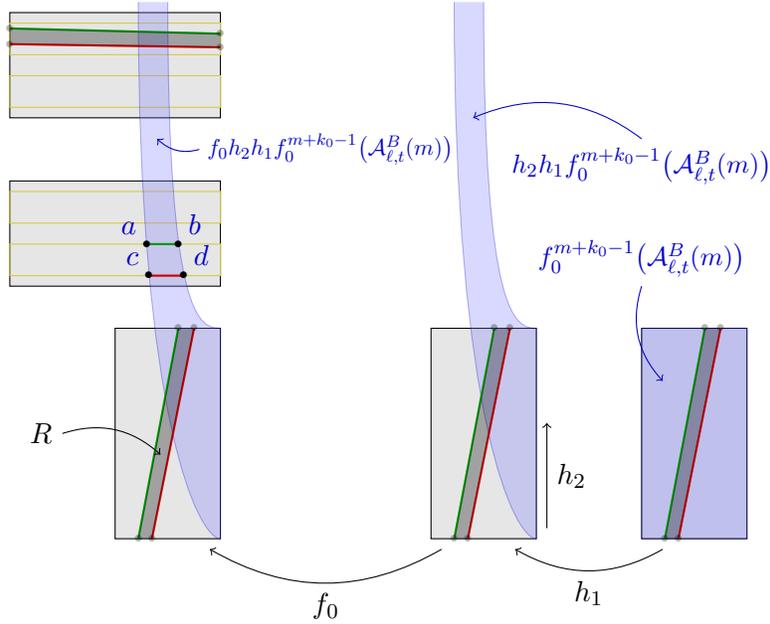
\begin{figure}
\begin{center}
\begin{tikzpicture}[scale=1.4]
\filldraw[fill=black!10!white,draw=black] (1,0) rectangle (2,2);
\filldraw[fill=black!10!white,draw=black] (4,0) rectangle (5,2);
\filldraw[fill=black!10!white,draw=black] (6,0) rectangle (7,2);
\filldraw[color=blue!80!black, fill=blue!70!white, opacity=.25](6,0) rectangle (7,2);
\filldraw[fill=black!10!white,draw=black] (0,2.4) rectangle (2,3.4);
\filldraw[fill=black!10!white,draw=black] (0,4) rectangle (2,5);
\draw[draw=yellow!80!black] (0,2.5) rectangle (2,2.8);
\draw[draw=yellow!80!black] (0,3) rectangle (2,3.3);
\draw[draw=yellow!80!black] (0,4.1) rectangle (2,4.4);
\draw[draw=yellow!80!black] (0,4.6) rectangle (2,4.9);

\filldraw[color=blue!80!black, fill=blue!50!white, opacity=.3] (1.22,5.1) .. controls +(0,-4) and +(-.2,0) .. (2,0) -- (2,2).. controls +(-.2,0) and +(0,-3).. (1.5,5.1);

\filldraw[color=blue!80!black, fill=blue!50!white, opacity=.3](4.22,5.1) .. controls +(0,-4) and +(-.2,0) ..  (5,0)-- (5,2).. controls +(-.2,0) and +(0,-3).. (4.5,5.1);

\draw[color=green!60!black, thick] (1.3,2.8) -- (1.6,2.8);
\draw[color=red!80!black, thick] (1.3,2.5) -- (1.65,2.5);

\draw (1.3,2.8) node[scale=.6]{$\bullet$} node[above left, color=blue!80!black]{$a$};
\draw (1.6,2.8) node[scale=.6]{$\bullet$} node[above right, color=blue!80!black]{$b$};
\draw (1.32,2.5) node[scale=.6]{$\bullet$} node[above left, color=blue!80!black]{$c$};
\draw (1.65,2.5) node[scale=.6]{$\bullet$} node[above right, color=blue!80!black]{$d$};

\draw[color=green!60!black, thick] (1.22,0) -- (1.6,2);
\draw[color=red!80!black, thick] (1.35,0) -- (1.75,2);
\fill[black, opacity=.3] (1.22,0)node[scale=.6]{$\bullet$} -- (1.6,2)node[scale=.6]{$\bullet$} -- (1.75,2)node[scale=.6]{$\bullet$} --(1.35,0)node[scale=.6]{$\bullet$};

\draw[<-,color=black!80!black] (1.43,.8) to[bend right] (0.5,1) node[left]{$R$};

\draw[color=green!60!black, thick] (4.22,0) -- (4.6,2);
\draw[color=red!80!black, thick] (4.35,0) -- (4.75,2);
\fill[black, opacity=.3] (4.22,0)node[scale=.6]{$\bullet$} -- (4.6,2)node[scale=.6]{$\bullet$} -- (4.75,2)node[scale=.6]{$\bullet$} --(4.35,0)node[scale=.6]{$\bullet$};

\draw[color=green!60!black, thick] (6.22,0) -- (6.6,2);
\draw[color=red!80!black, thick] (6.35,0) -- (6.75,2);
\fill[black, opacity=.3] (6.22,0)node[scale=.6]{$\bullet$} -- (6.6,2)node[scale=.6]{$\bullet$} -- (6.75,2)node[scale=.6]{$\bullet$} --(6.35,0)node[scale=.6]{$\bullet$};

\draw[color=red!80!black, thick] (2,4.67) -- (0,4.7);
\draw[color=green!60!black, thick] (2,4.8) -- (0,4.85);
\fill[black, opacity=.3] (2,4.67)node[scale=.6]{$\bullet$} -- (0,4.7)node[scale=.6]{$\bullet$} -- (0,4.85)node[scale=.6]{$\bullet$} -- (2,4.8)node[scale=.6]{$\bullet$};

\draw[->] (6.2,-.1) to[bend left] node[midway, below]{$h_1$} (4.8,-.1);
\draw[->] (5.1,.1) to node[midway, right]{$h_2$} (5.1,1.1);
\draw[->]  (4.1,-.1) to[bend left] node[midway, below]{$f_0$} (1.9,-.1);
\draw[<-,color=blue!80!black] (6.2,1.5) to[bend left] (6,2.4)node[above,scale=.9]{$f_0^{m+k_0-1}\big(\cA^B_{\ell,t}(m)\big)$};
\draw[<-,color=blue!80!black] (4.4,4) to[bend left] (6,3.8)node[below,scale=.9]{$h_2 h_1 f_0^{m+k_0-1}\big(\cA^B_{\ell,t}(m)\big)$};
\draw[<-,color=blue!80!black] (1.4,3.8) to[bend right] (1.8,3.7)node[right,scale=.78]{$f_0h_2 h_1 f_0^{m+k_0-1}\big(\cA^B_{\ell,t}(m)\big)$};
\end{tikzpicture}
\end{center}
\caption{\label{Fig.lefttop}Schematic representation of the construction of $P_2(\theta)$ from the $1$-persistent points: the yellow rectangles are the connected components of $1$-persistent points. The image of the blue rectangle $\cA_{\ell,t}^B(m)$ crosses all the components $P_1(\eta)_j$, with $\eta\prec\theta$. One of these crosses is represented with two boundary components in green and red. The inverse image under $f_1^{m+k_0}$ of this intersection gives the components of $P_2(\theta)$ represented in dark grey.}  
\end{figure}

Let us give a schematic idea of the proof. The global strategy is to argue by induction: we use the convexity property of the perturbation $h_2$ (see Remark~\ref{rem.convex}) and take advantage from the space below each blue rectangle.

We shall make reference to Figure~\ref{Fig.lefttop}. Using arguments similar to that of Lemma~\ref{LemIntersecBlue}, one can see that the set $f_1^{m+k_0}(\cA^B_{\ell,t}(m))$ (the blue set in the left-hand side of Figure~\ref{Fig.lefttop}) intersects all components $P_k(\eta)_j$, for $\eta\prec\theta$ (represented as yellow rectangles). One of these intersections is the set bounded by the cycle $abcd$, marked in the figure. By Lemma~\ref{lemQuasiRec}, the pre-image of these intersections under $f_1^{m+k_0}$ gives all the connected components of $P_{k+1}(\theta)$. Thus, the goal of the proof is to analyse the size of the pre-image under $f_1^{m+k_0}$ of this intersection. 

For this, one initially takes the pre-image by $f_0\circ h_2$, which yields the dark grey ``almost rectangle''(with almost vertical sides in green and in red in Figure~\ref{Fig.lefttop}) in the bottom-center of the figure. Using the fact that the map $(h_1)^{-1}$ is a similitude in restriction to this set, we obtain the set in the right-hand side of the figure, with almost vertical sides in green and red. Notice that all three grey rectangles with green and red boundaries in the bottom of Figure~\ref{Fig.lefttop} are similarly equivalent. Taking the pre-image under $f_0^{m+k_0-1}$ doesn't change the similarity class, because $f_0^{-m-k_0}$ is a rotation in restriction to the set we are interested in and $f_0$ is a diagonal matrix. One thus obtains the component of $P_{k+1}(\theta)$ represented in dark grey at the left top of the figure. By our analysis, this component is similar to the set $R$.  

Therefore, we need to estimate the \emph{shape} of the set $R$. In fact, it is its \emph{width} that we estimate, because the component of $k+1$-persistent is obtained from $R$ essentially rotating it by $-\pi/2$. 
 
As we shall see in the proof, the blue set on the left-hand side is the image of the grey rectangle in the left corner of Figure~\ref{Fig.lefttop} by the map $f_0\circ h_2\circ f_0^{-1}$.  Using our convexity assumption (see Remark~\ref{rem.convex}), we shall prove that the image under this map of each horizontal segment contained in $R$ is a convex curve joining the bottom and the top of a yellow rectangle (a component of $k$-persistent points), like the segment $ca$ which corresponds to the image of the bottom side of $R$.  




So, let us give an idea of how we use convexity to estimate the width of the set $R$. This idea is depicted in Figure~\ref{FigZoom}. 

To simplify the explanation, suppose that the yellow quasi-rectangles $LP_k(\theta)_j$ are actual rectangles (as in the figure), and rescale everything by $L_n$. 
As we explained above, the height of the rectangles at step $k+1$ are the horizontal size of the rectangle $R$ of Figure~\ref{Fig.lefttop}, which are the numbers $d_1$ and $d_2$ of Figure~\ref{FigZoom}. Our goal is to bound their sum from above, using the sum of the heights $z_1+z_2$ at step $k$. First, by convexity of the green curve, one has $d_1+d_2\le d'$. Again, by convexity, one has the inequality on the inclination:
\[\frac{z_1+z_2}{d'} \ge \frac{9\delta_1}{d''}.\]
Recall that after the rescaling the size of a blue rectangle is $\delta_1=\eps_1/10$ and the space between two consecutive blue rectangles is at least $\eps_1-\delta_1=9\delta_1$ (this was used already in the Inclination Lemma~\ref{LemInclination}). This implies
\[\frac{d'}{d'+d''} \le \frac{1}{9} \frac{z_1+z_2}{\delta_1}\ d''\frac{1}{d'+d''} \le \frac{1}{9} \frac{z_1+z_2}{\delta_1}.\]
The last term $\frac{z_1+z_2}{\delta_1}$ represents the left term of \eqref{EqLemEstSize}; it is bounded by the induction hypothesis. The factor $1/9$ corresponds to the ratio between the height of a blue rectangle and the space below it; it will give a rate of exponential decreasing. A rigorous version of the \emph{shape invariance} analysis, described above, shows that this bound on $d'/(d'+d'')$ (the width of $R$) implies a bound on the left term of \eqref{EqLemEstSize} at step $k+1$.

In the actual proof we will have to deal with the fact that the sets $P_k(\theta)_j$ are quasi-rectangles and not actual rectangles. To do this, we will use the inclination estimate given by Lemma~\ref{LemInclination} (which gives information after rescaling by $L_n$) together with a simple geometry lemma (Lemma~\ref{LemNonLin}).

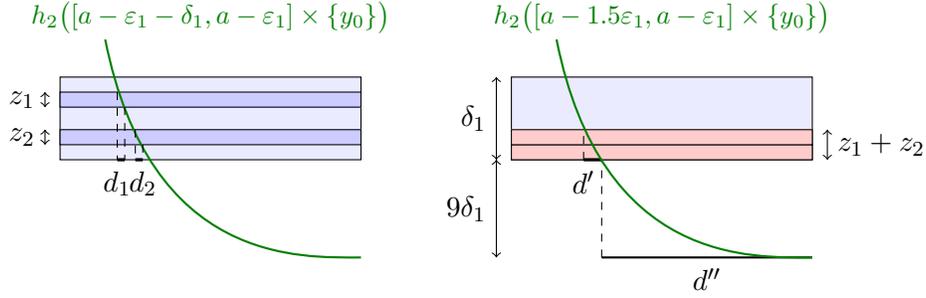
\begin{figure}
\begin{tikzpicture}[scale=2]

\filldraw[fill=blue!8!white] (-2,2.45) rectangle (0,3);
\filldraw[fill=blue!20!white] (-2,2.55) rectangle (0,2.65);
\filldraw[fill=blue!20!white] (-2,2.8) rectangle (0,2.9);
\draw[color=green!50!black, thick] (0,1.8) .. controls +(-.4,0) and +(.3,-1.5) .. (-1.7,3.25);
\draw[color=green!50!black] (-1,3.4) node{\small $h_2\big([a-\eps_1-\delta_1,a-\eps_1]\times\{y_0\}\big)$};
\draw[<->] (-2.1,2.55) -- (-2.1,2.65)node[midway, left] {$z_2$};
\draw[<->] (-2.1,2.9) -- (-2.1,2.8)node[midway, left] {$z_1$};
\draw[dashed] (-1.57,2.8) -- (-1.57,2.45);
\draw[dashed] (-1.62,2.9) -- (-1.62,2.45);
\draw[very thick] (-1.57,2.45) -- (-1.62,2.45) node[below]{$d_1$};
\draw[dashed] (-1.5,2.65) -- (-1.5,2.45);
\draw[dashed] (-1.45,2.55) -- (-1.45,2.45);
\draw[very thick] (-1.5,2.45) -- (-1.45,2.45) node[below]{$d_2$};

\filldraw[fill=blue!8!white] (1,2.45) rectangle (3,3);
\filldraw[fill=red!20!white] (1,2.45) rectangle (3,2.55);
\filldraw[fill=red!20!white] (1,2.55) rectangle (3,2.65);

\draw[<->] (3.1,2.45) -- (3.1,2.65)node[midway, right] {$z_1+z_2$};
\draw[dashed] (1.48,2.45) -- (1.48,2.65);
\draw[very thick] (1.48,2.45) node[below]{$d'$} -- (1.6,2.45) ;
\draw[<->] (.9,2.45) -- (.9,3)node[midway, left] {$\delta_1$};
\draw[<->] (.9,2.45) -- (.9,1.8)node[midway, left] {$9\delta_1$};
\draw[thick] (1.6,1.8) -- (3,1.8) node[midway,below]{$d''$};
\draw[dashed] (1.6,1.8) -- (1.6,2.45);

\draw[color=green!50!black, thick] (3,1.8) .. controls +(-.4,0) and +(.3,-1.5) .. (1.3,3.25);
\draw[color=green!50!black] (2,3.4) node{\small $h_2\big([a-1.5\eps_1,a-\eps_1]\times\{y_0\}\big)$};

\end{tikzpicture}
\caption{\label{FigZoom}Zoom on Figure \ref{Fig.lefttop}, and idea of the proof of Lemma \ref{LemEstSize}. The darker blue rectangles are some $P_k(\theta)_j$. The argument uses the convexity of the map $\varphi_2$ to say that $d'\ge d_1+d_2$. The rest of the proof consists in an estimation of nonlinearities, which follows from Lemma~\ref{LemNonLin}. In the detail, $z_i = |y_{\eta}^2(k)_i - y_{\eta}^1(k)_i|$ and $d_i = |x_{\eta}^2(k)_i - x_{\eta}^1(k)_i|$ (defined in Figure \ref{FigEstSize}).}
\end{figure}

\begin{proof}[Proof of Lemma~\ref{LemEstSize}]
Let us begin assuming that $\theta=(\ell,t,m)$. Denote $\lambda_{\theta}(k)_j=\frac{d_\theta(k)_j}{h_{\theta}}$. We shall prove the announced inequality, for every $m\geq n_1(h)$, by induction on $k$. The case $k=0$ is clear, so assume this has been proved until step $k$. Following the idea we described, we shall see the intersection of $f_1^{m+k_0}(\cA^B_{\ell,t}(m))\cap P_k(\eta)_j$, for $\eta\prec\theta$ as a set bounded by a cycle (like the cycle $abcd$ in Figure~\ref{Fig.lefttop}) whose vertices are intersections of convex a graph (the map $\vartheta$ that we shall define below) with the boundary of $P_k(\eta)_j$. So, our first step is to make precise this assertion. After that, we shall use $\vartheta$ to estimate the numbers $d_{\theta}(k+1)_j$.
Notice that 
\[
f_0\circ h_1\circ f_0^{m+k_0-1}\left(\cA^B_{\ell,t}(m)\right)=[a-2\eps_1-\delta_1,a-2\eps_1]\times\sigma^{-2m}[a-2\eps_1,a-3\eps_1].
\]
For each $(x,y)\in f_0\circ h_1\circ f_0^{m+k_0-1}\big(\cA^B_{\ell,t}(m)\big)$ we denote $f_0\circ h_2\circ f_0^{-1}(x,y) = \big(x,\, \vartheta(x)\big)$ (corresponding to the green curve in Figure~\ref{FigEstSize}), so that
\[
\vartheta(x) = y+\sigma\varphi_2(\sigma^2x)(\xi_2(\sigma^{-1}y)-\sigma^{-1}y).
\] 
For simplicity of notation, we suppress the dependence of $\vartheta(x)$ on $y$.  Moreover, for every $y$, along the intersections $\gr(\vartheta)\cap\cA^B_{\ell,w}(n)$, with $\eta=(\ell,w,n)\prec\theta$, the map $x\mapsto\vartheta(x)$ is convex (Remark~\ref{rem.convex}).
By point (2) of the Inclination Lemma~\ref{LemInclination}, for every $\eta=(\ell,w,n)\prec\theta$ the curve $L_n(\operatorname{graph}(\vartheta)\cap\cA^B_{\ell,w}(n))$ is a curve of inclination $>2$. As each set $\{L\tilde{P}_k(\eta)_j\}_{j\in J_{\eta}(k)}$ is a quasi-rectangle with inclination $<1/2$, we see that $L_n(\operatorname{graph}(\vartheta)\cap\cA^B_{\ell,w}(n))$ crosses the boundary of $L\tilde{P}_k(\eta)_j$ in exactly two points, which we denote by $(\tilde{x}^1_{\eta}(k)_j,\tilde{y}_{\eta}(k)_j)$, for the bottom intersection and $(\tilde{x}^2_{\eta}(k)_j,\tilde{y}^2_{\eta}(k)_j)$ for the top intersection (see Figure~\ref{FigEstSize}). We let $(x^1_{\eta}(k)_j,y^2_{\eta}(k)_j)$ and $(x^1_{\eta}(k)_j,y^2_{\eta}(k)_j)$ denote the respective images of these points under $L_n^{-1}$. These are precisely the intersections of $\operatorname{graph}(\vartheta)$ with the boundary of the set $P_{k}(\eta)_j$. Notice moreover that these intersection points do depend on $y$, though we have made the choice of suppress this dependence (in the notation).

\begin{figure}
	\begin{tikzpicture}[scale=1.6]
	\filldraw[fill=blue!8!white] (-2,2) rectangle (0,3);
	\draw[color=blue!50!black] (-.6,3.2) node{$P_{\eta_1}(0)$};
	\filldraw[fill=blue!8!white] (-2,4.5) rectangle (0,5.5);
	\draw[color=blue!50!black] (-.6,5.7) node{$P_{\eta_2}(0)$};
	\filldraw[fill=blue!20!white] (-2,2.2) rectangle (0,2.35);
	\filldraw[fill=blue!20!white] (-2,2.6) rectangle (0,2.8);
	\filldraw[fill=blue!20!white] (-2,4.7) rectangle (0,4.85);
	\filldraw[fill=blue!20!white] (-2,5.1) rectangle (0,5.3);
	\draw[<-, color=blue!40!black] (-1.9,2.7) to[bend right] (-2.2,1.8) node[below]{$P_{\eta_1}(k)_1$};
	\draw[<-, color=blue!40!black] (-1.8,2.3) to[bend left] (-1.8,1.4) node[below]{$P_{\eta_1}(k)_2$};
	
	\draw[<-, color=blue!40!black] (-1.9,5.2) to[bend right] (-2.2,4.3) node[below]{$P_{\eta_2}(k)_1$};
	\draw[<-, color=blue!40!black] (-1.8,4.75) to[bend left] (-1.9,3.8) node[below]{$P_{\eta_2}(k)_2$};
	
	\draw[<->, color=blue!40!black] (-2.4,2) -- (-2.4,3) node[midway, left]{$h_{\eta_1}$};
	\draw[<->, color=blue!40!black] (-2.4,4.5) -- (-2.4,5.5) node[midway, left]{$h_{\eta_2}$};
	\draw[<->, color=brown!60!black] (1,-1.1) -- (2,-1.1) node[midway, below]{$\delta_1$};
	
	\filldraw[fill=yellow!30!white] (1,-1) rectangle (2,1);
	\draw[color=yellow!10!black] (1,-.8) node[left] {$f_0^{n+k_0}\left(\cA^B_{\ell,t}(m)\right)$};
	\draw[color=red!70!black, thick] (1,.4) -- (2,.4);
	\draw[dashed] (-2,1) -- (2.3,1) node[right]{$Y_{\eta_1-1}^2$};
	\draw[<-, color=red!70!black] (1.5,.35) to[bend left] (.4,-.3) node[left]{$[a-1.5\eps_1,a-\eps_1]\times\{y\}$};
	\draw[color=green!50!black, thick] (0,.4) .. controls +(-1,0) and +(.3,-2.5) .. (-1.7,5.8);
	\draw[<-, color=green!50!black] (-.68,.9) to[bend right] (-1.3,.4);
	\draw[color=green!50!black](0,.4) node[below left]{$\vartheta\big([a-1.5\eps_1,a-\eps_1]\times\{y\}\big)$};
	\draw[dotted, color=gray, thick] (0,.4)--(1,.4);
	
	\draw (-1.07,2) node{$\times$};
	\draw (-1.12,2.2) node{$\times$};
	\draw (-1.15,2.35) node{$\times$};5
	\draw (-1.2,2.6) node{$\times$};
	\draw (-1.25,2.8) node{$\times$};
	\draw (-1.3,3) node{$\times$};
	\draw (-1.53,4.5) node{$\times$};
	\draw (-1.56,4.7) node{$\times$};
	\draw (-1.58,4.85) node{$\times$};
	\draw (-1.61,5.1) node{$\times$};
	\draw (-1.64,5.3) node{$\times$};
	\draw (-1.66,5.5) node{$\times$};
	
	\draw[<-] (.04,2) to (.4,1.5) node[right]{$Y_{\eta_1}^1$};
	\draw[<-] (.04,2.2) to (.8,1.9) node[right]{$y_{\eta_1}^1(k)_2$};
	\draw[<-] (.04,2.35) to (1.3,2.3) node[right]{$ y_{\eta_1}^2(k)_2$};
	\draw[<-] (.04,2.6) to (1.3,2.7) node[right]{$ y_{\eta_1}^1(k)_1$};
	\draw[<-] (.04,2.8) to (.8,3.1) node[right]{$y_{\eta_1}^2(k)_1$};
	\draw[<-] (.04,3) to (.4,3.5) node[right]{$Y_{\eta_1}^2$};
	
	\draw[<-] (.04,4.5) to (.4,4) node[right]{$Y_{\eta_2}^1$};
	\draw[<-] (.04,4.7) to (.8,4.4) node[right]{$y_{\eta_2}^1(k)_2$};
	\draw[<-] (.04,4.85) to (1.3,4.8) node[right]{$y_{\eta_2}^2(k)_2$};
	\draw[<-] (.04,5.1) to (1.3,5.2) node[right]{$y_{\eta_2}^1(k)_1$};
	\draw[<-] (.04,5.3) to (.8,5.6) node[right]{$y_{\eta_2}^2(k)_1$};
	\draw[<-] (.04,5.5) to (.4,6) node[right]{$Y_{\eta_2}^2$};
	
	\end{tikzpicture}
	\caption{\label{FigEstSize}Notations of the proof of Lemma \ref{LemEstSize}. For simplicity, the sets $P_{k}(\eta)_j$ have been represented as rectangles while they are quasi-rectangles in reality. The crosses are the points of intersection $(x_{\eta}^i(k)_j, y_{\eta}^i(k)_j)$, their $y$ coordinate is indicated at the right.}
\end{figure}
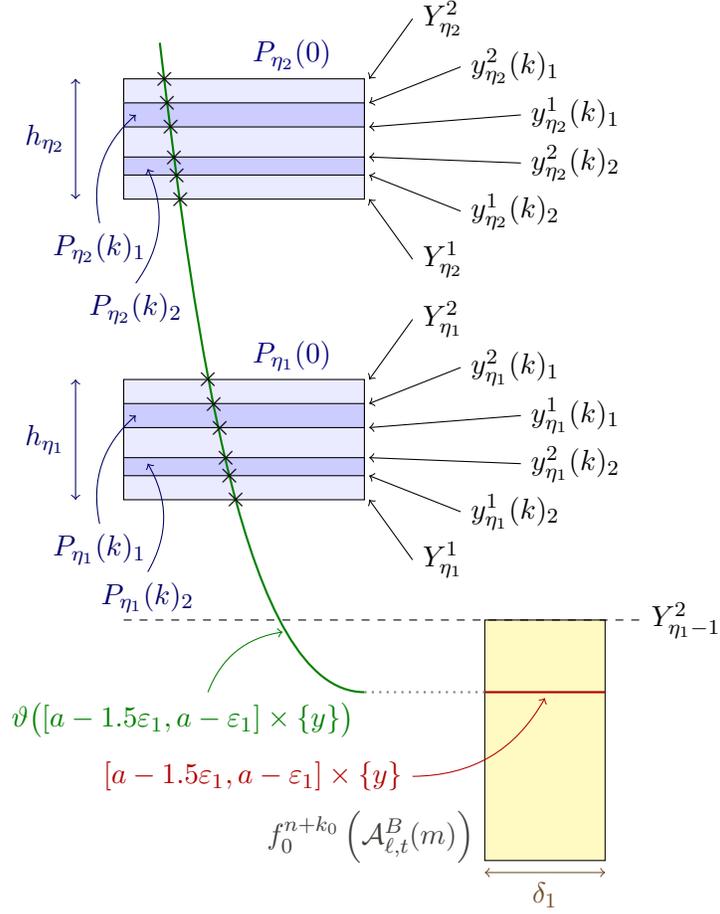

We claim that, for $i\in J_{\theta}(k+1)$ such that the connected component $P_{k+1}(\theta)_i$ corresponds to the pre-image under $f_1^{m+k_0}$ of the intersection $f_1^{m+k_0}\left(\cA^B_{\ell,t}(m)\right)\cap P_k(\eta)_j$, it holds 
\[
d_{\theta}(k+1)_i=\sup\left\{|x_{\eta}^1(k)_j-x^2_{\eta}(k)_j|;\,y\in \sigma^{-2m}[a-2\eps_1,a-3\eps_1]\right\}.
\]
To see this, notice that $P_{k+1}(\theta)_i$ is equal to the union of the images of the segments $[x^1_{\eta}(k)_j,x^2_{\eta}(k)_j]\times\{y\}$ under the map $f_0^{-m-k_0+1}\circ h_1^{-1}\circ f_0^{-1}$. Since $h_1$ is a translation and $f_0^{m+k_0}$ is a rotation (after rescaling), one deduces the formula. Moreover, the same reasoning says that the map $f_0^{-m-k_0+1}\circ h_1^{-1}\circ f_0^{-1}$ preserves proportions on the horizontal direction and thus 
\begin{equation}
\label{e.jevole}
\frac{d_{\theta}(k+1)_i}{h_{\theta}}=\frac{\sup\left\{|x_{\eta}^1(k)_j-x^2_{\eta}(k)_j|;\,y\in \sigma^{-2m}[a-2\eps_1,a-3\eps_1]\right\}}{\delta_1}.
\end{equation}
We proceed now to estimate $|x_{\eta}^1(k)_j-x^2_{\eta}(k)_j|$. For this, we shall consider also the points $(X^1_{\eta}(y),Y^1(\eta))$ and $(X^2_{\eta}(y),Y^2_{\eta}(y))$ the intersections of $\operatorname{graph}(\vartheta)$ with (respectively) the bottom and the top parts of $\cA^B_{\ell,w}(n)$ (see Figure~\ref{FigEstSize}).
Using the fact that $\prec$ is a total order on a finite set, for any $\eta$ but the smallest one, one can define its precursor and denote it by $\eta-1$. If $\eta$ is the smallest one, we denote $\big( X^2_{\eta-1}(y),\, Y^2_{\eta-1}(y) \big)$ the intersection point of $\operatorname{graph}(\vartheta)$ with
$\R\times \sigma^{-2m}(a-2\eps_1)$.
Since $DL_n$ is a diagonal matrix we have that 
\[
\frac{|y^1_{\eta}(k)_j-y^2_{\eta}(k)_j|}{d_{\eta}(k)_j}=\frac{|\tilde{y}^1_{\eta}(k)_j-\tilde{y}^2_{\eta}(k)_j|}{\tilde{d}_{\eta}(k)_j}.
\]
Moreover, as $L\tilde{P}_{\eta}(k)_j$ is a quasi-rectangle with inclination $<1/2$ we may use Lemma~\ref{LemNonLin} to deduce that the right-hand side of the above equality is bounded by $4/3$. Thus, we can write 
\begin{equation}
\label{e.dansemacabre}
|y^1_{\eta}(k)_j-y^2_{\eta}(k)_j|\leq\frac{4}{3}\lambda_{\eta}(k)_jh_{\eta}.
\end{equation} 
Now, the slope inequality for convex maps applied to $\vartheta$ gives us 
\[\left| \frac{y^1_\eta(k)_j - y^2_\eta(k)_j}{x^1_\eta(k)_j -  x^2_\eta(k)_j} \right| \ge 
\left| \frac{Y^1_\eta(y) - Y^2_{\eta-1}(y)}{X^1_\eta(y) - X^2_{\eta-1}(y)} \right|,\]
and so
\[\left| \frac{x^1_\eta(k)_j - x^2_\eta(k)_j}{X^1_\eta(y) - X^2_{\eta-1}(y)} \right| \le 
\left| \frac{y^1_\eta(k)_j - y^2_\eta(k)_j}{Y^1_\eta(y) - Y^2_{\eta-1}(y)} \right|.\]
On the other hand, as $h_{\eta}=\sigma^{-n}\delta_1$ and since $8\delta_1<\eps_1$ we have
$$|Y^1_\eta(y) - Y^2_{\eta-1}(y)| \ge 8h_\eta.$$
These two inequalities combined with \eqref{e.dansemacabre} lead to 
\begin{equation}\label{EqQuotientB}
\left| \frac{x^1_\eta(k)_j - x^2_\eta(k)_j}{X^1_\eta(y) - X^2_{\eta-1}(y)} \right| \le 
\frac{4}{3} \frac{\lambda_\eta(k)_jh_\eta}{8h_\eta}\le 
\frac{\lambda_\eta(k)_j}{4}.
\end{equation}

Our next goal is to prove the following bound: 
\[
|X^1_{\eta}(y)-X^2_{\eta-1}(y)|\leq |X^1_{\eta}(\sigma^{-2m}(a-2\eps_1))-X^2_{\eta-1}(\sigma^{-2m}(a-2\eps_1))|.
\]
Noticing that $X^{i}_{\eta}(y)=\vartheta^{-1}(Y^i_{\eta}(y))$, it will be obtained as a consequence of the following.

\begin{claim}
	\label{claim.alorsalors}
For every $r,t$ such that $\sigma^{-n_0-1}(b+2\eps_1)>r>t\geq\sigma^{-2m}(a-2\eps_1)$ the function $y\mapsto\vartheta^{-1}(t)-\vartheta^{-1}(r)$ is increasing.
\end{claim}

\begin{proof}[Proof of Claim~\ref{claim.alorsalors}]
First, recall that $\xi_2$ is affine on the interval $(0,q_2)$, and its derivative equal to $\beta\in(0,1)$, so that $\xi_2-\Id$ is decreasing on the interval $(0,q_2)$ (see Figure~\ref{f.grafxi2}) and thus $\xi(y)\eqdef \xi_2(\sigma^{-1}y)-\sigma^{-1}y$ has derivative equal to $\sigma^{-1}(\beta-1)<0$.  As $\vartheta(x)=y+\sigma\varphi_2(\sigma^2x)\xi(y)$ we see that 
\[
\vartheta^{-1}(t)=\sigma^{-2}\varphi_2^{-1}\left(\frac{t-y}{\sigma\xi(y)}\right),
\]
where we consider the inverse branch of $\varphi_2$ above the interval $[\sigma^2(a-2\eps_1-\delta_1),\sigma^2(a-2\eps_1)]$. In order to estimate the derivative of this map let us  consider, for every $\sigma^{-n_0-1}(b+2\eps_1)>t\geq\sigma^{-2m}(a-2\eps_1)$,
\[
X_t(y)=\frac{t-y}{\sigma\xi(y)}.
\] 
Since $\sigma\xi_2(\sigma^{-1}y)>\sigma^{-n_0-1}(b+4\eps_1)$, we have that
\[
X_t(y)<\frac{\sigma^{-n_0-1}(b+2\eps_1)-y}{\sigma^{-n_0-1}(b+4\eps_1)-y}<1,
\]
which shows that $\varphi_2^{-1}(X_t(y))$ is meaningful. Moreover, as the function $y\mapsto\frac{\sigma^{-n_0-1}(b+2\eps_1)-y}{\sigma^{-n_0-1}(b+4\eps_1)-y}$ is decreasing we deduce further that
\[
X_t(y)\leq\frac{\sigma^{-n_0-1}(b+2\eps_1)}{\sigma^{-n_0-1}(b+4\eps_1)}<\frac{b+3\eps_1}{b+4\eps_1}.
\]
Therefore, if $x_t=\varphi_2^{-1}(X_t(y))$ then $\varphi_2^{\prime\prime}(x_t)>0$, by our convexity assumption on $\varphi_2$ (see Figure~\ref{FigPhi2}).
We claim that $X_t(y)$ has negative derivative (with respect to $y$). To see this, first notice that 
\[
\sigma\xi_2\left(\sigma^{-1}y\right)>\sigma^{-n_0-1}(b+4\eps_1)>\sigma^{-n_0-1}(b+2\eps_1)>t,
\]
and thus $\sigma\xi_2\left(\sigma^{-1}y\right)-y>t-y$. Combining this with the $\sigma\xi^{\prime}(y)=\beta-1$ one deduces that
\[
X_t^{\prime}(y)=\frac{-\sigma\xi(y)-(t-y)\sigma\xi^{\prime}(y)}{\sigma^2\xi(y)^2}\leq\frac{-1}{\sigma^2\xi(y)^2}\times\left(\sigma\xi_2(\sigma^{-1}y)-y-(t-y)\right)<0.
\]
As a by-product we also obtain that
\[
X_t^{\prime}(y)-X_r^{\prime}(y)=\frac{-\sigma\xi^{\prime}(y)(t-r)}{\sigma^2\xi(y)^2}<0.
\]
Moreover, as $r>t$ the definition of $X_t(y)$ yields $X_r(y)>X_t(y)$. As $\varphi_2$ is decreasing on the interval $[\sigma^2(a-2\eps_1-\delta_1),\sigma^2(a-2\eps_1)]$, this implies $x_r<x_t$. By convexity, we deduce then $\varphi_2^{\prime}(x_r)<\varphi_2^{\prime}(x_t)<0$ (recall the graph of $\varphi_2$ in Figure~\ref{FigPhi2}), and therefore
\[
\frac{1}{\varphi_2^{\prime}(x_t)}<\frac{1}{\varphi_2^{\prime}(x_r)}.
\]
 Because $X_t^{\prime}(y)<0$, this yields
\[
\frac{1}{\varphi_2^{\prime}(x_t)}X_t^{\prime}(y)-\frac{1}{\varphi_2^{\prime}(x_r)}X_r^{\prime}(y)>\frac{1}{\varphi_2^{\prime}(x_r)}\left(X_t^{\prime}(y)-X_r^{\prime}(y)\right)>0.
\] 
Since the derivative of $y\mapsto\vartheta^{-1}(t)-\vartheta^{-1}(r)$ equals a positive constant (namely $\sigma^{-2}$) times the left-hand side above, the claim follows. 
\end{proof}

Notice that $X^{i}_{\eta}(y)=\vartheta^{-1}(Y^i_{\eta}(y))$. Thus, Claim~\ref{claim.alorsalors} yields
\[
|X^1_{\eta}(y)-X^2_{\eta-1}(y)|\leq |X^1_{\eta}(\sigma^{-2m}(a-2\eps_1))-X^2_{\eta-1}(\sigma^{-2m}(a-2\eps_1))|.
\]
Thus, it follows from \eqref{EqQuotientB} that 
\[
|x^1_{\eta}(k)_j-x^2_{\eta-1}(k)_j|\leq \big|X^1_{\eta}(\sigma^{-2m}(a-2\eps_1))-X^2_{\eta-1}(\sigma^{-2m}(a-2\eps_1))\big|\frac{\lambda_{\eta}(k)_j}{2}.
\]
Finally, from the equality \eqref{e.jevole} one obtains the bound
\[d_{\theta}(k+1)_j \le \big|X^1_\eta(\sigma^{-2m}(a-2\eps_1)) - X^2_{\eta-1}(\sigma^{-2m}(a-2\eps_1))\big|\frac{\lambda_\eta(k)_i}{2}\times \frac{h_{\theta}}{\delta_1}, \]
and so
\[\sum_{j\in J_\theta(k+1)}\lambda_{\theta}(k+1)_j \le \sum_{\eta\prec\theta} \sum_{i\in J_\eta(k)}\left|X^1_\eta(\sigma^{-2m}(a-2\eps_1)) - X^2_{\eta-1}(\sigma^{-2m}(a-2\eps_1))\right|\frac{\lambda_\eta(k)_i}{2 \delta_1}.\]

But the sum of the lengths of disjoint subintervals of an interval of length $\delta_1$ is smaller than $\delta_1$, so
\[\sum_{\eta\prec\theta}\big|X^1_\eta(\sigma^{-2m}(a-2\eps_1)) - X^2_{\eta-1}(\sigma^{-2m}(a-2\eps_1))\big| \le \delta_1.\]
Hence,
\[\sum_{j\in J_\theta(k+1)} \lambda_\theta(k+1)_j \le \frac{1}{2}\sup_{\eta\prec\theta} \left( \sum_{j\in J_\eta(k)} \lambda_\eta(k)_j\right) \le 2^{-k-1}.\]
This establishes the induction and completes the proof.
\end{proof}

\subsubsection{Coding} Let us show how to conclude Proposition~\ref{PropCoding} from Lemma~\ref{LemEstSize}.
\begin{proof}[Proof of Proposition \ref{PropCoding}]
The semi-conjugation stems from Lemma \ref{lemQuasiRec}, and in particular the fact that the intersections between quasi-rectangles and images of rectangles are Markovian. Then, Lemma \ref{LemEstSize} ensures that the connected components of $\Gamma$ (\emph{i.e.}, the nested intersections of quasi-rectangles) are not ``thick''. 

More precisely, we will use the fact that if we have a nested intersection of regions that are between two graphs of Lipschitz maps of Lipschitz constant $2$, such that the maximal height of the regions goes to zero, then the intersection is again the graph of a Lipschitz map with Lipschitz constant $2$. This comes from the following characterization of a Lipschitz map $\gamma$ with Lipschitz constant $2$: for any $x_0$, the graph of $\gamma$ is included in the cone
\[\big\{(x,y)\in\R^2 \,;\, |y-\gamma(x_0)| \le 2 |x-x_0|\big\}.\]

Now, combining Lemmas \ref{lemQuasiRec} and \ref{LemEstSize} we see that for any $\theta = (h,v,n)$, the set $\bigcap_{k\in\nt} L_n(P_k(\theta))$ is obtained as a decreasing intersection of quasi-rectangles whose Lebes\-gue measure decreases exponentially to 0 (this is obtained by applying Fubini's theorem). Hence, this set has zero measure, and is homeomorphic to the product of a Cantor set with a segment, and each of its connected components is the graph of a Lipschitz map with Lipschitz constant smaller than $2$. Using $L_n$, we see that the same holds for $\bigcap_{k\in\nt} P_k(\theta)$, with the additional property that the Lipschitz constant of each component goes to 0 as $n$ goes to infinity (denoting $\theta = (h,v,n)$). The only accumulation points of $\Gamma = \bigcup_\theta \bigcap_{k\in\nt} P_k(\theta)$ are $[a-3\eps_1,a-2\eps_1] \times\{0\}$ and $[b+2\eps_1,b+3\eps_1] \times\{0\}$. The property about the Lipschitz constant going to 0 then implies that $\overline\Gamma$ itself is homeomorphic to the product of a Cantor set with a segment.
\end{proof}

\subsection{The orbit of $\Gamma$ is nowhere dense}\label{secnowhere} To finish the proof of Proposition \ref{l.persmedidanula} (which is our main goal in Section \ref{SecGeom}), it only remains to prove the following fact.

\begin{prop}\label{p.gamanowheredense}
The set\, $\mathcal{O}_{f_1}(\Gamma)=\bigcup_{n\in\Z}f_1^{n}(\Gamma)$ is nowhere dense.
\end{prop}

The proof will be based on the following general criterion.

\begin{lema}\label{lemaPA} Let $X$ be a compact metric space and $f:X \to X$ a homeomorphism. Let\, $\Gamma^+\subset X$ be such that $f(\Gamma^+)\subset\Gamma^+$ satisfying the following property: any open set intersecting $\Gamma^+$ also intersects the interior of the complement of\, $\bigcup_{n\ge 0}f^{-n}(\Gamma^+)$. Then\, $\bigcup_{n\ge 0} f^{-n}(\Gamma^+)$ is nowhere dense.
\end{lema}

\begin{proof}
Consider a non-empty open set $V$, and let us prove that $\bigcup_{n\ge 0} f^{-n}(\Gamma^+)$ is not dense in $V$. If there is no $n\in\nt$ such that $f^{-n}(\Gamma^+)\cap V\neq\emptyset$, then there is nothing to prove. Otherwise, $\Gamma^+\cap f^{n_0}(V) \neq\emptyset$ for some $n_0\ge 0$. Hence, by hypothesis, $f^{n_0}(V)$ intersects the interior of the complement of $\bigcup_{n\ge 0}f^{-n}(\Gamma^+)$: there exists a nonempty open set 
\[U \subset f^{n_0}(V) \setminus \bigcup_{n\ge 0}f^{-n}(\Gamma^+) = f^{n_0}(V) \setminus \bigcup_{n\ge -n_0}f^{-n}(\Gamma^+)\]
(the equality comes from the forward invariance of $\Gamma^+$). Hence, $f^{-n_0}(U)$ is a nonempty open set included in $V$ and disjoint from $\bigcup_{n\ge 0}f^{-n}(\Gamma^+)$.
\end{proof}

As an application of Lemma~\ref{lemaPA}, we prove the following.

\begin{lema}\label{maislemaPA} The set $W^s_{f_1}(O) \cup W^u_{f_1}(O)$ is nowhere dense.
\end{lema}

\begin{proof} As the proof is identical for both sets, we only prove that $W^s_{f_1}(O)$ is nowhere dense, using the criterion given by Lemma \ref{lemaPA}. For the set $\Gamma^+$, we choose $[-1,1]\times \{0\}$, which is a forward invariant subset of $W^s_{f_1}(O)$ satisfying $W^s_{f_1}(O) = \bigcup_{n\ge 0}f_1^{-n}(\Gamma^+)$. The fact that any open set intersecting $\Gamma^+$ also intersects the interior of the complement of $W^s_{f_1}(O)$ comes from the fact that it has to intersect the set $\big\{(x,y)\in\R^2;\,\,xy>0\big\}$ and that $W^s_{f_1}(O)$ is disjoint from this set (because of the form of the perturbations $h_1$ and $h_2$, the $f_1$-orbit of any point that lies inside one bounded component of the complement of the figure-$8$ loop of $W^s_{f_0}(O)$, stays inside this bounded component). 
\end{proof}

Now, consider the ($f_1$ forward invariant) set $\Gamma^+=\bigcup_{n \geq 0}f_{1}^{n}(\Gamma)$. The dynamics of $\Gamma^+$ under $f_1$, \emph{as a set}, combines linear hyperbolic iterates close to the origin, followed by rotational iterates outside a neighbourhood of the origin (of course, much more interesting dynamics occur \emph{inside} $\Gamma^+$, due to the action of $h_1$ and $h_2$, as it was carefully described in Proposition \ref{PropCoding} of this section). This observation enables us to obtain the following description of $\Gamma^+$:

\begin{equation}
\label{e.gamamais}
\Gamma^+=\bigcup_{\substack{h=\ell,r\\v=t,b}}\bigcup_{n\geq n_1(h)}\bigcup_{j=0}^{n+k_0-1}f_1^j\left(\Gamma\cap\cA^B_{h,v}(n)\right). 
\end{equation}
Indeed, recall that $\Gamma$ is invariant under the first return map $g_1$, and that $g_1|_{\cA^B_{h,v}(n)\cap\Gamma}=f_1^{n+k_0}$. These observations are enough to establish \eqref{e.gamamais}. From \eqref{e.gamamais} we deduce that $\overline{\Gamma^+}\setminus\Gamma^+$ is contained in $W_{f_0}^{s}(0) \cup W_{f_0}^{u}(0)$, and since both $W_{f_0}^{s}(0)$ and $W_{f_0}^{u}(0)$ are nowhere dense sets we can easily conclude that $\Gamma^+$ itself is nowhere dense. This fact will be combined with Lemma \ref{lemaPA} in order to conclude the proof of Proposition~\ref{p.gamanowheredense}. With this purpose, consider
\[
\mathcal{O}\eqdef\bigcup_{\substack{h=\ell,r\\v=t,b}}\,\bigcup_{n\geq n_1(h)}\bigcup_{j=0}^{n+k_0-1}\operatorname{int}\left[f_1^j\left(\cA_{h,v}^B(n)\right)\right].
\]

\begin{lema}
	\label{l.stromae}
We have
\[\left(\bigcup_{n\ge 0} f_1^{-n}(\Gamma^+) \right)\cap\overline{\mathcal{O}}\subset\Gamma^+.\]
\end{lema}
\begin{proof}
Take a point $z\in f_1^{-j}(\Gamma^+)\cap\overline{\mathcal{O}}$, for some $j\geq 0$. As $z\in\overline{\mathcal{O}}$, there must exist some $n\geq n_0$, some $m\geq 0$ and some choice of $h=\ell,r$ and $v=t,b$ such that $z^{\prime}=f_1^{-m}(z)\in\cA^B_{h,v}(n)$.
We claim that $z^{\prime}\in\Gamma$. As $z$ is a positive iterate of $z^{\prime}$, this implies the lemma. Let us now prove the claim. Assume by contradiction that $z^{\prime}\notin\Gamma$. Then, according to \eqref{azullemadascores} of Lemma~\ref{l.allerverslesbleu}, we must have $z^{\prime}\in W^s(\mathcal{Q})$. Since $\mathcal{Q}$ is a periodic trapping region for $f_1$ with period $n_0+k_0+1$ (see Lemma~\ref{l.criandopoco}), for every sufficiently large $k$ one has that 
\[  
f_1^k(z^{\prime}) \in \bigcup_{j=0}^{n_0+k_0}f_1^j(\mathcal{Q}). 
\]On the other hand, we have $f_1^{j+m}(z^{\prime})\in\Gamma^+$, and as $\Gamma^+$ is a forward invariant set, we conclude that every large iterate of $z^{\prime}$ belongs to $\Gamma^+\cap\left[  \bigcup_{j=0}^{n_0+k_0}f_1^j(\mathcal{Q})\right]$. However, expression \eqref{e.gamamais} for $\Gamma^+$ shows that this intersection is empty This contradiction proves the claim, and concludes the proof.
\end{proof}

We can now conclude the proof of Proposition~\ref{p.gamanowheredense}.

\begin{proof}[Proof of Proposition~\ref{p.gamanowheredense}] Let us verify that Lemma~\ref{lemaPA} can be applied to the set $\Gamma^+$. First, the map $f_1$ is compactly supported. Second, $\Gamma^+$ is forward invariant by definition. Finally, any open set $V$ intersecting $\Gamma^+$ also intersects $\mathcal O$ (as \eqref{e.gamamais} shows that $\Gamma^+\subset\mathcal{O}$). Combining Lemma~\ref{l.stromae} with the fact that $\Gamma^+$ is nowhere dense (as explained above, before the statement of Lemma \ref{l.stromae}), we deduce that $V$ intersects the interior of the complement of\, $\bigcup_{n\ge 0}f_1^{-n}(\Gamma^+)$. This allows to apply Lemma~\ref{lemaPA} and concludes the proof of the proposition.
\end{proof}

The same argument as above, with the obvious adaptation shows:

\begin{prop}\label{p.gamanowheredense2}
$O_{f_1}(\Gamma^e)$ is nowhere dense.
\end{prop}

\section{Orbit exclusion II: proof of Theorem~\ref{main.exemplonovo}}\label{sec.orbitexcii}

Let $f_1\in\difp$ be the diffeomorphism constructed in Section~\ref{SecDesc}. The reader should keep in mind both Proposition \ref{p.bluetrichotomy} and Proposition \ref{l.persmedidanula} from the previous sections, whose statements can be summarized as follows: there exists a set $\Gamma\cup\Gamma^e\subset\bigcup_{n \geq n_0}\big(C_n \setminus S_n\cup C_n^e\setminus S^e_n\big)$, whose orbit $\bigcup_{n\in\Z}f_1^{n}(\Gamma\cup\Gamma^e)$ has zero Lebesgue measure and is nowhere dense, such that
$$\cB_{f_1}(\delta_O)\,\subset\, W_{f_1}^s(O)\cup\mathcal{S}\cup (-\mathcal{S})\cup\mathcal{S}^e\cup \mathcal{O}_{f_1}\big(\Gamma\cup(-\Gamma)\cup\Gamma^e\big)$$
and
$$\cB_{f_1}(\delta_O)\,\supset\, W_{f_1}^s(O)\cup\mathcal{S}\cup(-\mathcal{S})\cup\mathcal{S}^e\cup\mathcal{O}_{f_1}(\Gamma^e)\,.$$
From Lemmas~\ref{maislemaPA}, \ref{p.gamanowheredense} and \ref{p.gamanowheredense2} we know that $W_{f_1}^s(O)$, $\mathcal{O}_{f_1}\big(\Gamma\cup(-\Gamma)\big)$ and $\mathcal{O}_{f_1}(\Gamma^e)$ are nowhere dense and zero measure sets, from which we deduce that the statistical basin $\cB_{f_1}(\delta_O)$ coincides with the union of stable towers $\mathcal{S} \cup (-\mathcal{S}) \cup \mathcal{S}^e$ up to a nowhere dense zero Lebesgue measure set. We shall now prove Theorem~\ref{main.exemplonovo} by performing a perturbation of $f_1$ aimed to toss points of the boxes $S_n$ out of the basin of the origin, so that what will remain will be a nowhere dense set with positive Lebesgue measure. More precisely, our main result in this section is the following proposition, which immediately implies Theorem \ref{main.exemplonovo}.

\begin{prop}\label{propfinalf2}
There exists $h_3\in \Diff^1(\R^2)$, such that $f_2 = h_3\circ f_1$\nomenclature{$f_2$}{Second and final perturbation of $f_0$, such that $f_2=h_2\circ f_1$} is a compactly supported diffeomorphism that has a hyperbolic fixed point at the origin $O$ which is of saddle type and, moreover,  
$\cB_{f_2}(\delta_O)$ is nowhere dense and has positive Lebesgue measure.
\end{prop}

Unfortunately, our construction only allows us to work in $C^1$ regularity (see Lemma~\ref{l.cum} below).

\subsection{Description of the diffeomorphism $h_3$}
Let us give an informal idea of the proof of Proposition~\ref{propfinalf2}. We shall construct the diffeomorphism $h_3$ as a composition of an infinite number of $C^\infty$-diffeomorphisms with disjoint supports, in the $\eps_1$-stable boxes $C_n$. For this purpose, we fix a Cantor set on the segment $I$ and compose with a perturbation formed of small vertical pushes outside of this Cantor set, once at each 4 returns in a box $S_n$. 

We shall give the precise definition of $h_3$ only in the first quadrant. A similar construction is easily performed in the third, and by symmetry one can complete the definitions. For boxes of the exterior $\mathcal{L}^e$ of the figure eight attractor of $f_0$, one also can do a similar construction, by taking into account the fact that the return map of $f_1$ for $\bigcup_{n\ge n_0} s_v(S_n)$ is a rotation of angle $\pi$ after renormalization (and not $\pi/2$ as for the boxes $S_n$). Hence, the perturbations have to be performed every 2 returns, and not 4 as for boxes $S_n$.

\subsubsection{Orbit of a box $S_m$}
Recall from Proposition~\ref{buildTowers} that for each $m\geq n_0$ there exists a first return map $g_0=f_0^{m+k_0}:S_m\to S_{2m}$ (see Definition \ref{defgzero}). Therefore, each positive integer $m \geq n_0$ is either a starting point for an orbit of a box or a positive iterate of some box $S_k$, $n_0 \le k<m$, under $g_0$. From now on we fix a positive integer $m$ which is the starting point of an orbit. Notice that this amounts to say that either $n_0\leq m<2n_0$ or $m$ is an odd positive integer. Such an $m$ will be called a \emph{starting integer}.

\subsubsection{Choice of bump functions}
We fix a $C^{\infty}$ function $\psi:[1,\sigma]\to[0,1]$ such that $\psi^{-1}(0)=[1,\sigma]\setminus(a,b+\eps_1)$ and $\psi^{-1}(1)=[a+\eps_1,b]$. We further assume that $\psi$ is increasing over $[a,a+\eps_1]$.

Let $K\subset I=[a,b]$ be a Cantor set (\emph{i.e.}, a perfect subset of $I$ with empty interior) with positive one-dimensional Lebesgue measure. Just for simplicity, let us assume that the extremal points of $K$ are precisely $a$ and $b$ (in other words: $I$ is the smallest interval containing $K$). Applying Whitney's extension theorem \cite{whitney1934analytic}, there exists a $C^{\infty}$ function $\varphi:\R\to [0,1]$ such that $\varphi^{-1}(0)=K\cup(\R \setminus I)$, see Figure \ref{figphi}.\nomenclature{$\varphi$}{Map $\varphi:\R\to [0,1]$ such that $\varphi^{-1}(0)=K$} Finally, let $\delta:\nt\to\{0,1\}$ be the \emph{characteristic function} of the set $\Omega=\Omega(n_0)\subset\nt$ given by\footnote{For boxes $s_v(S_n)$ of the exterior $\mathcal{L}^e$ of the attractor, one has to change the powers of $16$ by powers of $4$, as the return map of $f_1$ to the union of such boxes is, after renormalization, a rotation of angle $\pi$.}
$$\Omega=\bigcup_{\substack{\text{$m$ starting}\\ \text{integer}}}\ \bigcup_{d\in\nt}\{16^dm\}=\left(\bigcup_{m=n_0}^{2n_0-1}\bigcup_{d\in\nt}\{16^dm\}\right)\bigcup\left(\bigcup_{m=n_0}^{+\infty}\bigcup_{d\in\nt}\{16^d(2m+1)\}\right).$$The set $\Omega$ represents the set of box indices in which we will perform a perturbation. It contains one index over four in each orbit under $g_0$ of a starting box.

\begin{figure}
\begin{tikzpicture}
\draw[->] (-.2,0) -- (6.8,0);
\draw[->] (0,-.2) node[below]{$a$} -- (0,1);
\draw (6.4,0) node{$|$} node[below]{$b$};

\draw[color=blue!80!black] (0,0)to[out=0,in=180] (.1,.1)to[out=0,in=180] (.2,0)to[out=0,in=180] (.4,.2)to[out=0,in=180] (.6,0)to[out=0,in=180] (.7,.1)to[out=0,in=180] (.8,0)to[out=0,in=180] (1.2,.4)to[out=0,in=180] (1.6,0)
to[out=0,in=180] (1.7,.1)to[out=0,in=180] (1.8,0)to[out=0,in=180] (2,.2)to[out=0,in=180] (2.2,0)to[out=0,in=180] (2.3,.1)to[out=0,in=180] (2.4,0)to[out=0,in=180] (3.2,.8);

\begin{scope}[xscale=-1,xshift=-6.4cm]
\draw[color=blue!80!black] (0,0)to[out=0,in=180] (.1,.1)to[out=0,in=180] (.2,0)to[out=0,in=180] (.4,.2)to[out=0,in=180] (.6,0)to[out=0,in=180] (.7,.1)to[out=0,in=180] (.8,0)to[out=0,in=180] (1.2,.4)to[out=0,in=180] (1.6,0)
to[out=0,in=180] (1.7,.1)to[out=0,in=180] (1.8,0)to[out=0,in=180] (2,.2)to[out=0,in=180] (2.2,0)to[out=0,in=180] (2.3,.1)to[out=0,in=180] (2.4,0)to[out=0,in=180] (3.2,.8);
\end{scope}
\end{tikzpicture}
\caption{\label{figphi} The map $\varphi$}
\end{figure}
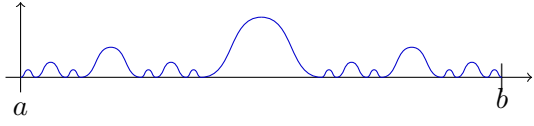

\subsubsection{Choice of the vertical push}
Let $c>0$ be such that $\|\varphi\|_{C^1}<c$ and $\|\psi\|_{C^1}<c$. Pick a small positive number $\kappa>0$ such that $\frac{c^2\kappa}{\log 2}<\eps_1$, and consider the sequence
\begin{equation}\label{EqDefEpsN}
\epsilon_n\eqdef\frac{\kappa\sigma^{-n}}{\log n}.
\end{equation}
Then, we have that $c\sigma^n \epsilon_n<1$ for every $n\geq 2$ and $\sigma^n\epsilon_n\to 0$ as $n\to\infty$.
The number $\epsilon_n$ represents the maximum amount with which our perturbation will push vertically each point in a box $S_n$. Notice that, at the scale of the vertical size of the box $S_n$ (which is $\sim\sigma^{-n}$) the sequence $\epsilon_n$ \emph{is not summable}. This is a crucial feature for our perturbation can really toss out point from the stable tower $\mathcal S$.

\subsubsection{Definition of the perturbation}

Recall that $\tilde{S}_n\eqdef \tilde{I}\times\sigma^{-n}\tilde{I}$. For each $(x,y)\in\R^2$, we set\nomenclature{$h_2$}{Diffeomorphism used to perturb $f_1$ and to throw some points of the boxes away from $\cB{\delta_O}$}
\[h_3(x,y)=\left\{\begin{array}{ll}
(x\,,\ y) & \text{if } (x,y)\notin \bigcup_{n\ge n_0} \tilde{S}_n,\\
(x\,,\ y+\delta_n\,\varphi(x)\,\psi(\sigma^ny)\,\epsilon_n) & \text{if } (x,y)\in\tilde{S}_n,\:\text{for some}\:n\geq n_0.
\end{array}\right.\]
Note that$$\supp(h_3)\subsetneq\bigcup_{n\in\Omega(n_0)}S_n\cup\big(I\times\sigma^{-n}(b,b+\varepsilon_1)\big)\subsetneq\bigcup_{n \geq n_0}C_n\subsetneq\bigcup_{n \geq n_0}\widetilde{S}_n\,.$$Moreover, from the definitions of $\varphi$ and $\psi$, we see that $h_3$ is $C^{\infty}$ on $\R^2\setminus(\tilde{I}\times\{0\})$ and is continuous on $\R^2$.

\begin{lema}\label{l.cum}
$h_3\in\dif^1(\R^2) \setminus \dif^2(\R^2)$.
\end{lema}

\begin{proof}[Proof of Lemma \ref{l.cum}]
Let $p = (x,y)\in\tilde{S}_n$ for some $n \geq n_0$. The Jacobian matrix of $h_3$ at $p$ is given by 
$$Dh_3(p)=\begin{bmatrix}
	1 & 0\\
	\delta_n\varphi^{\prime}(x)\psi(\sigma^ny)\epsilon_n & 1+\delta_n\varphi(x)\sigma^n\psi^{\prime}(\sigma^ny)\epsilon_n
	\end{bmatrix}.$$
From the definition of $\epsilon_n$ and the choice of $\kappa$, it follows that the map $\Phi:\R^2\to\R^2$ defined by $h_3=\operatorname{Id}+\Phi$ is a contraction (for $\|\cdot\|_{C^1}$). This proves that $h_3$ is a homeomorphism. Moreover, by definition of $\epsilon_n$, $\sigma^n\epsilon_n\to 0$ as $n\to+\infty$ (see \eqref{EqDefEpsN}), and then $Dh_3(p)\to \operatorname{Id}$ as $p\to\tilde{I}\times\{0\}$ (that is, as $n$ goes to infinity). Since $h_3$ is a homeomorphism, with the inverse function theorem we conclude that $h_3\in\dif^1(\R^2)$. Finally, we prove that $h_3$ is not $C^2$. Let $\beta(x,y)\eqdef\delta_n\varphi(x)\sigma^n\psi^{\prime}(\sigma^ny)\epsilon_n$. Then,
	$$\partial_y\beta(x,y) = \delta_n\varphi(x)\sigma^{2n}\psi^{\prime\prime}(\sigma^ny)\epsilon_n,$$
	and $\sigma^{2n}\epsilon_n=\frac{(b-a)\sigma^n}{\log n}\to\infty$ as $n\to\infty$.   
\end{proof}

\begin{remark}
	\label{rem.smoothonthewater}
This lemma shows that our example intrinsically $C^1$ but not $C^2$. Let us point out that there is a natural strategy emerging from our construction to get a possibly similar example of regularity $C^\infty$:  instead of performing perturbations as in Section \ref{sec.orbitexcii}, one could add a perturbation similar to $h_2\circ h_1$, but with support of type $[a',b']\times J$, where $a<a'<b'<b$ and $J$ an interval with 0 in its interior. Then, one could iterate the process by adding a perturbation with support $[a'',b'']\times J'$, with $a<a''<b''<a'$, etc. More precisely,  considering a fat Cantor set $K$, one could compose $f_0$ with  an infinite number of diffeomorphisms similar to $h_2\circ h_1$ \eqref{Def.h}, so that for each of them, the counterpart of $\varphi_2$ is supported in one of the holes of $K$. Choosing carefully the counterparts of the $\xi_2$ (that is, so that their norm decrease sufficiently fast), one can ensure that the resulting perturbation is $C^\infty$. Unfortunately, this construction complicates a lot the studies of both the countable number of attractive regions appearing (the counterpart of the region $\mathcal Q$) and the Markov structure arising (that is, the counterpart of $\Gamma$, see Definition \ref{defpersistent}).		
\end{remark}


\subsection{Proof of Proposition~\ref{propfinalf2}} 
We define $f_2\in\dif^1(\R^2)$ as
$$f_2 = h_3\circ f_1=(h_3 \circ h_2 \circ h_1) \circ f_0\,.$$
Recall that, by definition of $h_3$, we have
$$\supp(h_3)\subsetneq\bigcup_{n\in\Omega(n_0)}S_n\cup\big(I\times\sigma^{-n}(b,b+\varepsilon_1)\big)\subsetneq\bigcup_{n \geq n_0}C_n\subsetneq\bigcup_{n \geq n_0}\widetilde{S}_n\,.$$
In particular, $f_2$ coincides with $f_1$ outside the tower of boxes
$$f_{1}^{-1}\big(\bigcup_{n \geq n_0}\widetilde{S}_n\big)=f_{0}^{-1}\big(\bigcup_{n \geq n_0}\widetilde{S}_n\big)=\bigcup_{n \geq n_0}\sigma^{-2}\tilde{I}\times\sigma^{-(n+1)}\tilde{I}\,.$$

\subsubsection{Dynamics of $f_2$}\label{UltimaObsefva} Note that the set $\mathcal{Q}$ (recall Lemma \ref{l.criandopoco}) is also a trapping region for $f_2$: the iterations $f^\ell_1\left(\mathcal Q\right)$ are always disjoint from $\supp(h_3)$. We denote 
\[
W^s_{f_2}(\mathcal Q)\eqdef\{z\in\R^2;\,\exists n\geq 0: f_2^n(z)\in\mathcal Q\}.
\]

The key difference between $f_2$ and $f_1$, of course, occurs inside $S_n$. There are points $z\in S_n$ such that $h_3(z)\notin S_n$. Nevertheless, our choice of $\kappa$ implies that if $z \in f_{1}^{-1}(S_n)$ then $f_2(z)\in C_n$ and moreover if $f_2(z)\notin S_n$ then it belongs to the \emph{orange region} (defined in \eqref{subsubsec.lemadascores}). Therefore, just as in Lemma~\ref{l.allerverslesbleu}, every $z \in f_{1}^{-1}(S_n)$ such that $f_2(z)\notin S_n$ satisfies $z\in W^s_{f_2}(\mathcal Q) \cup \mathcal{O}_{f_2}(\Gamma)$.
Moreover, Item \eqref{azullemadascores} of Lemma \ref{l.allerverslesbleu} also holds for $f_2$: every element of the blue tower which is not in $W^s_{f_2}(\mathcal{Q})$, belongs to $\Gamma$ (this last fact will be used in the proof of Lemma \ref{l.nuncadenso} below, in the same way that it was used during the proof of Lemma \ref{l.stromae}). We summarize this discussion in the following lemma.

\begin{lema}
\label{l.dinamicadefdois}
The set $\mathcal Q$ is a periodic trapping region for $f_2$, whose stable set $W^s_{f_2}(\mathcal Q)$ satisfies
\begin{enumerate}[(a)]
	\item $\Ba\setminus W^s_{f_2}(\mathcal{Q})\subset\Gamma$
	\item If $z \in f_{1}^{-1}(S_n)$ and $f_2(z)\notin S_n$, then $f_2(z)\in C_n$ and $z\in W^s_{f_2}(\mathcal{Q}) \cup \mathcal{O}_{f_2}(\Gamma)$
\end{enumerate} 
\end{lema}

Proposition~\ref{propfinalf2} will be a straightforward consequence of the following three lemmas.

\begin{lema}\label{LemFinal} For any $n\ge n_0$ and any $(x,y)\in S_n$ we have the following dichotomy.
\begin{itemize}
\item If $x\in K$, then $(x,y)\in \cB_{f_2}(\delta_O)$;
\item If $x\notin K$, then there exists $k\in\nt$ such that $f_2^k(x,y) \in \bigcup_{n \geq n_0}(\tilde{S}_n \setminus S_n)$.
\end{itemize}
\end{lema}

\begin{proof}[Proof of Lemma \ref{LemFinal}] For simplicity, and without loss of generality, fix some starting integer $m \geq n_0$ and some initial point $(x_0,y_0) \in S_m$. If $x_0\in K$, then $\varphi(x_0)=0$ and then $f_2^k(x_0,y_0)=f_1^k(x_0,y_0)=f_0^k(x_0,y_0)$ for every $k\ge 0$ (indeed, recall from the computations after Definition \ref{defgzero} that each point $g_0^{4d}(x_0,y_0)$ has first coordinate equal to $x_0$, and so $h_3$ does not affect it). Thus $(x_0,y_0)\in\cB_{f_2}(\delta_O)$, which proves the first point of Lemma \ref{LemFinal}. So let us move to the second point: assume that $(x_0,y_0) \in S_m$ is such that $x_0\notin K$, and suppose, by contradiction, that for any $k\in\nt$,\[f_2^k(x_0,y_0) \in \left(\bigcup_{n\ge n_0} S_n \right) \cup \left(\bigcup_{n\ge n_0} \tilde S_n \right)^\complement.\]By (b) of Lemma~\ref{l.dinamicadefdois} we deduce that for any $k\in\nt$, $f_2^k(x_0,y_0) \in \mathcal S$ (an orbit under $f_2$ that leaves $\mathcal S$ has to meet $\tilde S_n\setminus S_n$ for some $n$).
Following Section \ref{sec.figoito} (see Definitions \ref{defgzero} and \ref{defg1}) consider the map $g_2:\bigcup_{n \geq n_0}S_n\to\R^2$ given by $g_2=f_{2}^{n+k_0}$ on $S_n$\, for each $n \geq n_0$. Since we are assuming that $f_2^k(x_0,y_0)\in\mathcal S$, we can consider for each $d\in\nt$ the returning point$$(x_d,y_d)=g_{2}^{4d}(x_0,y_0) \in S_{16^dm}\,.$$Moreover, if just as in \eqref{EqTau0d} we consider $\tau_{2}^{d}=\sigma^{16^dm}\,y_d$\,, we have that $\tau_{2}^{d}\in[a,b]$ for all $d\in\nt$. However, a straightforward calculation shows that
$$x_d=x_0 \quad\mbox{and}\quad \tau_{2}^{d+1}=\tau_{2}^{d}+\beta_d\quad\mbox{for all $d\in\nt$,}$$
where each $\beta_d \geq 0$ is given by
$$\beta_d=\frac{\delta_{16^{d+1}m}\,\varphi(x_0)\,\psi(\tau_{2}^{d})\,\kappa\,(b-a)}{\log(16^{d+1}m)}\,.$$
As already explained, since we are assuming that $f_2^k(x_0,y_0)\in\mathcal S$, and hence that $g_2^d(x_0,y_0)\in\bigcup_{n\ge n_0} S_n$, we have that $\tau_{2}^{d}\in[a,b]$ for all $d\in\nt$, which implies that the sequence $\{\beta_d\}_{d\in\nt}$ is summable. This last assertion is absurd for the recursive formula above implies that $\tau_{2}^{d}$ is increasing in $d$. As $\psi$ was chosen to be increasing in $[a,b]$,
$$\beta_d \ge \frac{\,\varphi(x_0)\,\psi(\tau_{2}^{0})\,\kappa\,(b-a)}{\log m + (d+1)\log(16)}$$
(remark that $\varphi(x_0)>0$, because $x_0\notin K$). Thus, $\beta_d$ is not summable, concluding.
\end{proof}

We define the set $\widetilde{\mathcal{K}}$ of points of the boxes that will eventually stay in the stable towers as
\[\widetilde{\mathcal{K}} \eqdef \mathcal{K} \cup (-\mathcal{K}) \cup s_v(\mathcal{K}),\]
with $\mathcal{K}\subset\bigcup_{n \geq n_0}S_n$ defined by
\[\mathcal{K}\eqdef\bigcup_{n=n_0}^{+\infty}\big(K\times\sigma^{-n}I) \cup \big(I\times \{\sigma^{-n} a\}\big).\]
Remark that the set $\widetilde{\mathcal{K}}$  is nowhere dense but has positive Lebesgue measure (by Fubini's Theorem).

\begin{lema}\label{c.basindescription}
One has
\begin{equation}\label{EqDescBasin}
\cB_{f_2}(\delta_O)\ =\ 
W_{f_2}^{s}(O)
\cup \mathcal{O}_{f_2}(\widetilde{\mathcal{K}})
\cup \mathcal{O}_{f_2}(\Gamma^e)
\cup \big(\mathcal{O}_{f_2}(\Gamma \cup -\Gamma) \cap \cB_{f_2}(\delta_O)\big).
\end{equation}
\end{lema}

Remark that the sets $W_{f_2}^{s}(O)$, $\widetilde{\mathcal{K}}$ and $\Gamma^e$ are actually contained in $\cB_{f_2}(\delta_O)$, but the dynamics in $\bigcup_{n\in\Z}f_2^{n}(\Gamma)$ is a bit more intricated and only a (nonempty) part of it belongs to $\cB_{f_2}(\delta_O)$.

To simplify the exposition, we will consider the case of points of the right interior component of the figure eight attractor $\mathcal{L}_r^i$ (and hence only the set $\mathcal{K}$), the other cases being identical.

\begin{proof}[Proof of Lemma \ref{c.basindescription}]
From the first item in Lemma \ref{LemFinal} we know that $\mathcal{K}\subset\cB_{f_2}(\delta_O)$, and then
$$W_{f_2}^{s}(O)\cup \mathcal{O}_{f_2}(\mathcal{K})\subset\cB_{f_2}(\delta_O),$$
since $W_{f_2}^{s}(O)$ is obviously contained in $\cB_{f_2}(\delta_O)$. Conversely, by putting together Proposition \ref{p.bluetrichotomy} and the second item in Lemma \ref{LemFinal}, we deduce that if
$$p\in\cB_{f_2}(\delta_O)\cap\mathcal{L}_r^i\setminus\left[W_{f_2}^{s}(O)\cup\mathcal{O}_{f_2}(\Gamma)\right],$$
then $\displaystyle p\in\mathcal{O}_{f_2}(\mathcal{K})$.
\end{proof}

\begin{lema}\label{l.nuncadenso}
The four sets
$$W_{f_2}^{s}(O),\quad \mathcal{O}_{f_2}(\widetilde{\mathcal{K}}),\quad \mathcal{O}_{f_2}(\Gamma) \quad\mbox{and}\quad \mathcal{O}_{f_2}(\Gamma^e)$$
are nowhere dense in $\R^2$.
\end{lema}

\begin{proof}[Proof of Lemma \ref{l.nuncadenso}] From the definition of $\supp h_3$, we see at once that $W_{f_2}^{s}(O)=W_{f_1}^{s}(O)$, and then Lemma \ref{maislemaPA} implies that $W_{f_2}^{s}(O)$ is nowhere dense. From the definition of $h_1$, $h_2$ and $h_3$ we see that if $p\in\mathcal{K}$, then $f_2^n(p)=f_0^n(p)$ for all $n\in\Z$. Therefore, the fact that $\mathcal{O}_{f_2}(\mathcal{K})$ is nowhere dense follows from Section~\ref{sec.figoito} (more precisely, from Item \eqref{p6} of Proposition \ref{buildTowers}). Finally, the proof of the fact that $\mathcal{O}_{f_2}(\Gamma)$ is nowhere dense follows the same lines as in Section~\ref{secnowhere}. Indeed, note first that if $p\in\Gamma$, then $f_2^n(p)=f_1^n(p)$ for all $n\in\nt$. Therefore, $\mathcal{O}_{f_2}(\Gamma)$ equals the set $\Gamma^{+}$ from Section~\ref{secnowhere}, and in particular is nowhere dense. Moreover, Lemma \ref{l.stromae} still holds for $f_2$ (with the same proof; recall here that every element of the blue tower which is not in $\Gamma$ converges to $\mathcal{Q}$ under $f_2$, as mentioned in Section~\ref{UltimaObsefva}) and then we conclude that $\mathcal{O}_{f_2}(\Gamma)$ is nowhere dense with the help of Lemma \ref{lemaPA}, just as we did in the proof of Proposition~\ref{p.gamanowheredense}. The proof for $\Gamma^e$ is identical.
\end{proof}

\begin{remark}
Since the forward dynamics of $f_2$ in $\Gamma$ is the same as those for $f_1$, we deduce from Proposition \ref{PropCoding} that $f_2|_\Gamma$ has positive topological entropy as well as infinitely many periodic orbits.
\end{remark}

With Lemma \ref{c.basindescription} and Lemma \ref{l.nuncadenso} at hand, we are ready to prove Proposition~\ref{propfinalf2}, our final step in proving Theorem \ref{main.exemplonovo}.

\begin{proof}[Proof of Proposition~\ref{propfinalf2}]
As all the sets appearing in the right part of \eqref{EqDescBasin} --- that is, $W_{f_2}^{s}(O)$, $\mathcal K$, $\Gamma^e$ and $\Gamma$ --- have nowhere dense full orbit under $f_2$ (by Lemma~\ref{l.nuncadenso}), we deduce that $\cB_{f_2}(\delta_O)$ is nowhere dense. Moreover, it has positive Lebesgue measure as it contains $\mathcal K$.
\end{proof}

Adapting the whole construction to the figure-eight attractor (Lemma \ref{buildTowers2}), one gets the following counterpart of Proposition \ref{propfinalf2}.

\begin{prop}\label{FORA2!}
There exists $\hat f_2$ with two hyperbolic fixed points $O$ and $P$ which are of saddle type, such that the set of points with historic behaviour for $\hat f_2$ is nowhere dense and has positive Lebesgue measure.
\end{prop}

Moreover, one can see from the proof that the set of accumulation points of almost all points with historic behaviour is similar to that of Bowen's eye example: it is a segment contained in $[\delta_O, \delta_P]$ (that is: a \emph{convex combination} of the probability measures $\delta_O$ and $\delta_P$), which depends only on the eigenvalues at the hyperbolic fixed points.

\bibliographystyle{plain}
\bibliography{GGS_Refs}

\end{document}